\documentclass[12pt,reqno]{amsart}

\usepackage{amsmath}
\usepackage{amssymb}
\usepackage{amsthm}
\usepackage{amsfonts}
\usepackage{mathrsfs}
\usepackage{graphicx}
\usepackage{xcolor}
\usepackage{float}
\usepackage{placeins}
\usepackage{subcaption}
\usepackage{enumitem}
\usepackage{hyperref}

\newtheorem{theorem}{Theorem}[section]
\newtheorem{lemma}[theorem]{Lemma}
\newtheorem{proposition}[theorem]{Proposition}

\newtheorem{example}{Example}[section]
\newtheorem{remark}{Remark}[section]
\newtheorem{definition}{Definition}[section]

\begin{document}

\title[Hausdorff Dimension of Self-Affine Sets]{Hausdorff Dimension of a Class of Self-Affine Sets}

\author{Amal P. S.}
\address{APJ Abdul Kalam Technological University, Thiruvananthapuram, Kerala, India; Rajagiri School of Engineering and Technology, Kochi, Kerala, India}
\email{amalb9111@gmail.com}
\thanks{Corresponding author: Amal P. S.}

\author{Vinod Kumar P. B.}
\address{APJ Abdul Kalam Technological University, Thiruvananthapuram, Kerala, India; Muthoot Institute of Technology and Science, Kochi, Kerala, India}
\email{vinodkumarpb@mgits.ac.in}

\author{Ramkumar P. B.}
\address{APJ Abdul Kalam Technological University, Thiruvananthapuram, Kerala, India; Rajagiri School of Engineering and Technology, Kochi, Kerala, India}
\email{ramkumar\_pb@rajagiritech.edu.in}

\subjclass[2020]{28A80, 37C45}
\keywords{Self-affine sets, Iterated Function Systems, Hausdorff dimension, Open set condition}

\begin{abstract}
In this paper, exact Hausdorff dimension formulas for a class of self-affine attractors generated by affine Iterated Function Systems are derived. We consider systems containing an affine map whose $n$-th iterate is a similarity contraction, alongside standard similarities whose linear parts commute with the symmetric operator $A^\top A$, where $A$ is the linear part of the affine map. We prove that the attractor of such a system exists uniquely, and, under the Open Set Condition, we compute its exact Hausdorff dimension. We extend this framework to systems where all map compositions of some fixed length are similarities, and to systems where overlaps are exact homothetic copies of the attractor. We unify these approaches to establish dimension formulas for hybrid systems that combine multiple eventually contractive affine maps with universally aligned similarities. Finally, we conclude with a topological classification of these systems in the plane. For a two-map system comprising an affine map whose second iterate is a similarity with contraction ratio $c$, alongside an $f$-aligned similarity with ratio $r$, we prove that the precise parameter balance $c + r = 1$ acts as a strict topological bottleneck uniquely guaranteeing both the open set condition and the connectedness of the attractor.
\end{abstract}

\maketitle

\section*{Introduction}

The mathematical study of fractal geometry was rigorously formalized by Hutchinson in 1981. In his seminal work \cite{hutchinson1981}, Hutchinson introduced the framework of Iterated Function Systems (IFS) and proved that for any finite set of contraction mappings on a complete metric space, there exists a unique compact invariant set, known as the attractor. This framework was further developed and popularized by Barnsley \cite{barnsley1985, barnsley1986, barnsley2014}, whose Collage Theorem provided the inverse algorithm for approximating natural images with fractals, cementing the IFS as a central tool in both theoretical and applied geometry.

For systems consisting of similarity transformations (often referred to as self-similar systems), the dimension theory is well-developed. If the system satisfies a non-overlapping condition known as the \textit{Open Set Condition} (OSC), the Hausdorff dimension of the attractor is the unique solution $s$ to the Moran equation:
\begin{equation} \label{eq:moran}
\sum_{i=1}^m r_i^s = 1,
\end{equation}
where $r_i$ are the contraction ratios of the maps \cite{moran1946}. The Open Set Condition, crucial for the validity of this formula, requires the existence of a non-empty open set $U$ such that $\bigcup_{i} f_i(U) \subset U$ and $f_i(U) \cap f_j(U) = \emptyset$ for $i \neq j$. Schief \cite{schief1994} later strengthened this result by showing that the OSC is essentially necessary for the Hausdorff dimension to coincide with the similarity dimension defined by \eqref{eq:moran}.

In sharp contrast, the analysis of self-affine sets—where the generating maps contract space non-uniformly—presents fundamental difficulties. Unlike the self-similar case, there is no universal scalar equation like \eqref{eq:moran} that determines the dimension for all configurations. The standard theory, established by Falconer \cite{falconer1988}, provides a method to estimate the dimension using singular value functions (the \textit{affinity dimension}). However, this approach generally yields an upper bound rather than an exact formula for specific, explicitly defined systems.

Since Falconer's foundational work, considerable effort has been dedicated to determining precisely when the Hausdorff dimension coincides with the affinity dimension. Hueter and Lalley \cite{hueter1993} provided some of the earliest sufficient geometric conditions for the validity of Falconer's formula for totally disconnected self-affine sets in $\mathbb{R}^2$. More recently, breakthroughs have been achieved by analyzing the ergodic theory of self-affine measures and their projections. Bárány, Hochman, and Rapaport \cite{barany2017} proved that for planar self-affine sets satisfying the strong open set condition, the Hausdorff dimension equals the affinity dimension under mild irreducibility and non-compactness assumptions. Morris and Shmerkin \cite{morris2016} established similar dimension equalities by showing that the affinity dimension can be approximated by the Lyapunov dimension of self-affine measures on positive subsystems. The role of the associated Furstenberg measures and projection properties has been further elucidated by Falconer and Kempton \cite{falconer2016}, while Bárány, Käenmäki, and Koivusalo \cite{barany2018} established that, given fixed translation vectors, the Hausdorff dimension equals the affinity dimension for almost all choices of matrices, relying on Ledrappier-Young theory.

The scope of the theory has also been expanded to address systems that do not satisfy strict separation conditions. Jordan, Pollicott, and Simon \cite{jordan2006} utilized random translational perturbations to compute the dimension of self-affine attractors, bypassing traditional norm restrictions. For systems with overlaps, Bárány, Rams, and Simon \cite{barany2015} calculated dimensions for diagonally affine planar iterated function systems, and Hochman and Rapaport \cite{hochman2019} computed the dimension for planar self-affine sets with overlaps assuming exponential separation and total irreducibility. Furthermore, generalized approaches utilizing positive linear operators and extensions of the Krein-Rutman theorem have been employed by Nussbaum, Priyadarshi, and Verduyn Lunel \cite{nussbaum2012} to characterize the dimension of invariant sets generated by infinitesimal similitudes.

Concurrent with these developments in dimension theory, the classical Iterated Function System framework has undergone various structural generalizations. Among these, the strict requirement that every generating map must be a strict contraction has been relaxed in several ways. One such approach is the recent introduction of $G^n$-contractions and $G$-Iterated Function Systems ($G$-IFS) \cite{amal2026}, which extended the classical theory to systems where the generating functions become strictly contractive only under higher-order iteration. While this prior work primarily established the existence and topological properties of $G$-attractors, it also observed that the Hausdorff dimension of certain specific $G$-attractors might satisfy an algebraic relation analogous to the classical Moran equation. This observation serves as the foundational premise for the exact dimension formulas we derive in the present paper.

Building upon this premise, we identify a class of affine systems that admit precise, simple dimension formulas. We show that when the affine maps satisfy specific algebraic conditions—namely, \textit{$f$-aligned similarity} or \textit{$k$-iterate similarity}—the dimension calculation simplifies dramatically. This allows us to find exact values without the ambiguity of generic estimates.

The paper is organized as follows. In Section 1, we introduce the concept of $f$-aligned similarity, defined by the commutation of similarity maps with the symmetric matrix of a general affine map, and prove that this condition allows for an exact dimension calculation. In Section 2, we provide illustrative examples, explicitly verifying these conditions for specific systems. In Section 3, we present an example to demonstrate the importance of the alignment condition, showing that the dimension formula fails when the similarity maps do not satisfy the alignment condition. In Section 4, we generalize our results to \textit{$k$-iterate similarity systems}, analyzing affine maps that become similarities under composition. In Section 5, we unify these concepts by establishing a dimension formula for \textit{hybrid systems} that combine eventually-contractive affine maps with universally aligned similarity contractions. Finally, in Section 6, we provide a complete algebraic and topological classification for planar systems generated by exactly two maps, establishing the exact parameter constraints required for the open set condition and global connectedness to hold.

\section{f-aligned similarity}

The study of self-affine sets is often complicated by the general nature of affine transformations that contract or stretch space differently in different directions. However, certain affine systems possess a hidden structural regularity that allows the attractor to be decomposed into countably many self-similar copies of itself. We illustrate this phenomenon with the following explicit example.

\begin{example} \label{ex:direct_calculation}
Consider the affine IFS on $\mathbb{R}^2$ generated by the maps:
\begin{align*}
    f(x, y) &= \left( -\frac{y}{3}, \, x \right), \\
    g(x, y) &= \left( \frac{2}{3}x - 1, \, \frac{2}{3}y \right).
\end{align*}
The map $g$ is a similarity with contraction ratio $r_g = 2/3$. The map $f$ is not a strict contraction in the Euclidean metric, as it preserves distances for vectors aligned with the $x$-axis. However, the system is eventually contractive. Specifically, the second iterate $f^2(x,y) = (-x/3, -y/3)$ is a strict contraction with ratio $1/3$. Since every composition of length 2 (i.e., $f^2, fg, gf, g^2$) is a strict contraction, the existence of a unique non-empty compact attractor $A$ is guaranteed.

We now verify that this system satisfies the Open Set Condition (OSC). Consider the open square region $U = (-3, 1) \times (-3, 1)$.

The image of $U$ under $f$ is:
\[
    f(U) = \left\{ \left(-\frac{y}{3}, x\right) : -3 < x < 1, -3 < y < 1 \right\} = \left(-\frac{1}{3}, 1\right) \times (-3, 1).
\]
The image of $U$ under $g$ is:
\begin{align*}
    g(U) &= \left\{ \left(\frac{2x}{3} - 1, \frac{2y}{3}\right) : -3 < x < 1, -3 < y < 1 \right\} \\
         &= \left(-3, -\frac{1}{3}\right) \times \left(-2, \frac{2}{3}\right).
\end{align*}
It is evident that $f(U) \subset U$ and $g(U) \subset U$. Furthermore, the projections of $f(U)$ and $g(U)$ onto the first coordinate axis are the disjoint intervals $(-1/3, 1)$ and $(-3, -1/3)$ respectively. Thus, $f(U) \cap g(U) = \emptyset$, confirming that the OSC holds.

The attractor $A$ satisfies the decomposition (see Figure \ref{fig:infinite_decomposition}):
\begin{equation} \label{eq:decomp_union}
    A = g(A) \cup \bigcup_{k=0}^{\infty} h_k(A),
\end{equation}
where $h_k = f \circ g^k \circ f$. Explicit calculation shows that $h_k$ is a similarity with ratio $r_k = \frac{1}{3} (\frac{2}{3})^k$. The disjointness of $f(U)$ and $g(U)$ ensures that the components in \eqref{eq:decomp_union} are $\mathcal{H}^s$-almost disjoint.

\begin{figure}[htbp]
    \centering
    \includegraphics[width=0.75\textwidth]{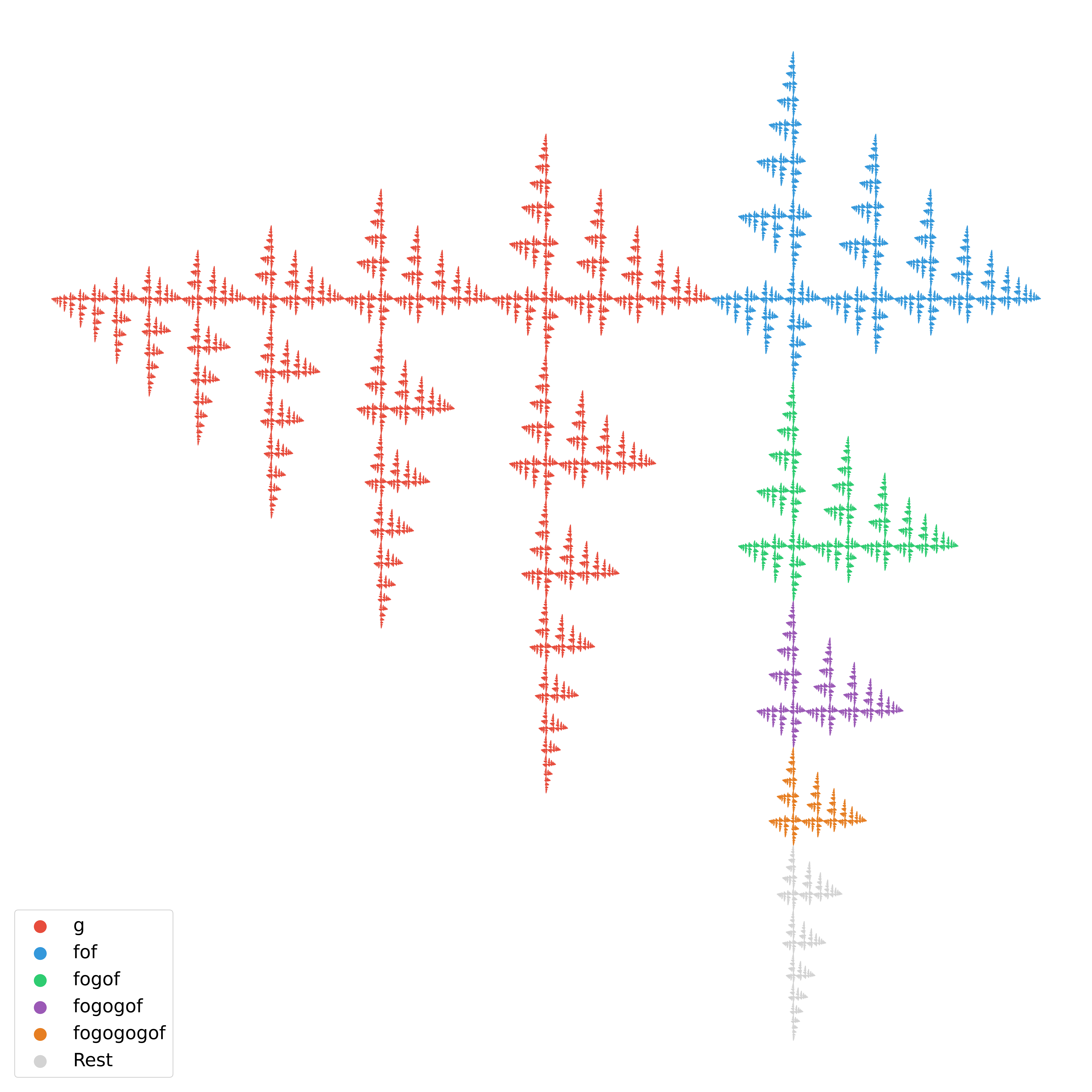}
    \caption{The attractor $A$ decomposed into self-similar subsets. The large red region is the image $g(A)$. The remaining sequence of smaller regions corresponds to the images under the maps $f \circ g^k \circ f$ for $k \ge 0$.}
    \label{fig:infinite_decomposition}
\end{figure}

Let $s = \dim_{\mathrm{H}}(A)$. Applying the $s$-dimensional Hausdorff measure $\mathcal{H}^s$ to \eqref{eq:decomp_union} and using the scaling property $\mathcal{H}^s(T(E)) = c^s \mathcal{H}^s(E)$ for similarities, we obtain:
\[
    \mathcal{H}^s(A) = \left(\frac{2}{3}\right)^s \mathcal{H}^s(A) + \sum_{k=0}^{\infty} \left[ \frac{1}{3} \left(\frac{2}{3}\right)^k \right]^s \mathcal{H}^s(A).
\]
Assuming for the moment that $0 < \mathcal{H}^s(A) < \infty$ (which we will prove later), dividing by $\mathcal{H}^s(A)$ yields:
\[
    1 = \left(\frac{2}{3}\right)^s + \left(\frac{1}{3}\right)^s \sum_{k=0}^{\infty} \left[ \left(\frac{2}{3}\right)^s \right]^k.
\]
Summing the geometric series (where $(2/3)^s < 1$) gives:
\[
    1 = \left(\frac{2}{3}\right)^s + \frac{(1/3)^s}{1 - (2/3)^s}.
\]
Rearranging terms, we find:
\[
    1 - \left(\frac{2}{3}\right)^s = \frac{(1/3)^s}{1 - (2/3)^s} \implies \left[ 1 - \left(\frac{2}{3}\right)^s \right]^2 = \left(\frac{1}{3}\right)^s.
\]
Taking the square root leads to the characteristic equation:
\[
    \left(\frac{2}{3}\right)^s + \left(\frac{1}{\sqrt{3}}\right)^s = 1.
\]
Numerical approximation yields the dimension $s \approx 1.464$.
\end{example}

The preceding example exhibits two structural regularities that extend beyond standard self-similar geometry. First, the map $f$, while not a similarity itself, becomes one upon iteration (specifically, $f^2$ is a similarity). Second, the interaction between $f$ and $g$ allows the stretching of $f$ to align with the scaling of $g$. We now formalize these properties by introducing the concepts of \textit{$G^n$-similarity contractions} and \textit{$f$-aligned similarities}.

\begin{definition}
Let $f: \mathbb{R}^m \to \mathbb{R}^m$ be an affine map and let $n$ be a positive integer. The map $f$ is called a $G^n$-similarity contraction if the $n$-th iterate $f^n$ (the composition of $f$ with itself $n$ times) is a similarity contraction. Explicitly, there exists a constant $c \in (0, 1)$ such that:
\begin{equation}
    \| f^n(x) - f^n(y) \| = c \| x - y \| \quad \text{for all } x, y \in \mathbb{R}^m.
\end{equation}
If in addition there is no integer $k$, $1 \le k < n$, such that $f^k$ is a similarity contraction, then $f$ is called a strict $G^n$-similarity contraction.
\end{definition}

\begin{definition}\label{def:f_aligned}
Let $f : \mathbb{R}^m \to \mathbb{R}^m$ be an affine transformation with linear part $A$.
A similarity transformation $g : \mathbb{R}^m \to \mathbb{R}^m$ with linear part $S$ is called $f$-aligned if
\begin{equation}
    S (A^\top A) = (A^\top A) S.
\end{equation}
Equivalently, $S$ commutes with the symmetric matrix $A^\top A$, which determines the stretching induced by $f$.
Geometrically, this means that the action of $g$ preserves the eigenspaces of $A^\top A$, ensuring that its scaling and rotational components are aligned with the principal axes of $f$.
\end{definition}

\begin{remark}
    The geometric significance of the $f$-alignment condition is that it ensures the sets $f(g(A))$ and $g(f(A))$ are geometrically congruent (up to translation) for any set $A \subset \mathbb{R}^m$. Without this condition, the order of composition matters: applying a non-aligned rotation before $f$ can result in a set with different geometric properties than applying it after. The $f$-alignment guarantees that the stretching remains consistent, regardless of the composition order.
\end{remark}

\begin{theorem}\label{thm:existence_uniqueness_aligned}
Let $f: \mathbb{R}^m \to \mathbb{R}^m$ be a $G^n$-similarity contraction with contraction ratio $c$ (i.e., $f^n$ is a similarity with ratio $c < 1$). Let $\{g_1, \dots, g_p\}$ be a set of $f$-aligned similarity contractions on $\mathbb{R}^m$ with ratios $\{c_1, \dots, c_p\}$. Then the Iterated Function System (IFS) $\mathcal{F} = \{f, g_1, \dots, g_p\}$ admits a unique attractor.
\end{theorem}

\begin{proof}
We analyze the contraction factor of an arbitrary composite map $w$ of length $L$ generated by the IFS. Let $q$ be the number of occurrences of $f$ in $w$, and let $\{g_{i_1}, \dots, g_{i_{L-q}}\}$ be the subsequence of similarity maps appearing in $w$.

The $f$-alignment condition implies that for any vector $v \in \mathbb{R}^m$, we have:
\[ \| f(g_k(v)) \| = c_{i_k} \| f(v) \|. \]
This allows us to factor the scalar contraction ratios of the $g$-maps out of the composition norm, regardless of their position in the sequence:
\[ \text{Lip}(w) = \left( \prod_{k=1}^{L-q} c_{i_k} \right) \text{Lip}(f^q). \]

We now bound $\text{Lip}(f^q)$. Writing $q = k n + j$ with $0 \le j < n$, we have $f^q = f^j \circ (f^n)^k$. Since $f^n$ is a similarity with ratio $c$, $(f^n)^k$ has ratio $c^k$. Letting $M = \max_{0 \le j < n} \text{Lip}(f^j)$, we obtain:
\[ \text{Lip}(f^q) \le M c^{\lfloor q/n \rfloor}. \]

Substituting this back into the expression for $\text{Lip}(w)$, and letting $r_{\max} = \max\{c_1, \dots, c_p\} $:
\[ \text{Lip}(w) \le M (r_{\max})^{L-q} c^{\lfloor q/n \rfloor}. \]

Since $r_{\max} < 1$ and $c < 1$, the term $(r_{\max})^{L-q} c^{\lfloor q/n \rfloor}$ approaches 0 as $L \to \infty$ (at least one exponent diverges). Thus, for sufficiently large $L$, every map in $\mathcal{F}^L$, the set of all composition functions of length $L$ generated by $\mathcal{F}$, is a strict contraction. By Hutchinson's Theorem applied to $\mathcal{F}^L$, there exists a unique compact attractor.
\end{proof}

Having established the existence and uniqueness of the attractor, we now turn our attention to determining its fractal dimension.

\begin{lemma} \label{lem:hutchinson_counting}
Let $\{V_i\}$ be a collection of disjoint open subsets of $\mathbb{R}^n$. Let $U$ be a ball of radius $r$. Suppose there exist constants $0 < c_1 \le c_2 < \infty$ such that every set $V_i$ contains a ball of radius $c_1 r$ and is contained in a ball of radius $c_2 r$.

Then the number of closures $\overline{V_i}$ that intersect $U$ is bounded by a constant $M$ dependent only on the dimension $n$ and the ratio $c_2/c_1$:
\[ M \le (1 + 2c_2)^n c_1^{-n}. \]
\begin{flushright}
(Hutchinson \cite{hutchinson1981}, Lemma 5.3)
\end{flushright}
\end{lemma}

\begin{proposition}[Mass Distribution Principle] \label{prop:mass_distribution}
Let $\mu$ be a mass distribution (a finite Borel measure) on a set $A$ with $0 < \mu(A) < \infty$. If there exist constants $C > 0$ and $\delta > 0$ such that $\mu(U) \le C (\text{diam}(U))^s$ for all sets $U$ with diameter $\text{diam}(U) \le \delta$, then $\mathcal{H}^s(A) \ge \mu(A)/C$.
\begin{flushright}
(Falconer \cite{falconer_book}, Principle 4.2)
\end{flushright}
\end{proposition}

\begin{theorem} \label{thm:general_dimension}
Let $f: \mathbb{R}^m \to \mathbb{R}^m$ be a $G^n$-similarity contraction such that the $n$-th iterate $f^n$ is a similarity with ratio $c \in (0, 1)$. Let $g_1, \dots, g_k$ be $k$ distinct $f$-aligned similarity contractions on $\mathbb{R}^m$ with ratios $r_1, \dots, r_k$, respectively.

Assuming the Open Set Condition (OSC) holds, the Hausdorff dimension $s$ of the attractor $A$ of the IFS $\mathcal{F} = \{f, g_1, \dots, g_k\}$ is the unique real number satisfying:
\begin{equation} \label{eq:general_formula}
    c^{s/n} + \sum_{i=1}^k r_i^s = 1.
\end{equation}
Moreover, the $s$-dimensional Hausdorff measure satisfies $0 < \mathcal{H}^s(A) < \infty$.
\end{theorem}

\begin{proof}
First, we prove that $\mathcal{H}^s(A) < \infty$. Let $s$ be the unique solution to Equation \eqref{eq:general_formula}. We assign a ``formal weight'' $w(\psi)$ to each function $\psi \in \{f, g_1, \dots, g_k\}$ as follows:
\[
    w(f) = c^{1/n}, \quad \text{and} \quad w(g_i) = r_i \quad \text{for } i=1,\dots,k.
\]
By the definition of $s$, these weights satisfy the normalization condition:
\begin{equation} \label{eq:sum_unity_final}
    \sum_{\psi \in \mathcal{F}} w(\psi)^s = 1.
\end{equation}

Let $\mathcal{I}^p$ be the set of all finite words of length $p$ formed by  the maps in $\mathcal{F}$. For any word $\omega = (\omega_1, \dots, \omega_p) \in \mathcal{I}^p$, we define the composite map $h_\omega = \psi_{\omega_1} \circ \dots \circ \psi_{\omega_p}$. The set $h_\omega(A)$ represents the image of the attractor under this map. The collection $\{ h_\omega(A) : \omega \in \mathcal{I}^p \}$ forms a cover of $A$.

Let $p$ be large enough so that it can be decomposed as $p = n p_1 + p_2$, where $p_1, p_2$ are positive integers and $p_2$ satisfies $n-1 \le p_2 \le 2(n-1)$.

Consider an arbitrary word $\omega \in \mathcal{I}^p$. Since the maps $g_i$ are $f$-aligned similarity contractions, the  diameter of the image set is invariant under permutations of the functions in the composition $h_\omega$. If $\pi(\omega)$ is a permutation of $\omega$, then:
\[ \text{diam}(h_\omega(A)) = \text{diam}(h_{\pi(\omega)}(A)).\]
We rearrange the sequence of functions composing $h_\omega$ into two components: a ``head'' block $h^{\text{head}}$ of length $n p_1$ and a ``tail'' block $h^{\text{tail}}$ of length $p_2$.
Since the tail length $p_2 \ge n-1$, we can always partition the total count of $f$'s in $\omega$ such that the head block contains exactly a multiple of $n$ instances of $f$ (say, $q n$ times), and the tail block contains the remaining instances.

We define the \textit{Total Formal Weight} of the word $\omega$ as $W(\omega) = \prod_{j=1}^p w(\omega_j)$.
We now bound the diameter $\text{diam}(h_\omega(A))$ in terms of $W(\omega)$.

First, consider the head estimate. The head contains exactly $q n$ instances of $f$. Since $f^{q n}$ is a similarity with ratio $c^q = (c^{1/n})^{q n}$, the geometric contraction of the head matches the product of its formal weights exactly. For any bounded set $B$:
\[ \text{diam}(h^{\text{head}}(B)) = \left( \prod_{\psi \in \text{head}} w(\psi) \right) \text{diam}(B). \]

Next, consider the tail estimate. The tail has length $p_2$. We define a uniform distortion constant $K$ over the set of all possible tails $\mathcal{T} = \bigcup_{j=n-1}^{2(n-1)} \mathcal{I}^j$.
\[ K = \sup_{\nu \in \mathcal{T}} \frac{\text{diam}(h_\nu(A))}{W(\nu) \text{diam}(A)}. \]
Since $\mathcal{T}$ is a finite set, $K < \infty$. Thus:
\[ \text{diam}(h^{\text{tail}}(A)) \le K \cdot W(\text{tail}) \cdot \text{diam}(A). \]

Combining these estimates:
\[ \text{diam}(h_\omega(A)) \le \text{diam}(h^{\text{head}}(h^{\text{tail}}(A))) \le W(\text{head}) \cdot K \cdot W(\text{tail}) \cdot \text{diam}(A). \]
Since $W(\text{head}) \cdot W(\text{tail}) = W(\omega)$, we obtain the uniform bound:
\[ \text{diam}(h_\omega(A)) \le K \text{diam}(A) W(\omega). \]

Using these geometric bounds, we estimate the sum of the $s$-th powers of the diameters for the cover at level $p$:
\[ \Sigma_p = \sum_{\omega \in \mathcal{I}^p} (\text{diam}(h_\omega(A)))^s \le \sum_{\omega \in \mathcal{I}^p} (K \text{diam}(A) W(\omega))^s. \]
Factoring out the constants:
\[ \Sigma_p \le K^s (\text{diam}(A))^s \sum_{\omega \in \mathcal{I}^p} W(\omega)^s. \]
The summation term is exactly the expansion of the sum of weights raised to the power $p$:
\[ \sum_{\omega \in \mathcal{I}^p} \left( \prod_{j=1}^p w(\omega_j) \right)^s = \left( \sum_{\psi \in \mathcal{F}} w(\psi)^s \right)^p. \]
By Equation \eqref{eq:sum_unity_final}, the term inside the parentheses is 1. Therefore:
\[ \Sigma_p \le K^s (\text{diam}(A))^s  < \infty. \]
Since the sum is uniformly bounded for all $p$, we conclude that $\mathcal{H}^s(A) < \infty$.

Now, we prove that $\mathcal{H}^s(A) > 0$. We define a mass distribution $\mu$ on the attractor $A$ by assigning mass to the sets $h_\omega(A)$ such that $\mu(h_\omega(A)) = (W(\omega))^s$. To ensure this assignment yields a valid measure, we verify the consistency condition: for any word $\omega$, the mass assigned to $h_\omega(A)$ must equal the sum of the masses of its images under the maps in $\mathcal{F}$. Using the relation $\sum_{\psi \in \mathcal{F}} w(\psi)^s = 1$, we have:
\[
    (W(\omega))^s = (W(\omega))^s \cdot 1 = (W(\omega))^s \left( \sum_{\psi \in \mathcal{F}} w(\psi)^s \right) = \sum_{\psi \in \mathcal{F}} (W(\omega \psi))^s.
\]
This additivity ensures that the definition is consistent across all levels of the construction. Hence, $\mu$ is a well-defined measure of total mass 1 supported on $A$.

Let $\mathcal{I}^* = \bigcup_{p=1}^\infty \mathcal{I}^p$ denote the set of all finite words. For any word $\omega = (\omega_1, \dots, \omega_p)$, let $\omega^- = (\omega_1, \dots, \omega_{p-1})$ denote the word obtained by removing the last symbol.

Let $U$ be an arbitrary ball of radius $r < 1$. We define a finite cut-set $\mathcal{Q} \subset \mathcal{I}^*$ by truncating every sequence at the first index where the formal weight drops below $r$:
\[ \mathcal{Q} = \{ \omega \in \mathcal{I}^* : W(\omega) \le r < W(\omega^-) \}. \]
Let $w_{\min} = \min_{\psi} w(\psi)$. The stopping condition ensures that for all $\omega \in \mathcal{Q}$:
\begin{equation} \label{eq:weight_comparable}
    w_{\min} r < W(\omega) \le r.
\end{equation}

By the Open Set Condition, there exists a non-empty bounded open set $V$ such that the images $\{ \psi(V) : \psi \in \mathcal{F} \}$ are disjoint subsets of $V$. Since $V$ is open and bounded, we can find finite positive real numbers $a_1$ and $a_2$ such that $V$ contains a ball of radius $a_1$ and is contained in a ball of radius $a_2$. The collection $\mathcal{V} = \{ h_\omega(V) : \omega \in \mathcal{Q} \}$ consists of disjoint open sets. Moreover, the attractor is contained in the union of their closures:
\[ A \subseteq \bigcup_{\omega \in \mathcal{Q}} \overline{h_\omega(V)}. \]
We now verify that $\mathcal{V}$ satisfies the conditions of Lemma \ref{lem:hutchinson_counting}.

Consider any $\omega \in \mathcal{Q}$ of length $p$. We decompose the length $p = n p_1 + p_2$ and the map $h_\omega = h^{\text{head}} \circ h^{\text{tail}}$ as in the Upper Bound proof.
The head $h^{\text{head}}$ is a similarity with ratio $W(\text{head})$. The tail map $h^{\text{tail}}$ belongs to $\mathcal{T}$, the finite set of compositions generated by $\mathcal{F}$ with lengths ranging from $n-1$ to $2(n-1)$. For each map $\psi \in \mathcal{T}$, let $\sigma_{\min}(\psi)$ and $\sigma_{\max}(\psi)$ denote the smallest and largest singular values of its linear part. We define the uniform bounds:
\[ c_{\min} = \min_{\psi \in \mathcal{T}} \sigma_{\min}(\psi), \quad c_{\max} = \max_{\psi \in \mathcal{T}} \sigma_{\max}(\psi), \quad \text{and} \quad w_{\min}^{\text{tail}} = \min_{\nu \in \mathcal{T}} W(\nu). \]
Since $\mathcal{T}$ is finite and consists of non-singular affine maps, $0 < c_{\min} \le c_{\max} < \infty$ and $w_{\min}^{\text{tail}} > 0$. Consequently, for any map $\psi \in \mathcal{T}$ and any ball $B$ of radius $\rho$, the image set $\psi(B)$ contains a ball of radius $c_{\min} \rho$ and is contained in a ball of radius $c_{\max} \rho$.

\begin{enumerate}
    \item Inner Ball Condition:
    Since $V$ contains a ball of radius $a_1$, $h^{\text{tail}}(V)$ contains a ball of radius $c_{\min} a_1$. Consequently, $h_\omega(V)$ contains a ball of radius $W(\text{head}) c_{\min} a_1$.
    Using the bound $W(\text{head}) \ge w_{\min} r$:
    \[ W(\text{head}) c_{\min} a_1 \ge (w_{\min} c_{\min} a_1) r. \]
    Let $k_1 = w_{\min} c_{\min} a_1$. Thus, $h_\omega(V)$ contains a ball of radius $k_1 r$.

    \item Outer Ball Condition:
    Since $V$ is contained in a ball of radius $a_2$, $h^{\text{tail}}(V)$ is contained in a ball of radius $c_{\max} a_2$. Consequently, $h_\omega(V)$ is contained in a ball of radius $W(\text{head}) c_{\max} a_2$.
    Using the bound $W(\text{head}) \le r/w_{\min}^{\text{tail}}$:
    \[ W(\text{head}) c_{\max} a_2 \le \left( \frac{c_{\max} a_2}{w_{\min}^{tail}} \right) r. \]
    Let $k_2 = c_{\max} a_2 (w_{\min}^{\text{tail}})^{-1}$. Thus, $h_\omega(V)$ is contained in a ball of radius $k_2 r$.
\end{enumerate}

The collection $\mathcal{V}$ consists of disjoint open sets. Furthermore, we have established that every set $h_\omega(V) \in \mathcal{V}$ contains a ball of radius $k_1 r$ and is contained in a ball of radius $k_2 r$. Thus, the collection satisfies the hypothesis of Lemma \ref{lem:hutchinson_counting}. Consequently, the number of closures $\overline{h_\omega(V)}$ intersecting the ball $U$ is bounded by a constant $M$.
Since $h_\omega(A) \subset \overline{h_\omega(V)}$, the number of sets $h_\omega(A)$ intersecting $U$ is also bounded by $M$.

The mass of $U$ is bounded by the sum of the masses of these intersecting sets:
\[ \mu(U) \le \sum_{\substack{\omega \in \mathcal{Q} \\ h_\omega(A) \cap U \neq \emptyset}} \mu(h_\omega(A)) = \sum_{\substack{\omega \in \mathcal{Q} \\ h_\omega(A) \cap U \neq \emptyset}} (W(\omega))^s. \]
Since the number of terms in the sum is at most $M$ and $W(\omega) \le r$, we have:
\[ \mu(U) \le M r^s. \]
Since any set $E$ of diameter $d < 1$ is contained in a ball of radius $d$, we have $\mu(E) \le M d^s = M (\text{diam}(E))^s$. Hence by the Mass Distribution Principle (Proposition \ref{prop:mass_distribution}), $\mathcal{H}^s(A) > 0$.
\end{proof}

\begin{remark}
Alternatively, we can derive the dimension formula heuristically by assuming that the $s$-dimensional Hausdorff measure $\mu$ is finite and positive on the attractor $A$. Let $s = \dim_H(A)$ and let $\mu$ be the restriction of the $s$-dimensional Hausdorff measure to $A$. By the invariance of the attractor, we have the decomposition:
\begin{equation} \label{eq:decomp_heur}
    A = f(A) \cup \bigcup_{i=1}^k g_i(A).
\end{equation}
Under the Open Set Condition, these images are pairwise disjoint up to a set of $\mu$-measure zero. Using the additivity of the measure and the fact that each $g_i$ is a similarity with ratio $r_i$, we obtain:
\[
\mu(A) = \mu(f(A)) + \sum_{i=1}^k r_i^s \mu(A) = \mu(f(A)) + R(s)\mu(A),
\]
where \(R(s) = \sum_{i=1}^k r_i^s\). Rearranging terms, the measure of the image under $f$ is given by:
\begin{equation} \label{eq:step1_heur}
    \mu(f(A)) = (1 - R(s))\,\mu(A).
\end{equation}

We now iterate the map $f$. Applying $f$ to both sides of \eqref{eq:decomp_heur} yields:
\[
f(A) = f^2(A) \cup \bigcup_{i=1}^k f(g_i(A)).
\]
Since the maps $g_i$ are $f$-aligned,  $\|f(g_i(v))\| = r_i \|f(v)\|$ for any vector $v$. This implies that the set $f(g_i(A))$ is geometrically similar to $f(A)$ with scaling factor $r_i$. Consequently, the Hausdorff measure scales as:
\[
\mu(f(g_i(A))) = r_i^s \mu(f(A)).
\]
Substituting this into the measure equation for $f(A)$:
\[
\mu(f(A)) = \mu(f^2(A)) + \sum_{i=1}^k r_i^s \mu(f(A)) = \mu(f^2(A)) + R(s)\mu(f(A)).
\]
Solving for $\mu(f^2(A))$ and substituting \eqref{eq:step1_heur}:
\[
\mu(f^2(A)) = (1 - R(s))\mu(f(A)) = (1 - R(s))^2 \mu(A).
\]
Repeating this iterative process $n$ times yields:
\[
\mu(f^n(A)) = (1 - R(s))^n \mu(A).
\]
By the definition of a $G^n$-similarity contraction, the map $f^n$ acts as a similarity transformation on $\mathbb{R}^m$ with ratio $c$. Therefore, it scales the $s$-dimensional measure of any set by $c^s$:
\[
\mu(f^n(A)) = c^s \mu(A).
\]
Equating the two expressions derived for $\mu(f^n(A))$:
\[
c^s \mu(A) = (1 - R(s))^n \mu(A).
\]
Assuming  $\mu(A) \in (0, \infty)$, we divide by $\mu(A)$:
\[
c^s = (1 - R(s))^n.
\]
Taking the $n$-th root and substituting the definition of $R(s)$:
\[
c^{s/n} = 1 - \sum_{i=1}^k r_i^s,
\]
which rearranges to the desired formula:
\[
c^{s/n} + \sum_{i=1}^k r_i^s = 1.
\]
\end{remark}

\section{Illustrative Examples}

In this section, we apply the theoretical framework developed in Theorem \ref{thm:general_dimension} to specific classes of self-affine systems. We note that all affine Iterated Function Systems presented herein satisfy the hypotheses of Theorem \ref{thm:existence_uniqueness_aligned} and the alignment conditions of Definition \ref{def:f_aligned}. Furthermore, the Open Set Condition is satisfied for each system. Consequently, the existence of a unique attractor is guaranteed, and the generalized dimension formula \eqref{eq:general_formula} is directly applicable. This reduces the determination of the Hausdorff dimension to finding the unique positive solution $s$ of the corresponding scalar equation. We present three distinct examples to demonstrate the geometric diversity of the fractals generated by these systems.

\begin{example} \label{ex:golden_ratio_fractal}
Consider the affine IFS on $\mathbb{R}^2$ generated by the maps:
\begin{align*}
    f_1(x, y) &= \left( \frac{y}{2} + 1, \; x \right), \\
    g_1(x, y) &= \left( -\frac{x}{2}, \; \frac{y}{2} + 1 \right).
\end{align*}
The attractor $A_1$ generated by this system is visualized in Figure \ref{fig:golden_fractal}. The recursive structure is highlighted in the right panel, where the red region corresponds to the first-order image $g_1(A_1)$, the blue region to the second-order image $f_1^2(A_1)$, and subsequent colors represent deeper levels of recursion.

\begin{figure}[htbp]
    \centering
    \begin{subfigure}[t]{0.48\textwidth}
        \centering
        \includegraphics[width=\textwidth]{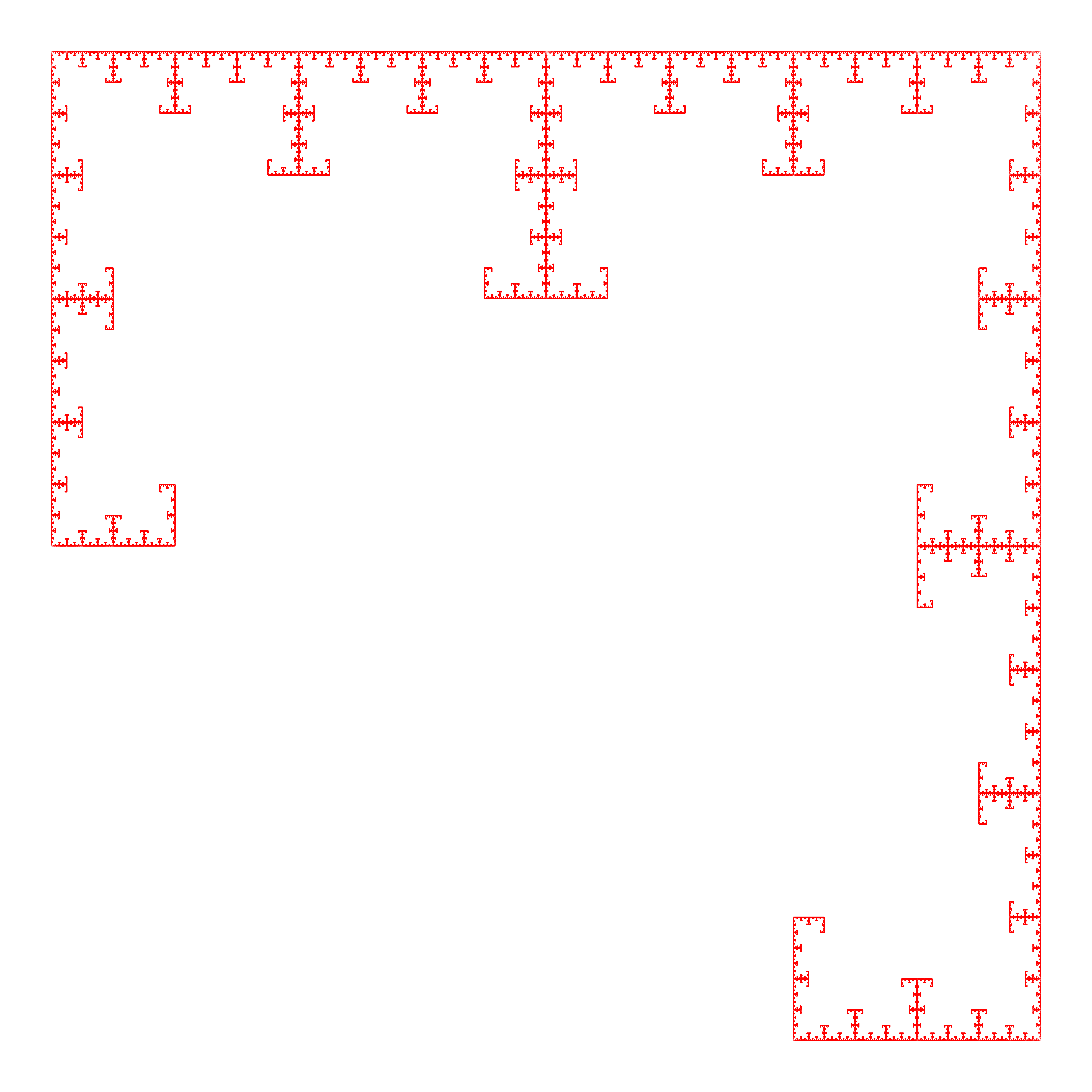}
        \caption{The attractor $A_1$.}
    \end{subfigure}
    \hfill
    \begin{subfigure}[t]{0.48\textwidth}
        \centering
        \includegraphics[width=\textwidth]{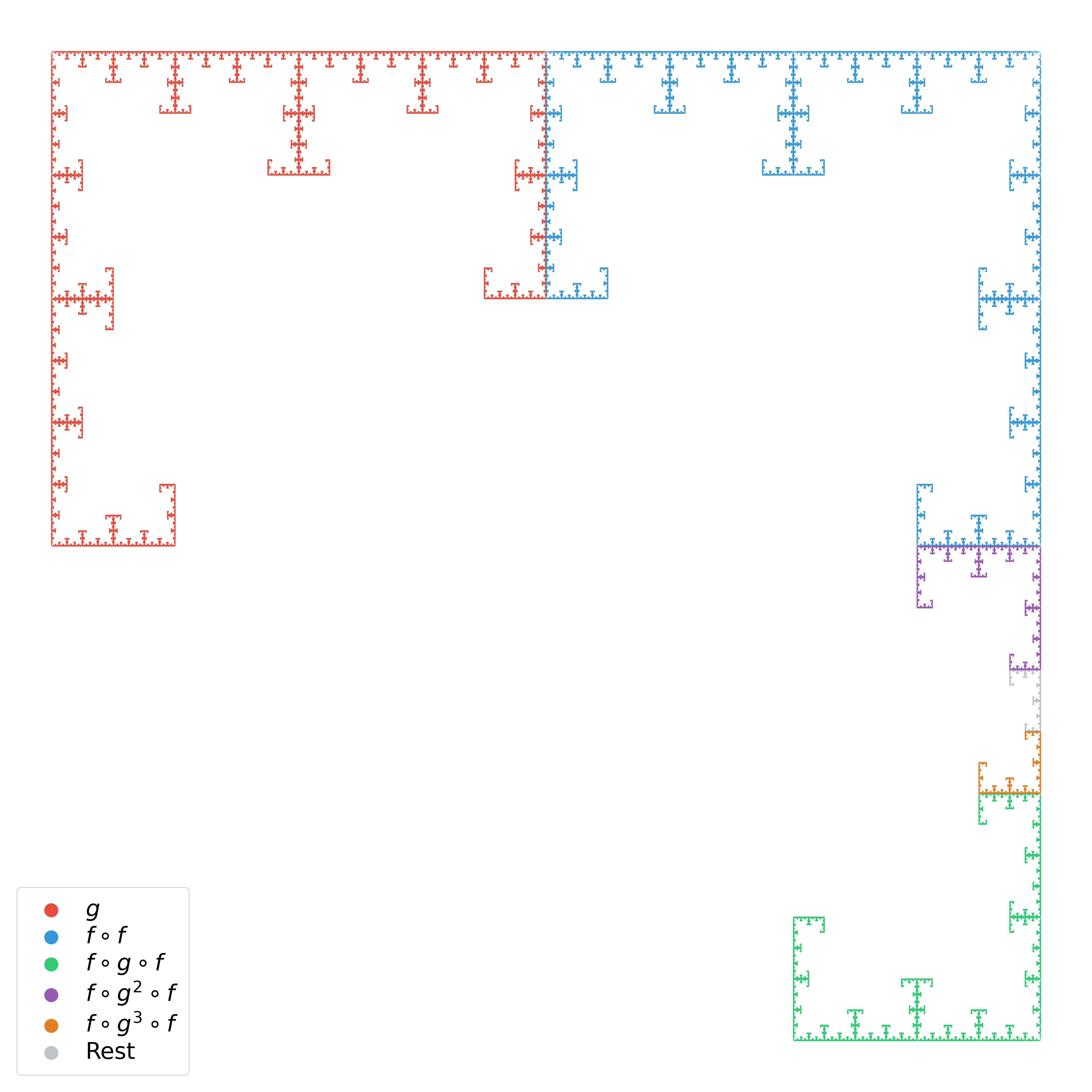}
        \caption{Decomposition of $A_1$ into self-similar copies.}
    \end{subfigure}
    \caption[The attractor A1 of the IFS \{f1, g1\} from Example \ref{ex:golden_ratio_fractal}]{The attractor $A_1$ of the IFS $\{f_1, g_1\}$ from Example \ref{ex:golden_ratio_fractal}. The multicolor plot (b) highlights the recursive structure: Red is $g_1(A_1)$, Blue is $f_1^2(A_1)$, etc.}
    \label{fig:golden_fractal}
\end{figure}

We verify that the system satisfies the conditions of Theorem \ref{thm:general_dimension}:
\begin{enumerate}
    \item The map $f_1$ has the linear part $M_{f_1} = \begin{pmatrix} 0 & 0.5 \\ 1 & 0 \end{pmatrix}$. Although $f_1$ permutes and scales coordinates differently, its second iterate $f_1^2$ acts as a similarity with ratio $c = 1/2$. Thus, we set $n=2$.
    \item The map $g_1$ is a similarity with ratio $r=1/2$. Furthermore, a direct calculation shows that the linear part of $g_1$ commutes with the symmetric matrix $M_{f_1}^\top M_{f_1}$, confirming that $g_1$ is $f_1$-aligned.
\end{enumerate}

Applying the dimension formula $c^{s/n} + r^s = 1$, we obtain:
\[
\left(\frac{1}{2}\right)^{s/2} + \left(\frac{1}{2}\right)^s = 1.
\]
Substituting $u = (1/2)^{s/2}$ yields the quadratic equation $u^2 + u - 1 = 0$. The positive solution is the inverse of the Golden Ratio, $u = 1/\varphi = (\sqrt{5}-1)/2$. Solving for $s$, we find the exact dimension:
\[
s = 2 \log_2(\varphi) \approx 1.388.
\]
\end{example}

\begin{example} \label{ex:multi_branch_fractal}
We extend the application to a system with multiple similarity branches. Let the IFS be defined by:
\begin{align*}
    f_2(x, y) &= \left( -\frac{y}{3}, \; x \right), \\
    g_{2a}(x, y) &= \left( \frac{x}{3} + 1, \; \frac{y}{3} \right), \\
    g_{2b}(x, y) &= \left( \frac{x}{3} - 1, \; \frac{y}{3} \right).
\end{align*}
The attractor $A_2$ is displayed in Figure \ref{fig:multi_branch}. The right panel explicitly visualizes the recursive decomposition, showing the primary components $g_{2a}(A_2)$ and $g_{2b}(A_2)$ alongside the self-similar progression of constituent pieces, where each colored component represents a self-similar copy of the attractor.

\begin{figure}[htbp]
    \centering
    \begin{subfigure}[t]{0.48\textwidth}
        \centering
        \includegraphics[width=\textwidth]{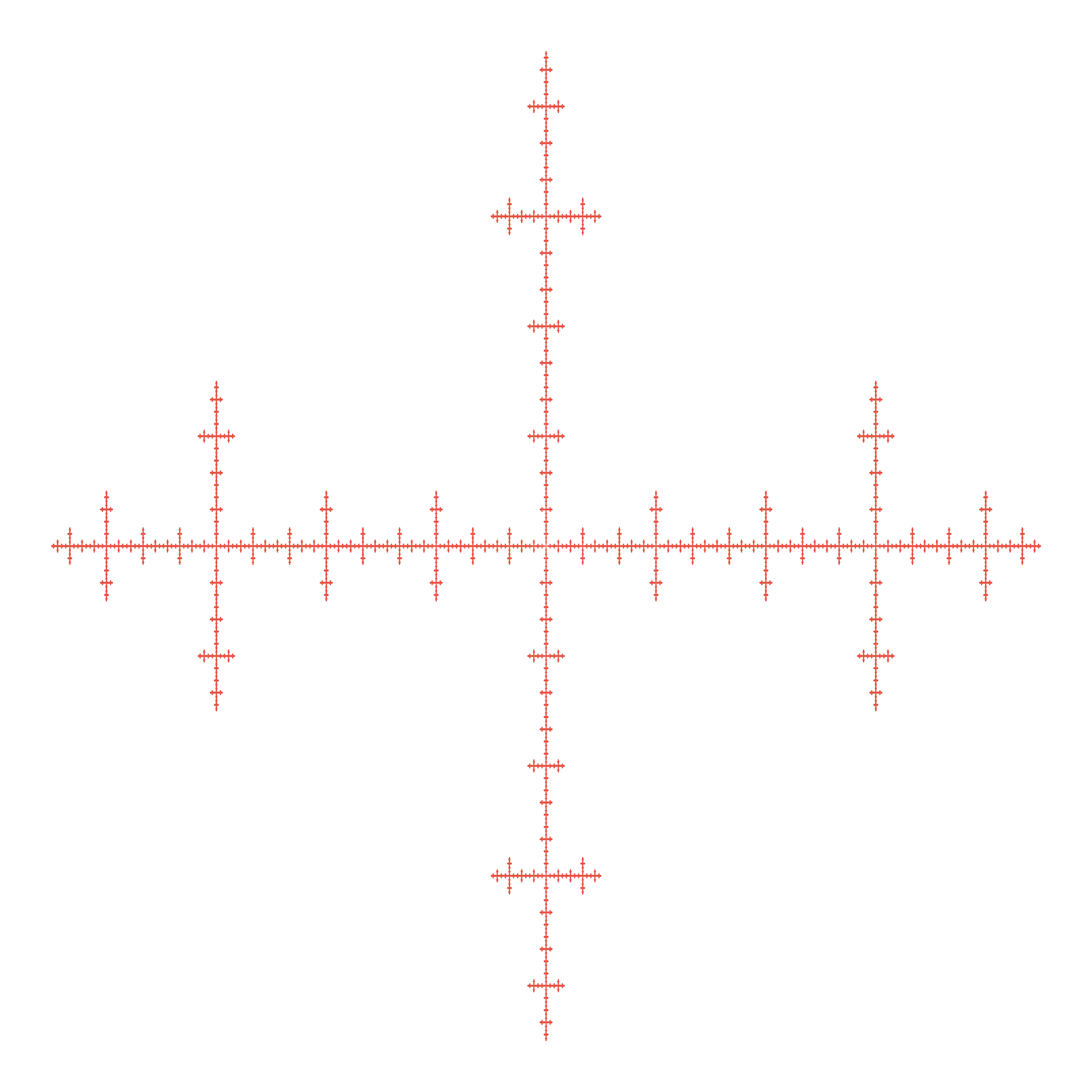}
        \caption{The attractor $A_2$.}
    \end{subfigure}
    \hfill
    \begin{subfigure}[t]{0.48\textwidth}
        \centering
        \includegraphics[width=\textwidth]{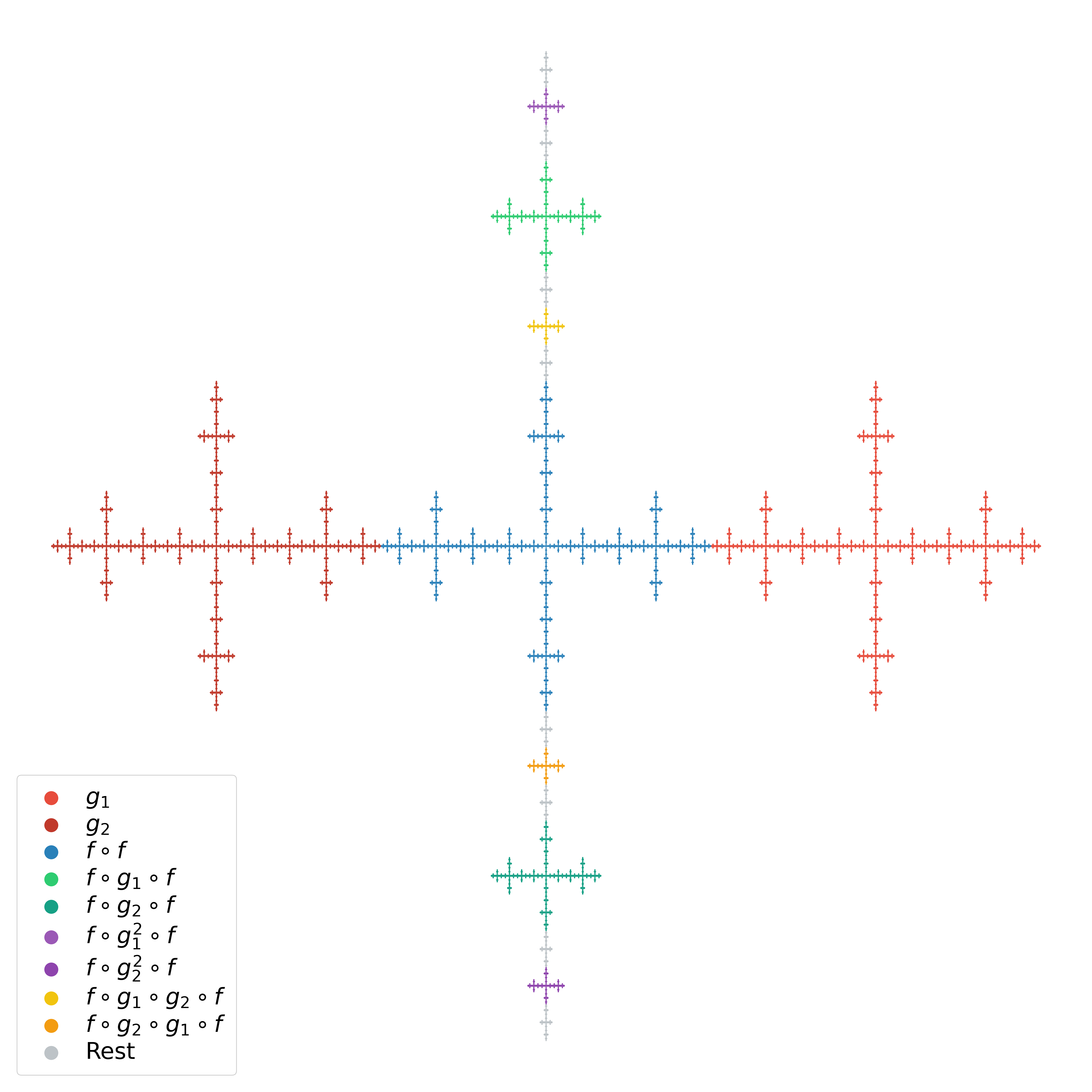}
        \caption{Decomposition of $A_2$ into self-similar copies.}
    \end{subfigure}
    \caption[The attractor $A_2$ of the IFS $\{f_2, g_{2a}, g_{2b}\}$ from Example \ref{ex:multi_branch_fractal}]{The attractor $A_2$ of the IFS $\{f_2, g_{2a}, g_{2b}\}$ from Example \ref{ex:multi_branch_fractal}. The right panel (b) visualizes the recursive decomposition into self-similar copies, where the colored regions represent the constituent components generated by the system.}
    \label{fig:multi_branch}
\end{figure}

The map $f_2$ represents a rotation by $90^\circ$ combined with non-uniform scaling, such that its second iterate $f_2^2$ is a similarity with ratio $c = 1/3$. The maps $g_{2a}$ and $g_{2b}$ are similarities with ratio $r = 1/3$ and zero rotation. Since their linear parts are scalar multiples of the identity, they commute with the linear part of $f_2$, ensuring that both $g_{2a}$ and $g_{2b}$ are $f_2$-aligned.

Substituting these parameters into the dimension formula ($n=2$) gives:
\[
c^{s/2} + r_a^s + r_b^s = 1 \implies \left(\frac{1}{3}\right)^{s/2} + 2\left(\frac{1}{3}\right)^s = 1.
\]
Letting $u = (1/3)^{s/2}$, we obtain $u + 2u^2 = 1$, which factors as $(2u - 1)(u + 1) = 0$. The unique positive solution is $u = 1/2$. Thus:
\[
\left(\frac{1}{3}\right)^{s/2} = \frac{1}{2} \implies 3^{s/2} = 2 \implies s = 2\log_3(2) = \log_3(4) \approx 1.2619.
\]
\end{example}

\begin{example} \label{ex:mixed_scaling_fractal}
We consider a third system defined by the maps:
\begin{align*}
    f_3(x, y) &= \left( \frac{y}{4}, \; 2x \right), \\
    g_3(x, y) &= \left( -\frac{x}{2}, \; -\frac{y}{2} + 1 \right).
\end{align*}
The attractor $A_3$ is displayed in Figure \ref{fig:example3}. The right panel identifies the components $g_3(A_3)$ and $f_3^2(A_3)$, alongside other smaller self-similar copies of the attractor, to illustrate the recursive structure of the system.

\begin{figure}[htbp]
    \centering
    \begin{subfigure}[t]{0.48\textwidth}
        \centering
        \includegraphics[width=\textwidth]{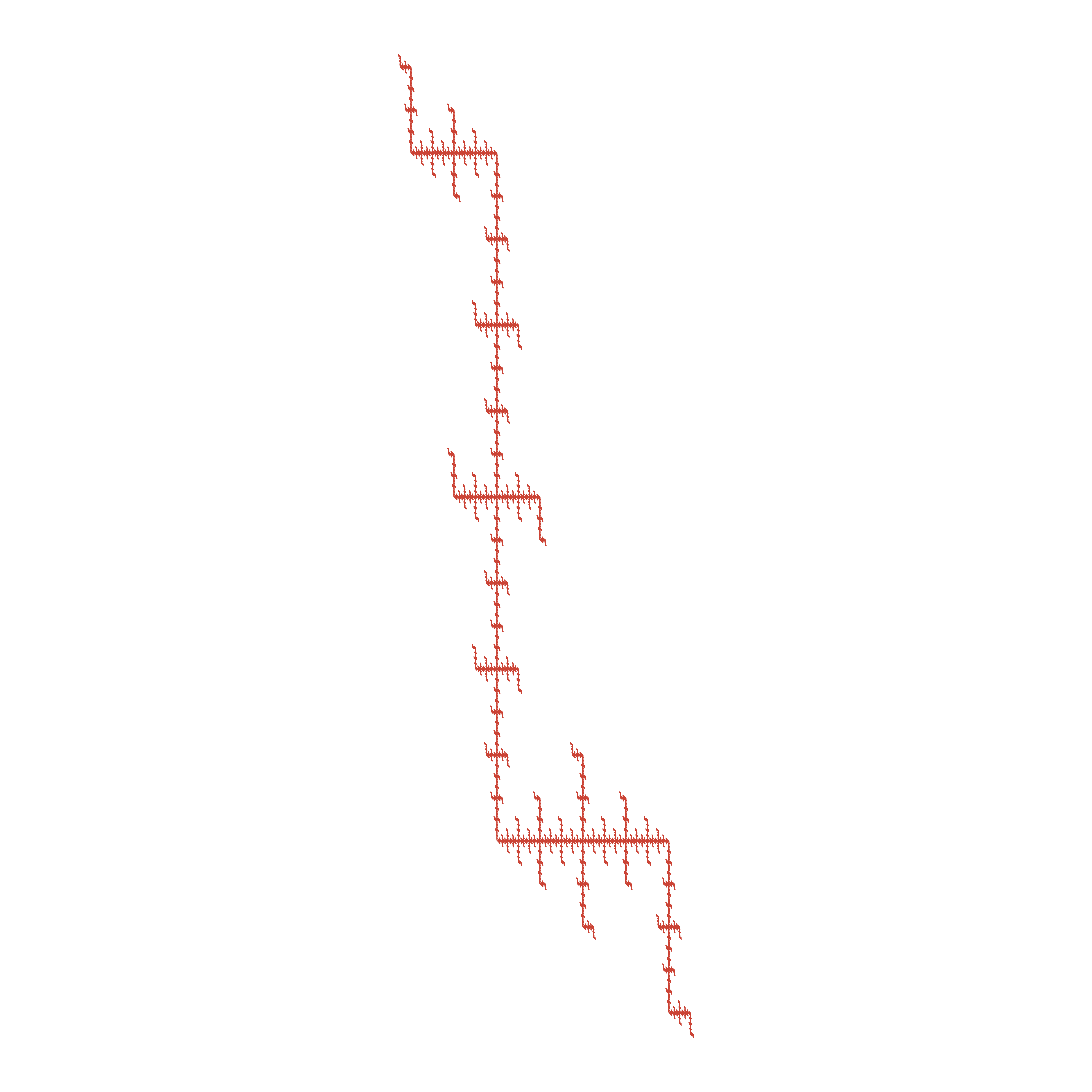}
        \caption{The attractor $A_3$.}
    \end{subfigure}
    \hfill
    \begin{subfigure}[t]{0.48\textwidth}
        \centering
        \includegraphics[width=\textwidth]{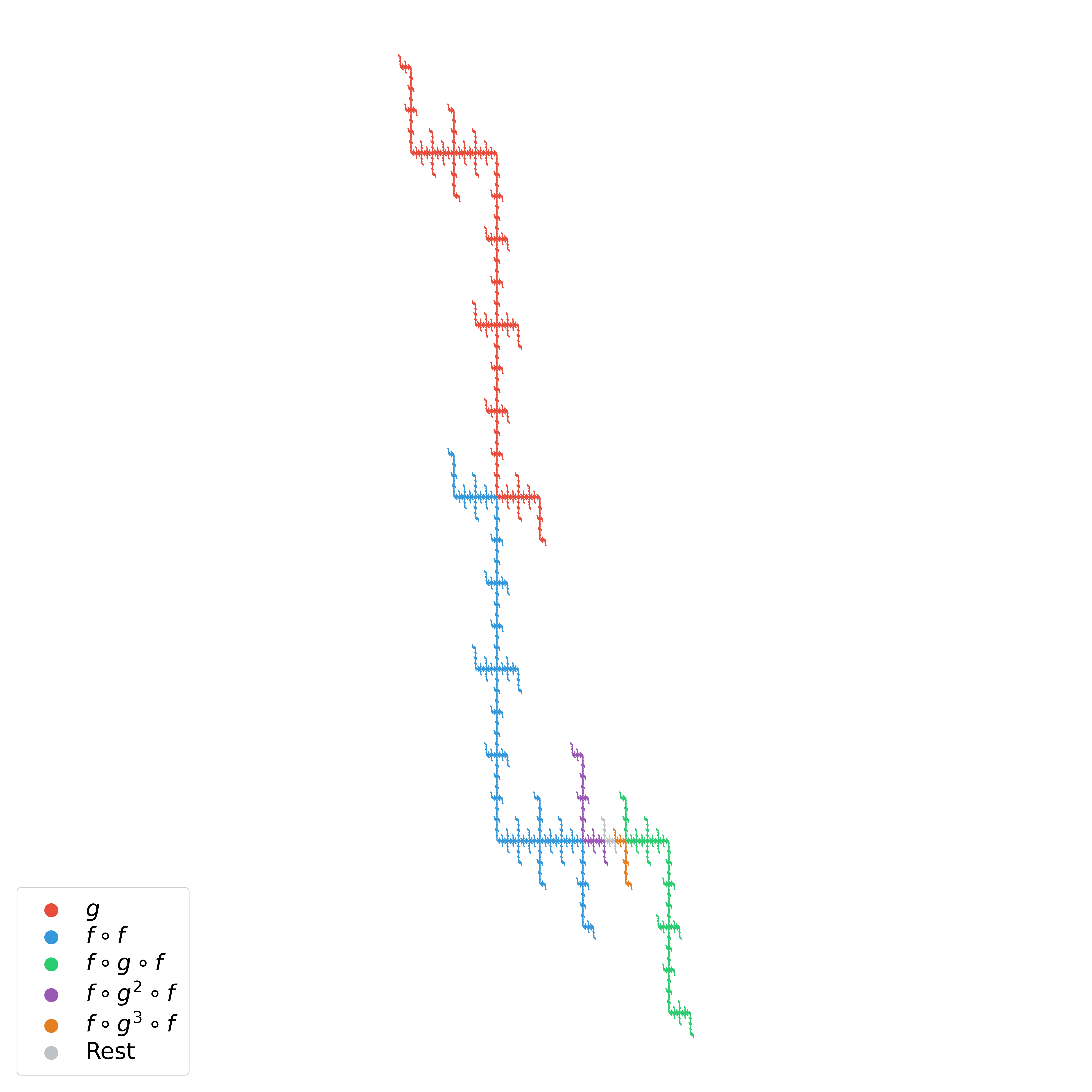}
        \caption{Decomposition of $A_3$ into self-similar copies.}
    \end{subfigure}
   \caption[Structure of the attractor $A_3$ from Example \ref{ex:mixed_scaling_fractal}]{The attractor $A_3$ of the IFS $\{f_3, g_3\}$ from Example \ref{ex:mixed_scaling_fractal}. The right panel (b) identifies the components $g_3(A_3)$ and $f_3^2(A_3)$, along with various smaller self-similar copies that constitute the attractor.}
    \label{fig:example3}
\end{figure}

The map $f_3$ exhibits significant non-uniform scaling, expanding in one coordinate while contracting in the other. However, the second iterate yields $f_3^2(x,y) = (x/2, y/2)$, which is a uniform similarity with ratio $c = 1/2$. The map $g_3$ is a standard similarity with ratio $r = 1/2$. As the linear parts of $g_3$ and $f_3^\top f_3$ are both diagonal, they commute, satisfying the $f_3$-alignment condition.

Substituting these parameters into the dimension formula (with $n=2$) gives:
\[
c^{s/2} + r^s = 1 \implies \left(\frac{1}{2}\right)^{s/2} + \left(\frac{1}{2}\right)^s = 1.
\]
This equation is identical to the one derived in Example \ref{ex:golden_ratio_fractal}. Consequently, the Hausdorff dimension is the same:
\[
s = 2\log_2(\varphi) \approx 1.388.
\]
\end{example}

\begin{example} \label{ex:example_four_fractal}
Consider the affine IFS on $\mathbb{R}^2$ generated by the maps:
\begin{align*}
    f_4(x, y) &= \left( -y, \; \frac{x}{3} \right), \\
    g_{4a}(x, y) &= \left( -\frac{x}{2} + 10, \; \frac{y}{2} - 9 \right), \\
    g_{4b}(x, y) &= \left( -\frac{x}{2} - 10, \; \frac{y}{2} + 9 \right).
\end{align*}
The attractor $A_4$ generated by this system is displayed in Figure \ref{fig:example_four_fractal}.

\begin{figure}[htbp]
    \centering
    \includegraphics[width=0.65\textwidth]{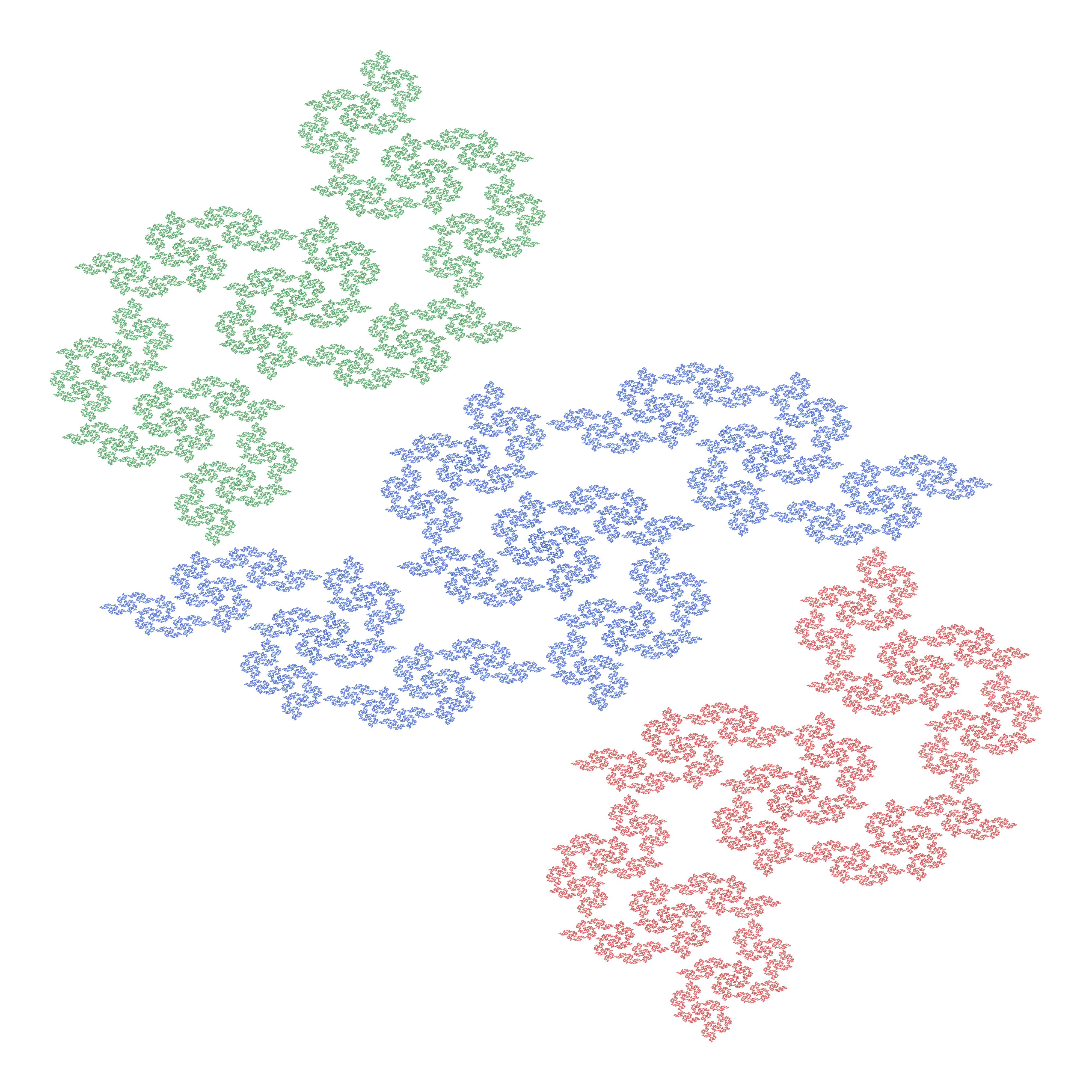}
    \caption{The attractor $A_4$ of the IFS $\{f_4, g_{4a}, g_{4b}\}$ from Example \ref{ex:example_four_fractal}. The colors correspond to the disjoint images of the attractor under the generating maps: $f_4(A_4)$ is shown in blue, $g_{4a}(A_4)$ in red, and $g_{4b}(A_4)$ in green.}
    \label{fig:example_four_fractal}
\end{figure}

We verify the hypotheses of Theorem \ref{thm:general_dimension}:
\begin{enumerate}
    \item The second iterate of $f_4$ is $f_4^2(x, y) = (-x/3, -y/3)$, which is a similarity with ratio $c = 1/3$. Thus, $n=2$.
    \item The maps $g_{4a}$ and $g_{4b}$ are similarities with ratio $r = 1/2$, and both are $f_4$-aligned.
\end{enumerate}

Consequently, the Hausdorff dimension $s$ satisfies the equation:
\[
\left(\frac{1}{3}\right)^{s/2} + 2\left(\frac{1}{2}\right)^s = 1.
\]
Solving numerically yields $s \approx 1.713$.

\end{example}

\section{Example: Importance of the Alignment Condition}

The $f$-aligned condition in Theorem \ref{thm:general_dimension} is not merely a simplifying assumption but an essential geometric constraint. If the similarity maps introduce rotations that misalign the principal axes of the affine map $f$, the standard dimension formula fails, even when the Open Set Condition (OSC) is rigorously satisfied. In this section, we construct an explicit example to illustrate this failure, where the formula strictly underestimates the actual Hausdorff dimension.

\begin{example} \label{ex:counter_example}
Consider an IFS on $\mathbb{R}^2$ consisting of one affine map $f$ and two similarity maps $g$ and $h$:
\begin{align*}
    f(x, y) &= \left( -\frac{y}{50}, \; 2x \right), \\
    g(x, y) &= \left( -\frac{y}{4} + 1, \; \frac{x}{4} + 11 \right), \\
    h(x, y) &= \left( -\frac{y}{4} - 1, \; \frac{x}{4} - 11 \right).
\end{align*}
\end{example}

The attractor $A$ generated by this system is visualized in Figure \ref{attractorcounter_example}.

\begin{figure}[htpb]
    \centering
    \includegraphics[width=0.6\textwidth]{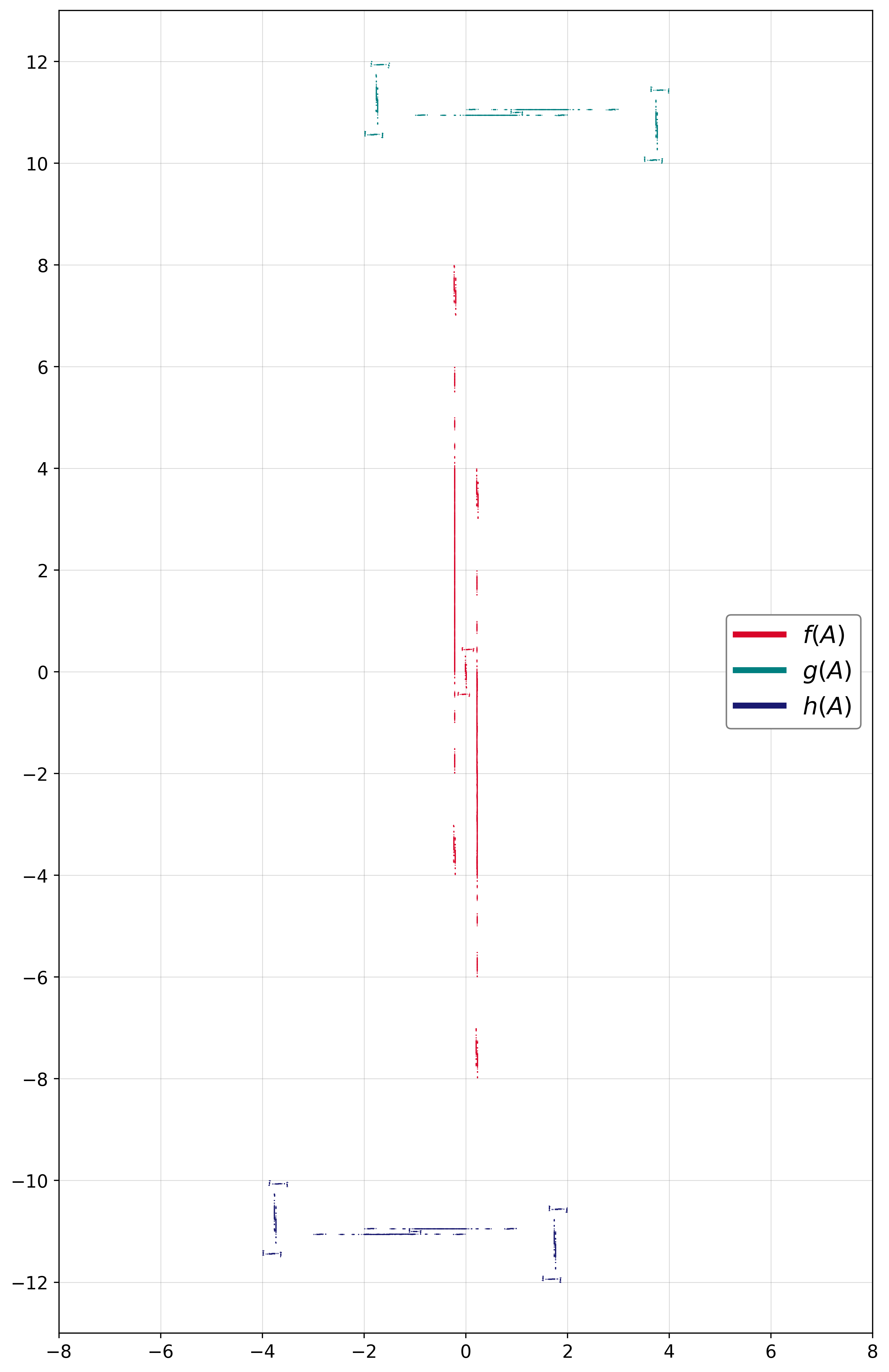}
    \caption[Counter-example Attractor]{The attractor of the IFS $\{f,g,h\}$. The structure contains a subset whose vertical projection is the interval $[-4, 4]$, demonstrating that $\dim_H(A) \ge 1$.}
    \label{attractorcounter_example}
\end{figure}

\subsection{Verification of Parameters and OSC}

\begin{enumerate}
    \item $G^2$-similarity of $f$: The map $f$ is not a similarity, as it contracts the $y$-axis by $1/50$ and stretches the $x$-axis by $2$. However, its second iterate is a uniform contraction:
    \[ f^2(x,y) = f\left(-\frac{y}{50}, 2x\right) = \left( -\frac{2x}{50}, 2\left(-\frac{y}{50}\right) \right) = \left( -\frac{x}{25}, -\frac{y}{25} \right). \]
    Thus, $f$ acts as a $G^2$-similarity with ratio $c = 1/25$.
    
    \item Similarities $g$ and $h$: Both $g$ and $h$ are similarities involving a rotation of $90^\circ$ and a scaling factor of $r = 1/4$.
    
    \item Open Set Condition: This system satisfies the OSC with the open rectangle $U = (-4, 4) \times (-12, 12)$. The images $f(U)$, $g(U)$, and $h(U)$ are contained within $U$ and are mutually disjoint. Specifically, the projection of the images onto the $y$-axis yields the disjoint intervals $(-8, 8)$ for $f(U)$, $(10, 12)$ for $g(U)$, and $(-12, -10)$ for $h(U)$.
    
    \item Non-alignment: The principal axes of $f$ are the expanding $x$-axis and the contracting $y$-axis. However, $g$ and $h$ rotate the coordinate system by $90^\circ$, mapping the $x$-axis to the $y$-axis. This explicitly mixes the expanding and contracting directions of $f$, violating the $f$-alignment condition.
\end{enumerate}

\subsection{Discrepancy with the Dimension Formula}
If the alignment condition were ignored and the formula from Theorem \ref{thm:general_dimension} were applied, the Hausdorff dimension $s$ would satisfy:
\[ c^{s/2} + 2r^s = 1 \]
Substituting $c = 1/25$ and $r = 1/4$:
\[ \left(\frac{1}{25}\right)^{s/2} + 2\left(\frac{1}{4}\right)^s = 1 \implies \left(\frac{1}{5}\right)^s + 2\left(\frac{1}{4}\right)^s = 1. \]
Numerical solution of this equation yields $s \approx 0.754$. This result suggests that the attractor is totally disconnected and dust-like (dimension strictly less than 1).

\subsection{Proof that \texorpdfstring{$\dim_H(A) \ge 1$}{dimH(A) >= 1}}
To rigorously establish that this result is incorrect, we consider the subsystem $\mathcal{F}'$ generated by the composite maps $\phi_1 = f \circ g$ and $\phi_2 = f \circ h$. Since $A$ is invariant under the original IFS, it must contain the attractor $A'$ of the subsystem, implying $\dim_H(A) \ge \dim_H(A')$.

Computing the action of these compositions reveals a decoupling of coordinates:
\begin{align}
    \phi_1(x, y) &= \left( -\frac{x}{200} - \frac{11}{50}, \; -\frac{y}{2} + 2 \right), \\
    \phi_2(x, y) &= \left( -\frac{x}{200} + \frac{11}{50}, \; -\frac{y}{2} - 2 \right).
\end{align}
Note that the $y$-coordinate of the image depends solely on the $y$-coordinate of the pre-image. Thus, the orthogonal projection $P_y(A')$ onto the $y$-axis is the attractor of the 1D IFS defined by:
\[ \psi_1(y) = -\frac{y}{2} + 2, \quad \psi_2(y) = -\frac{y}{2} - 2. \]
The interval $J = [-4, 4]$ is invariant under this 1D system, as $\psi_1(J) = [0, 4]$ and $\psi_2(J) = [-4, 0]$, whose union is $J$. By the uniqueness of the attractor guaranteed by Hutchinson's Theorem, $P_y(A') = [-4, 4]$. Since the dimension of a projection cannot exceed the dimension of the set itself:
\[ 1 = \dim_H(P_y(A')) \le \dim_H(A') \le \dim_H(A). \]
This result ($\dim_H(A) \ge 1$) strictly contradicts the value $s \approx 0.754$ obtained from the formula, proving that the alignment condition is indispensable for the validity of the dimension formula.

\section{Dimension of Affine Systems with \texorpdfstring{$k$}{k}-Iterate Similarity}

In this section, we establish the general formula for the Hausdorff dimension of attractors where the non-uniform nature of the individual maps is resolved at a higher iteration level. This result generalizes the Moran-Hutchinson formula to systems that are not similarity-based at the first order but become similarities under $k$-fold compositions.

\begin{theorem}\label{thm:general_k_dim}
Let $\{f_1, f_2, \dots, f_n\}$ be a set of affine maps on $\mathbb{R}^m$ satisfying the Open Set Condition (OSC). Let $k \ge 1$ be a fixed integer. Suppose that every composition of length $k$ chosen from the set $\{f_1, \dots, f_n\}$ is a similarity contraction.

Then there exists a unique non-empty compact invariant set $A$ (the attractor) satisfying $A = \bigcup_{i=1}^n f_i(A)$. Let $c_i$ denote the linear contraction ratio of the $k$-fold diagonal composition $f_i^k = f_i \circ \dots \circ f_i$. The Hausdorff dimension $s = \dim_H(A)$ is the unique solution to:
\begin{equation}
    \sum_{i=1}^n c_i^{s/k} = 1.
\end{equation}
\end{theorem}

\begin{proof}
Let $\mathcal{G}$ be the iterated function system consisting of all $n^k$ compositions of length $k$. By hypothesis, every map in $\mathcal{G}$ is a similarity contraction. Thus, by Hutchinson's Theorem, $\mathcal{G}$ admits a unique attractor $A$. It is a standard result that this set $A$ is also the unique invariant set for the original family $\{f_i\}_{i=1}^n$.

Since the original system $\{f_i\}_{i=1}^n$ satisfies the OSC with some open set $U$, the iterated system $\mathcal{G}$ also satisfies the OSC with the same set $U$. Specifically, for any distinct multi-indices $\mathbf{i}, \mathbf{j} \in \{1, \dots, n\}^k$, the images of $U$ under the corresponding maps in $\mathcal{G}$ are disjoint subsets of $U$.

Let $g_{\mathbf{j}} = f_{j_1} \circ f_{j_2} \circ \dots \circ f_{j_k}$ be an arbitrary element of $\mathcal{G}$. Since $\mathcal{G}$ is a system of similarities satisfying the OSC, the Hausdorff dimension $s$ is given by the Moran equation:
\begin{equation} \label{eq:moran_sum}
    \sum_{\mathbf{j} \in \{1,\dots,n\}^k} r_{\mathbf{j}}^s = 1,
\end{equation}
where $r_{\mathbf{j}}$ denotes the contraction ratio (Lipschitz constant) of the map $g_{\mathbf{j}}$.

In $\mathbb{R}^m$, the contraction ratio $r$ of a similarity map $S$ is related to its determinant by $r = |\det(S)|^{1/m}$. We first determine the determinant of the individual maps $f_i$. The diagonal composition $f_i^k$ is a similarity with ratio $c_i$, so:
\[ |\det(f_i^k)| = c_i^m \implies |\det(f_i)|^k = c_i^m \implies |\det(f_i)|^{1/m} = c_i^{1/k}. \]

Now, consider the composition $g_{\mathbf{j}} = f_{j_1} \circ \dots \circ f_{j_k}$. Its contraction ratio $r_{\mathbf{j}}$ is given by:
\begin{align*}
    r_{\mathbf{j}} &= |\det(f_{j_1} \circ \dots \circ f_{j_k})|^{1/m} \\
    &= \left( \prod_{p=1}^k |\det(f_{j_p})| \right)^{1/m} \\
    &= \prod_{p=1}^k |\det(f_{j_p})|^{1/m}.
\end{align*}
Substituting the relation $|\det(f_i)|^{1/m} = c_i^{1/k}$, this simplifies to:
\[ r_{\mathbf{j}} = \prod_{p=1}^k c_{j_p}^{1/k} = c_{j_1}^{1/k} c_{j_2}^{1/k} \dots c_{j_k}^{1/k}. \]

Substituting the expression for $r_{\mathbf{j}}$ into the Moran equation (\ref{eq:moran_sum}) and distributing the exponent $s$, we obtain:
\begin{equation}
    \sum_{j_1=1}^n \dots \sum_{j_k=1}^n \left( c_{j_1}^{s/k} c_{j_2}^{s/k} \dots c_{j_k}^{s/k} \right) = 1.
\end{equation}
Since the indices are independent, the multiple summation factorizes into the product of $k$ identical sums:
\[ \left( \sum_{i=1}^n c_i^{s/k} \right)^k = 1. \]
Taking the $k$-th root yields the desired result:
\[ \sum_{i=1}^n c_i^{s/k} = 1. \]
\end{proof}

We now extend the previous result, which assumed the Open Set Condition, to cases where the images $f_i(A)$ overlap in a structured manner.

\begin{theorem}\label{thm:gen_overlap_Rm}
Let $\{f_1, \dots, f_n\}$ be a finite set of affine maps on $\mathbb{R}^m$. Assume that for a fixed integer $k \ge 1$, every composition of length $k$ generated by the system is a similarity contraction. Let $A$ denote the unique non-empty compact attractor of the system, and let $c_i$ denote the contraction ratio of the $k$-fold diagonal composition $f_i^k = f_i \circ \dots \circ f_i$.

Suppose that the Open Set Condition is not satisfied. Specifically, assume that the union of the pairwise intersections of the images $f_i(A)$, given by $\bigcup_{1 \le i < l \le n} (f_i(A) \cap f_l(A))$, can be decomposed into $M$ disjoint sets $\{O_j\}_{j=1}^M$.

For each disjoint set $O_j$, let $I_j \subset \{1, \dots, n\}$ be the set of indices such that $O_j \subseteq f_i(A)$ for all $i \in I_j$, and $O_j \cap f_i(A) = \emptyset$ for all $i \notin I_j$. Let $q_j = |I_j|$ denote the multiplicity of the overlap $O_j$.

Furthermore, assume that each $O_j$ is homothetic to the attractor $A$. That is, each $O_j$ is the image of $A$ under a translation and a scaling by a scalar $\lambda_j \in \mathbb{R} \setminus \{0\}$, with magnitude $p_j = |\lambda_j| \in (0, 1)$. Then, the Hausdorff dimension $s = \dim_H(A)$ is the unique solution to the equation:
\begin{equation}
    \sum_{i=1}^n c_i^{s/k} - \sum_{j=1}^M (q_j - 1) p_j^s = 1.
\end{equation}
\end{theorem}

\begin{proof}
Since the system is eventually contractive (at iteration $k$), there exists a unique non-empty compact attractor $A$. Let $\mathcal{H}^s(A)$ denote its $s$-dimensional Hausdorff measure. We assume that $0 < \mathcal{H}^s(A) < \infty$.

The attractor satisfies $A = \bigcup_{i=1}^n f_i(A)$. We consider the sum of the measures of the images $f_i(A)$. This sum overcounts the regions where overlaps occur. Specifically, for any point $x$ in an overlap set $O_j$, the sum counts the measure $q_j$ times (once for each index $i \in I_j$), whereas the actual measure of the union counts it only once.

To balance the equation, we must subtract the excess multiplicity $(q_j - 1)$ for each overlap set:
\[ \mathcal{H}^s(A) = \sum_{i=1}^n \mathcal{H}^s(f_i(A)) - \sum_{j=1}^M (q_j - 1) \mathcal{H}^s(O_j). \]
Since each $O_j$ is homothetic to $A$ with scaling factor $p_j$, we have $\mathcal{H}^s(O_j) = p_j^s \mathcal{H}^s(A)$. Substituting this yields the measure balance equation for the first iteration:
\begin{equation} \label{eq:mult_level1}
    \sum_{i=1}^n \mathcal{H}^s(f_i(A)) = \mathcal{H}^s(A) \left( 1 + \sum_{j=1}^M (q_j - 1) p_j^s \right).
\end{equation}

We now extend this relation to the $k$-th iteration. Let $\mathcal{J}_k = \{1, \dots, n\}^k$ be the set of all length-$k$ indices. For any multi-index $\mathbf{j} = (j_1, \dots, j_k) \in \mathcal{J}_k$, let $g_{\mathbf{j}} = f_{j_1} \circ \dots \circ f_{j_k}$ denote the composed map.

We rely on the property that affine maps preserve the ratio of measures for homothetic sets. The linear part $L$ of any affine map commutes with scalar multiplication, i.e., $L(\lambda v) = \lambda L(v)$. Since $O_j$ is homothetic to $A$ (scaled by $\lambda_j$), the affine image $f(O_j)$ is homothetic to $f(A)$ with the same scaling factor $\lambda_j$. The affine transformation scales both sets identically, preserving the measure ratio:
\[ \frac{\mathcal{H}^s(f(O_j))}{\mathcal{H}^s(f(A))} = |\lambda_j|^s = p_j^s. \]
Applying this invariance recursively $k$ times to (\ref{eq:mult_level1}) yields:
\begin{equation} \label{eq:mult_levelk_bal}
    \sum_{\mathbf{j} \in \mathcal{J}_k} \mathcal{H}^s(g_{\mathbf{j}}(A)) = \mathcal{H}^s(A) \left( 1 + \sum_{j=1}^M (q_j - 1) p_j^s \right)^k.
\end{equation}

Conversely, utilizing the hypothesis that every $k$-length composition is a similarity contraction, Theorem \ref{thm:general_k_dim} implies the total measure scales as:
\begin{equation} \label{eq:mult_levelk_sim}
    \sum_{\mathbf{j} \in \mathcal{J}_k} \mathcal{H}^s(g_{\mathbf{j}}(A)) = \mathcal{H}^s(A) \left( \sum_{i=1}^n c_i^{s/k} \right)^k.
\end{equation}
Equating (\ref{eq:mult_levelk_bal}) and (\ref{eq:mult_levelk_sim}) and dividing by $\mathcal{H}^s(A)$, we obtain:
\[ \left( 1 + \sum_{j=1}^M (q_j - 1) p_j^s \right)^k = \left( \sum_{i=1}^n c_i^{s/k} \right)^k. \]
Taking the $k$-th root and rearranging terms yields the final result:
\[ \sum_{i=1}^n c_i^{s/k} - \sum_{j=1}^M (q_j - 1) p_j^s = 1. \]
\end{proof}

To illustrate the practical application of these formulae, we now examine specific affine systems where the $k$-iterate similarity property holds, explicitly calculating the dimension for cases with and without overlaps.

\begin{example}\label{ex:k2_sim}
Consider the Iterated Function System on $\mathbb{R}^2$ defined by the two affine maps:
\begin{align*}
    f(x, y) &= \left( -\frac{y}{6}, x \right), \\
    g(x, y) &= \left( \frac{y}{3} + 1, 2x + 1 \right).
\end{align*}
The linear parts of these maps are given by the matrices:
\[
    M_f = \begin{pmatrix} 0 & -1/6 \\ 1 & 0 \end{pmatrix}, \quad 
    M_g = \begin{pmatrix} 0 & 1/3 \\ 2 & 0 \end{pmatrix}.
\]
Neither $M_f$ nor $M_g$ is a similarity. However, we observe that for $k=2$, every possible composition is a similarity transformation:
\begin{itemize}
    \item The diagonal maps $f^2$ and $g^2$ are uniform similarities with contraction ratios $c_1 = 1/6$ and $c_2 = 2/3$, respectively.
    \item The mixed compositions $f \circ g$ and $g \circ f$ are scaled reflections with contraction ratio $r = 1/3$.
\end{itemize}
Since all compositions of length $k=2$ are similarities and the Open Set Condition is satisfied (as seen in Figure \ref{fig:k2_attractor}), we apply Theorem \ref{thm:general_k_dim}:
\[
    \sum_{i=1}^n c_i^{s/k} = 1 \implies \left(\frac{1}{6}\right)^{s/2} + \left(\frac{2}{3}\right)^{s/2} = 1.
\]
Solving this equation yields the Hausdorff dimension $s \approx 1.496$.
\begin{figure}[htpb]
    \centering
    \includegraphics[width=0.6\textwidth]{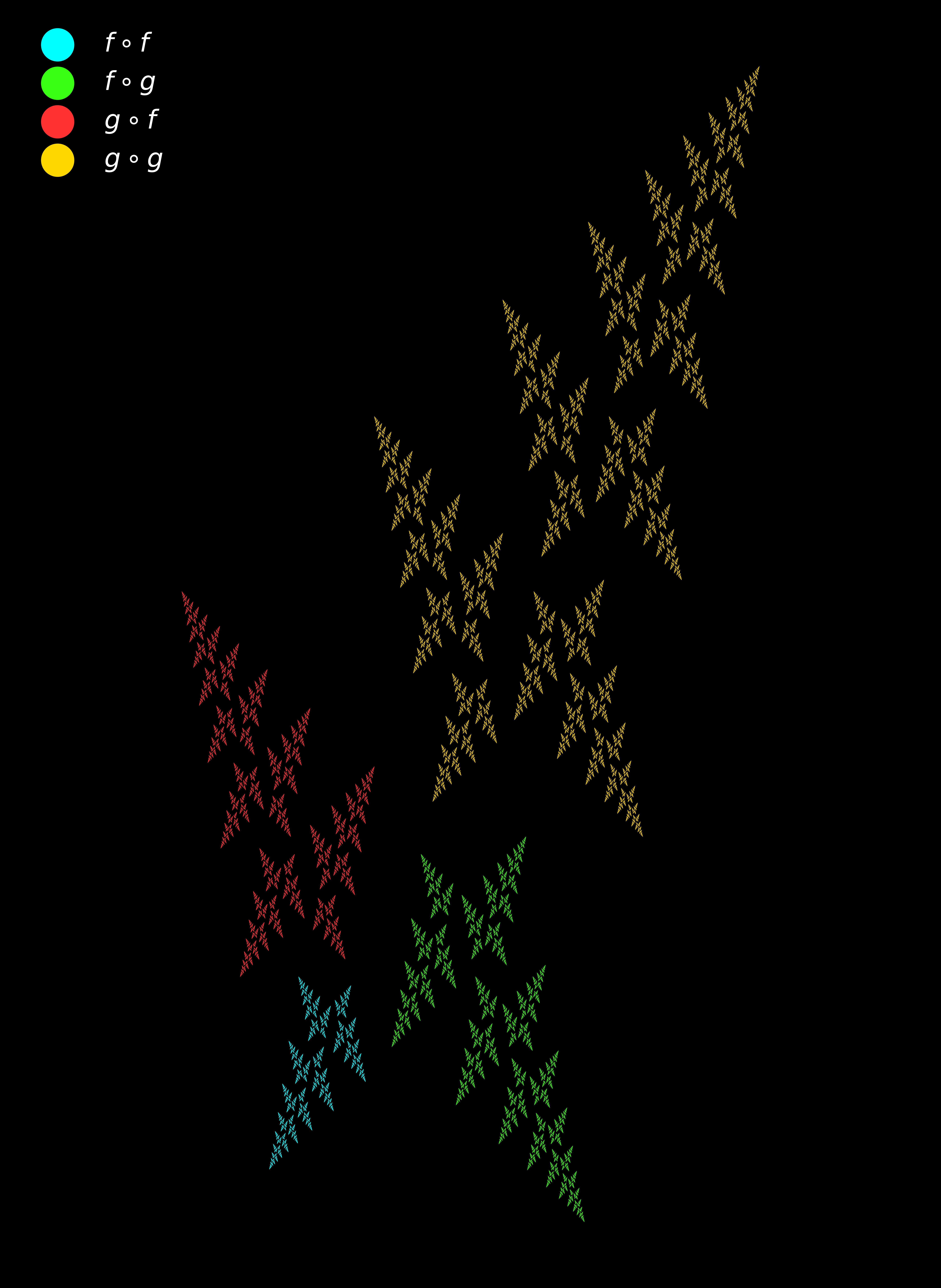}
    \caption{The attractor of the system $\{f,g\}$. The coloring distinguishes the four second-level images ($f^2, fg, gf, g^2$), demonstrating the separation of components required by the Open Set Condition.}
    \label{fig:k2_attractor}
\end{figure}
\end{example}

\begin{example}\label{ex:struct_overlap}

Consider the system defined by the four affine maps on $\mathbb{R}^2$:
\begin{align*}
    f_1(x, y) &= (y/4 - 1, -x), \\
    g_1(x, y) &= (-y/4 + 1, x), \\
    h_1(x, y) &= (-y/4, x + 2), \\
    k_1(x, y) &= (y/4, -x + 2).
\end{align*}
The linear parts of these maps are orthogonal matrices scaled by factors of $1$ and $1/4$. While the individual maps are not similarities, the system satisfies the $k$-iterate similarity condition for $k=2$. Specifically, any composition of length two from the set $\{f_1, g_1, h_1, k_1\}$ is a similarity contraction with ratio $c = 1/4$.

This system does not satisfy the Open Set Condition. As shown in Figure \ref{fig:overlap_attractor}, the images $h_1(A)$ and $k_1(A)$ intersect. The visual decomposition reveals that the overlap region $h_1(A) \cap k_1(A)$ consists of two distinct copies of the attractor $A$, each scaled by a factor $p = 1/4$.

We apply Theorem \ref{thm:gen_overlap_Rm} with $n=4$ maps, iteration depth $k=2$, contraction ratios $c_i = 1/4$, and $M=2$ overlap terms with scaling factors $p_1 = p_2 = 1/4$. The dimension $s$ is the unique solution to:
\begin{equation}
    \sum_{i=1}^4 c_i^{s/k} - \sum_{j=1}^2 p_j^s = 1
\end{equation}
Substituting the values:
\[
    4 \left(\frac{1}{4}\right)^{s/2} - 2 \left(\frac{1}{4}\right)^s = 1
\]
Simplifying with $x = (1/2)^s$:
\[
    4x - 2x^2 = 1 \implies 2x^2 - 4x + 1 = 0
\]
Solving the quadratic equation for $x < 1$:
\[
    x = \frac{4 - \sqrt{16 - 8}}{4} = \frac{4 - 2\sqrt{2}}{4} = 1 - \frac{\sqrt{2}}{2} \approx 0.2929
\]
The Hausdorff dimension is therefore:
\[
    s = \frac{\ln(1 - \frac{\sqrt{2}}{2})}{\ln(1/2)} \approx 1.771.
\]

\begin{figure}[htpb]
    \centering
    \includegraphics[width=0.7\textwidth]{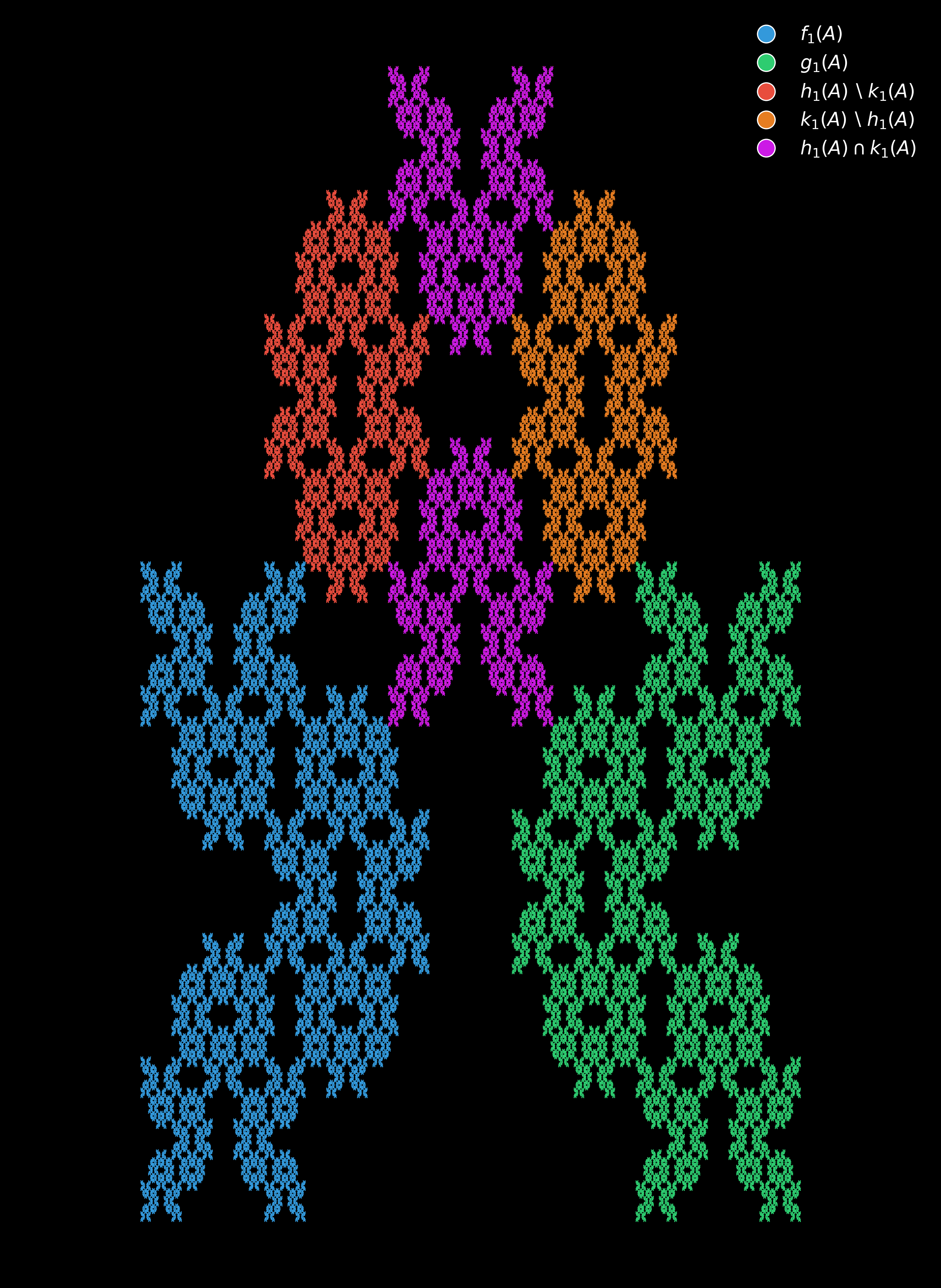}
    \caption{The attractor of the overlapping system $\{f_1, g_1, h_1, k_1\}$. The Open Set Condition fails as $h_1(A)$ and $k_1(A)$ intersect. The plot decomposes the set into components: Blue ($f_1(A)$), Green ($g_1(A)$), Red ($h_1(A) \setminus k_1(A)$), Orange ($k_1(A) \setminus h_1(A)$), and the overlap region in Purple ($h_1(A) \cap k_1(A)$).}
    \label{fig:overlap_attractor}
\end{figure}

\end{example}

\begin{example}\label{ex:mixed_overlap}

\begin{figure}[htpb]
    \centering
    \includegraphics[width=0.95\textwidth]{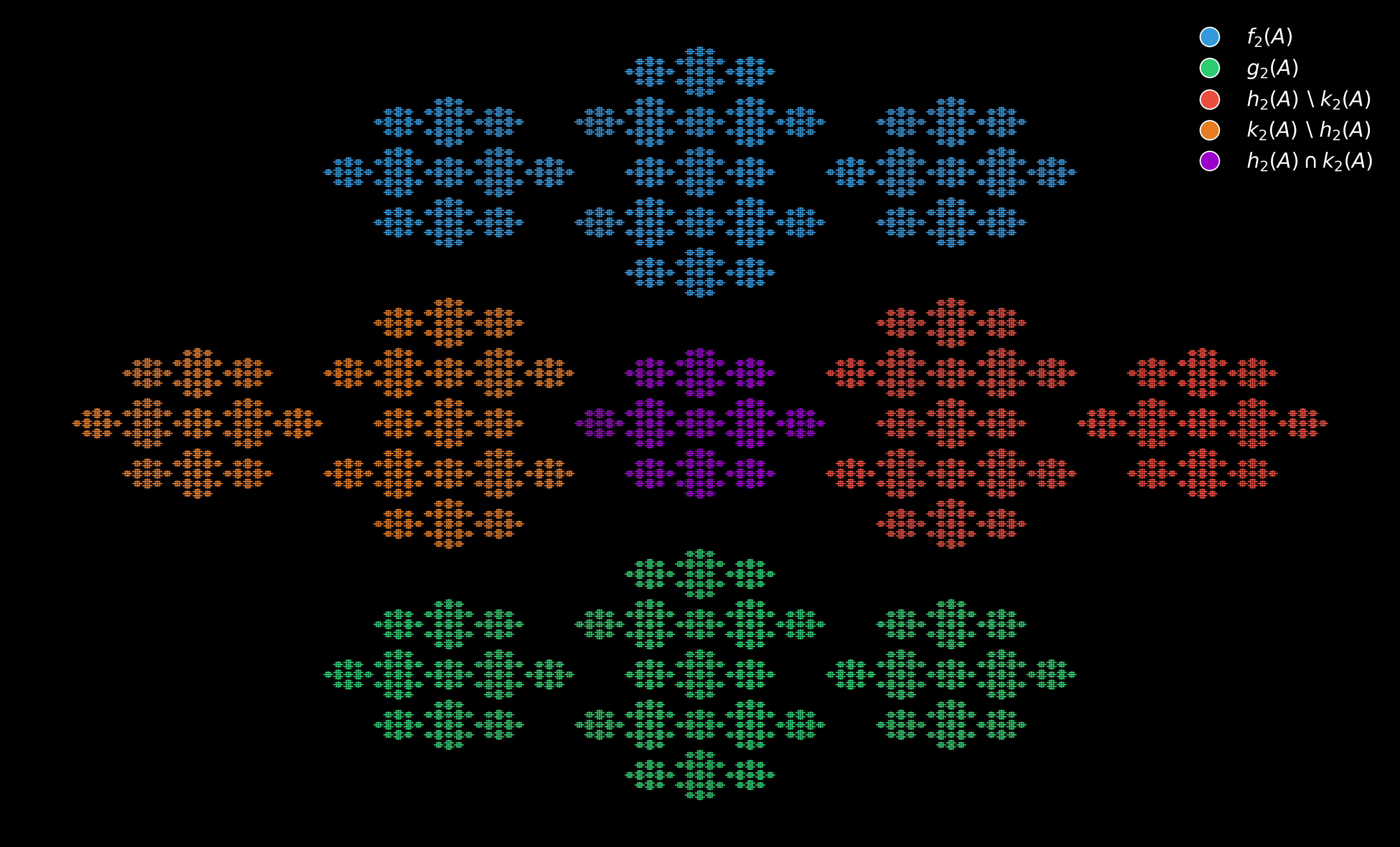}
    \caption{The attractor of the system $\{f_2, g_2, h_2, k_2\}$. The images $f_2(A)$ (Blue) and $g_2(A)$ (Green) are strictly disjoint from the rest of the set. The only overlap occurs between $h_2(A)$ (Red) and $k_2(A)$ (Orange) near the origin. The intersection region $h_2(A) \cap k_2(A)$ is highlighted in purple.}
    \label{fig:mixed_overlap}
\end{figure}

Consider the system defined by the following four affine maps on $\mathbb{R}^2$:
\begin{align*}
    f_2(x, y) &= (-y, x/5 + 1), \\
    g_2(x, y) &= (-y, x/5 - 1), \\
    h_2(x, y) &= (-y + 1, x/5), \\
    k_2(x, y) &= (-y - 1, x/5).
\end{align*}
Similar to the previous examples, the linear parts of these maps are not similarities. However, for $k=2$, the system exhibits similarity behavior. For any map $\phi \in \{f_2, g_2, h_2, k_2\}$, the composition $\phi^2$ is a similarity with contraction ratio $c = 1/5$.

The attractor is shown in Figure \ref{fig:mixed_overlap}. The visual decomposition confirms that the images $f_2(A)$ and $g_2(A)$ are completely isolated components.

The only violation of the Open Set Condition occurs between the images $h_2(A)$ and $k_2(A)$, which intersect in the central region. This intersection $h_2(A) \cap k_2(A)$ forms a self-similar copy of the attractor scaled by the factor $p = 1/5$.

Since $f_2(A)$ and $g_2(A)$ are isolated, they contribute no intersection terms. We apply Theorem \ref{thm:gen_overlap_Rm} with $n=4$ maps, iteration depth $k=2$, contraction ratios $c_i = 1/5$, and exactly $M=1$ overlap term with scaling factor $p_1 = 1/5$. The Hausdorff dimension $s$ satisfies:
\begin{equation}
    \sum_{i=1}^4 c_i^{s/k} - p_1^s = 1.
\end{equation}
Substituting the values:
\[
    4 \left(\frac{1}{5}\right)^{s/2} - \left(\frac{1}{5}\right)^s = 1.
\]
Let $x = (1/5)^{s/2}$. The equation becomes:
\[
    4x - x^2 = 1 \implies x^2 - 4x + 1 = 0.
\]
Solving for $x < 1$:
\[
    x = \frac{4 - \sqrt{16 - 4}}{2} = 2 - \sqrt{3}.
\]
Substituting back $x = (5^{-1/2})^s$:
\[
    5^{-s/2} = 2 - \sqrt{3} \implies -\frac{s}{2} \ln 5 = \ln(2 - \sqrt{3}).
\]
\[
    s = \frac{-2 \ln(2 - \sqrt{3})}{\ln 5} \approx 1.6365.
\]

\end{example}

\begin{remark}
One might argue that for the non-overlapping case in Example \ref{ex:k2_sim}, Theorem \ref{thm:general_k_dim} is not strictly necessary, as the dimension can be obtained by applying the standard Moran equation to the system $\mathcal{G} = \{f^2, fg, gf, g^2\}$, which consists of four similarity contractions. 

However, the utility of our approach becomes evident in the overlapping cases of Example \ref{ex:struct_overlap} and Example \ref{ex:mixed_overlap}. Although these attractors are also generated by the $16$ composite similarity mappings of level $k=2$, standard self-similar theory cannot be applied effectively. The Open Set Condition fails for these induced systems, rendering the Moran equation invalid. Moreover, applying the inclusion-exclusion principle to the $16$ similarity maps is intractable, as the overlaps at that level are not self-similar copies of the attractor.

In this context, Theorem \ref{thm:gen_overlap_Rm} is essential: it allows us to bypass the complexity of the similarity system by exploiting the simpler overlap structure available at the affine level ($k=1$), where the intersections are precise scaled copies of the attractor.
\end{remark}

\section{Generalization: Hybrid Systems}

We now unify the results of Theorem \ref{thm:general_dimension} and Theorem \ref{thm:general_k_dim} into a single framework. We consider a "hybrid" system consisting of a set of affine maps $\mathcal{F}$ (which resolve to similarities after $n$ iterations) and a set of standard similarity maps $\mathcal{G}$. For the dimension formula to hold, the geometric action of the similarities must be compatible with every affine map in the system.

\begin{definition}
Let $\mathcal{F} = \{f_1, \dots, f_N\}$ be a finite set of affine maps on $\mathbb{R}^m$. A similarity transformation $g: \mathbb{R}^m \to \mathbb{R}^m$ is said to be universally aligned with $\mathcal{F}$ if $g$ is $f_i$-aligned for every $i \in \{1, \dots, N\}$. Algebraically, if $S$ is the linear part of $g$ and $A_i$ is the linear part of $f_i$, then $S$ must commute with the symmetric matrix of every affine map:
\begin{equation}
    S (A_i^\top A_i) = (A_i^\top A_i) S \quad \text{for all } i = 1, \dots, N.
\end{equation}
\end{definition}

\begin{theorem} \label{thm:hybrid_existence}
Let $\mathcal{F} = \{f_1, \dots, f_q\}$ be a finite set of affine maps and $\mathcal{G} = \{g_1, \dots, g_p\}$ be a finite set of similarity contractions on $\mathbb{R}^m$. Assume the following conditions:
\begin{enumerate}
    \item Uniform $G^n$-Similarity: There exists an integer $n \ge 1$ such that for every sequence of indices $i_1, \dots, i_n \in \{1, \dots, q\}$, the composite map $f_{i_1} \circ \dots \circ f_{i_n}$ is a similarity contraction.
    \item Universal Alignment: Every map in $\mathcal{G}$ is universally aligned with $\mathcal{F}$.
\end{enumerate}
Then the Iterated Function System $\mathcal{H} = \mathcal{F} \cup \mathcal{G}$ admits a unique attractor.
\end{theorem}

\begin{proof}
We establish that the system $\mathcal{H}$ is eventually contractive. Let $w$ be an arbitrary composite map of length $L$ generated by the IFS $\mathcal{H}$. We denote the Lipschitz constant of a map $\phi$ by $\text{Lip}(\phi)$.

Let $m$ be the number of occurrences of maps from $\mathcal{G}$ in the composition $w$, and let $k = L-m$ be the number of occurrences of maps from $\mathcal{F}$. Let $\{r_{j_1}, \dots, r_{j_m}\}$ be the contraction ratios of the specific maps from $\mathcal{G}$ appearing in $w$.

The Universal Alignment condition implies that for any affine map $\phi \in \mathcal{F}$ and any similarity $g \in \mathcal{G}$ with contraction ratio $r$, the linear distortion satisfies:
\[ \| \phi(g(v)) - \phi(g(u)) \| = r \| \phi(v) - \phi(u) \| \quad \text{for all } u, v \in \mathbb{R}^m. \]
This allows us to factor the scalar contraction ratios of the maps from $\mathcal{G}$ out of the composition norm, regardless of their position in the sequence:
\[ \text{Lip}(w) = \left( \prod_{t=1}^{m} r_{j_t} \right) \text{Lip}(w_{\mathcal{F}}), \]
where $w_{\mathcal{F}}$ denotes the sub-composition of length $k$ consisting solely of the maps from $\mathcal{F}$ in their original relative order.

We now derive a bound for $\text{Lip}(w_{\mathcal{F}})$. By the Uniform $G^n$-Similarity condition, any composition of $n$ maps from $\mathcal{F}$ is a similarity contraction. Let $c_{\max} < 1$ be the maximum contraction ratio of all such compositions of length $n$.
We decompose the integer $k$ as $k = qn + j$, where $0 \le j < n$. We can express $w_{\mathcal{F}}$ as the composition of $q$ blocks of length $n$ and a remainder block of length $j$. Let $M = \max \{ \text{Lip}(\psi) : \psi \in \bigcup_{i=1}^{n-1} \mathcal{F}^i \}$ be the maximum Lipschitz constant of any affine composition of length strictly less than $n$.
The Lipschitz constant satisfies:
\[ \text{Lip}(w_{\mathcal{F}}) \le M (c_{\max})^q = M (c_{\max})^{\lfloor k/n \rfloor}. \]

Substituting this back into the expression for $\text{Lip}(w)$, and letting $r_{\max} = \max \{ r_1, \dots, r_p \} < 1$ be the maximum contraction ratio in $\mathcal{G}$:
\[ \text{Lip}(w) \le (r_{\max})^m \cdot M \cdot (c_{\max})^{\lfloor k/n \rfloor}. \]

As the total length $L \to \infty$, we must have $m \to \infty$ or $k \to \infty$ (or both). In either case, $\lim_{L \to \infty} \sup_{w \in \mathcal{H}^L} \text{Lip}(w) = 0$.
Thus, the IFS is eventually contractive. By Hutchinson's Theorem, there exists a unique non-empty compact attractor $A$.
\end{proof}

\begin{theorem} \label{thm:hybrid_dimension}
Let $\mathcal{F} = \{f_1, \dots, f_q\}$ be a finite set of affine maps and $\mathcal{G} = \{g_1, \dots, g_p\}$ be a finite set of similarity contractions on $\mathbb{R}^m$. Assume the following conditions:
\begin{enumerate}
    \item Uniform $G^n$-Similarity: There exists an integer $n \ge 1$ such that every composition of length $n$ generated by $\mathcal{F}$ is a similarity contraction.
    \item Universal Alignment: Every map in $\mathcal{G}$ is universally aligned with $\mathcal{F}$.
    \item Open Set Condition: The union system $\mathcal{H} = \mathcal{F} \cup \mathcal{G}$ satisfies the OSC.
\end{enumerate}
Let $c_i$ denote the contraction ratio of the $n$-th iterate $f_i^n$, and let $r_j$ denote the contraction ratio of the similarity $g_j$.
Then, the Hausdorff dimension $s = \dim_H(A)$ of the attractor is the unique solution to:
\begin{equation} \label{eq:hybrid_formula}
    \sum_{i=1}^q c_i^{s/n} + \sum_{j=1}^p r_j^s = 1.
\end{equation}
Moreover, the $s$-dimensional Hausdorff measure satisfies $0 < \mathcal{H}^s(A) < \infty$.
\end{theorem}

\begin{proof}
First, we prove that $\mathcal{H}^s(A) < \infty$. Let $s$ be the unique solution to Equation \eqref{eq:hybrid_formula}. We assign a ``formal weight'' $w(\psi)$ to each function $\psi \in \mathcal{H} = \mathcal{F} \cup \mathcal{G}$ as follows:
\[
    w(f_i) = c_i^{1/n} \text{ for } f_i \in \mathcal{F}, \quad \text{and} \quad w(g_j) = r_j \text{ for } g_j \in \mathcal{G}.
\]
By the definition of $s$, these weights satisfy the normalization condition:
\begin{equation} \label{eq:sum_unity_hybrid}
    \sum_{\psi \in \mathcal{H}} w(\psi)^s = 1.
\end{equation}

Let $\mathcal{I}^L$ be the set of all finite words of length $L$ formed by the maps in $\mathcal{H}$. For any word $\omega = (\psi_1, \dots, \psi_L) \in \mathcal{I}^L$, we define the composite map $h_\omega = \psi_1 \circ \dots \circ \psi_L$. The set $h_\omega(A)$ represents the image of the attractor under this map. The collection $\{ h_\omega(A) : \omega \in \mathcal{I}^L \}$ forms a cover of $A$.

Consider an arbitrary word $\omega \in \mathcal{I}^L$. Using the Universal Alignment condition, we can rearrange the positions of the similarities in the composition $h_\omega$. We can move them anywhere within the sequence without changing the relative order of the affine maps, and the diameter of the image set $h_\omega(A)$ will be unaffected.

Let $L$ be sufficiently large such that we can partition the length $L = L_1 + L_2$, where $L_1$ is a positive multiple of $n$ and the tail length $L_2$ satisfies $n-1 \le L_2 \le 2(n-1)$.
We rearrange the sequence of functions composing $h_\omega$ into a ``head'' block $h^{\text{head}}$ of length $L_1$ and a ``tail'' block $h^{\text{tail}}$ of length $L_2$ such that $h^{\text{head}}$ contains exactly a multiple of $n$ affine maps.

We define the \textit{Total Formal Weight} of the word $\omega$ as $W(\omega) = \prod_{j=1}^L w(\psi_j)$.
We now bound the  diameter $\text{diam}(h_\omega(A))$ in terms of $W(\omega)$.

First, consider the head estimate. The Head contains exactly a multiple of $n$ affine maps. By the Uniform $G^n$-Similarity condition, the composition of these affine maps forms a similarity contraction. Since the other maps in the Head are standard similarities, the entire map $h^{\text{head}}$ acts as a similarity transformation whose ratio matches the product of its formal weights exactly. For any bounded set $B$:
\[ \text{diam}(h^{\text{head}}(B)) = W(\text{head}) \text{diam}(B). \]

Next, consider the tail estimate. The Tail has length $L_2$, where $n-1 \le L_2 \le 2(n-1)$. We define $\mathcal{T}$ as the set of all possible tails, i.e., all words of length between $n-1$ and $2(n-1)$ generated by the union system $\mathcal{H}$:
\[ \mathcal{T} = \bigcup_{j=n-1}^{2(n-1)} \mathcal{I}^j, \]
where $\mathcal{I}^j$ denotes the set of all words of length $j$ formed by maps in $\mathcal{H} = \mathcal{F} \cup \mathcal{G}$. We define the uniform distortion constant $K$ as:
\[ K = \sup_{\nu \in \mathcal{T}} \frac{\text{diam}(h_\nu(A))}{W(\nu) \text{diam}(A)}. \]
Since $\mathcal{T}$ is a finite set of non-singular maps, $K$ is finite. Thus:
\[ \text{diam}(h^{\text{tail}}(A)) \le K \cdot W(\text{tail}) \cdot \text{diam}(A). \]

Combining these estimates:
\[ \text{diam}(h_\omega(A)) \le \text{diam}(h^{\text{head}}(h^{\text{tail}}(A))) \le W(\text{head}) \cdot K \cdot W(\text{tail}) \cdot \text{diam}(A). \]
Since $W(\text{head}) \cdot W(\text{tail}) = W(\omega)$, we obtain the uniform bound:
\[ \text{diam}(h_\omega(A)) \le K \text{diam}(A) W(\omega). \]

With these geometric bounds established, we estimate the sum of the $s$-dimensional diameters for the cover at level $L$:
\[ \Sigma_L = \sum_{\omega \in \mathcal{I}^L} (\text{diam}(h_\omega(A)))^s \le \sum_{\omega \in \mathcal{I}^L} (K \text{diam}(A) W(\omega))^s. \]
Factoring out the constants:
\[ \Sigma_L \le K^s (\text{diam}(A))^s \sum_{\omega \in \mathcal{I}^L} W(\omega)^s. \]
The summation term is exactly the expansion of the sum of weights raised to the power $L$:
\[ \sum_{\omega \in \mathcal{I}^L} W(\omega)^s = \left( \sum_{\psi \in \mathcal{H}} w(\psi)^s \right)^L. \]
By Equation \eqref{eq:sum_unity_hybrid}, the term inside the parentheses is 1. Therefore:
\[ \Sigma_L \le K^s (\text{diam}(A))^s  < \infty. \]
Since the sum is uniformly bounded for all $L$, we conclude that $\mathcal{H}^s(A) < \infty$.

Now, we prove that $\mathcal{H}^s(A) > 0$. We define a mass distribution $\mu$ on the attractor $A$ by assigning mass to the sets $h_\omega(A)$ such that $\mu(h_\omega(A)) = (W(\omega))^s$. To ensure this assignment yields a valid measure, we verify the consistency condition: for any word $\omega$, the mass assigned to $h_\omega(A)$ must equal the sum of the masses of its images under the maps in $\mathcal{H}$. Using the relation $\sum_{\psi \in \mathcal{H}} w(\psi)^s = 1$, we have:
\[
    (W(\omega))^s = (W(\omega))^s \cdot 1 = (W(\omega))^s \left( \sum_{\psi \in \mathcal{H}} w(\psi)^s \right) = \sum_{\psi \in \mathcal{H}} (W(\omega \psi))^s.
\]
This additivity ensures that the definition is consistent across all levels of the construction. Hence, $\mu$ is a well-defined measure of total mass 1 supported on $A$.

Let $\mathcal{I}^*$ denote the set of all finite words. For any word $\omega \in \mathcal{I}^*$, let $\omega^-$ denote the word obtained by removing the last symbol.

Let $U$ be an arbitrary ball of radius $r < 1$. We define a finite cut-set $\mathcal{Q} \subset \mathcal{I}^*$ by truncating every sequence at the first index where the formal weight drops below $r$:
\[ \mathcal{Q} = \{ \omega \in \mathcal{I}^* : W(\omega) \le r < W(\omega^-) \}. \]
Let $w_{\min} = \min_{\psi \in \mathcal{H}} w(\psi)$. The stopping condition ensures that for all $\omega \in \mathcal{Q}$:
\begin{equation} \label{eq:weight_comparable_hybrid}
    w_{\min} r < W(\omega) \le r.
\end{equation}

By the Open Set Condition, there exists a non-empty bounded open set $V$ such that the images $\{ \psi(V) : \psi \in \mathcal{H} \}$ are disjoint subsets of $V$. Since $V$ is open and bounded, we can find positive real numbers $a_1$ and $a_2$ such that $V$ contains a ball of radius $a_1$ and is contained in a ball of radius $a_2$. The collection $\mathcal{V} = \{ h_\omega(V) : \omega \in \mathcal{Q} \}$ consists of disjoint open sets. Moreover, the attractor is contained in the union of their closures:
\[ A \subseteq \bigcup_{\omega \in \mathcal{Q}} \overline{h_\omega(V)}. \]
We now verify that $\mathcal{V}$ satisfies the conditions of Lemma \ref{lem:hutchinson_counting}.

Consider any $\omega \in \mathcal{Q}$ of length $L$. We decompose the length $L = L_1 + L_2$ and the map $h_\omega = h^{\text{head}} \circ h^{\text{tail}}$ as in the Upper Bound proof.
The head $h^{\text{head}}$ is a similarity with ratio $W(\text{head})$. The tail map $h^{\text{tail}}$ corresponds to a word in the finite set $\mathcal{T} = \bigcup_{j=n-1}^{2(n-1)} \mathcal{I}^j$. For each map $\psi \in \mathcal{T}$, let $\sigma_{\min}(\psi)$ and $\sigma_{\max}(\psi)$ denote the smallest and largest singular values of its linear part. We define the uniform bounds:
\[ c_{\min} = \min_{\psi \in \mathcal{T}} \sigma_{\min}(\psi), \quad c_{\max} = \max_{\psi \in \mathcal{T}} \sigma_{\max}(\psi), \quad \text{and} \quad w_{\min}^{\text{tail}} = \min_{\nu \in \mathcal{T}} W(\nu). \]
Since $\mathcal{T}$ is finite and consists of non-singular maps, $0 < c_{\min} \le c_{\max} < \infty$ and $w_{\min}^{\text{tail}} > 0$. Consequently, for any map $\psi \in \mathcal{T}$ and any ball $B$ of radius $\rho$, the image $\psi(B)$ contains a ball of radius $c_{\min} \rho$ and is contained in a ball of radius $c_{\max} \rho$.

\begin{enumerate}
    \item Inner Ball Condition:
    Since $V$ contains a ball of radius $a_1$, $h^{\text{tail}}(V)$ contains a ball of radius $c_{\min} a_1$. Consequently, $h_\omega(V)$ contains a ball of radius $W(\text{head}) c_{\min} a_1$.
    Using the bound $W(\text{head}) \ge w_{\min} r$:
    \[ W(\text{head}) c_{\min} a_1 \ge (w_{\min} c_{\min} a_1) r. \]
    Let $k_1 = w_{\min} c_{\min} a_1$. Thus, $h_\omega(V)$ contains a ball of radius $k_1 r$.

    \item Outer Ball Condition:
    Since $V$ is contained in a ball of radius $a_2$, $h^{\text{tail}}(V)$ is contained in a ball of radius $c_{\max} a_2$. Consequently, $h_\omega(V)$ is contained in a ball of radius $W(\text{head}) c_{\max} a_2$.
    Using the bound $W(\text{head}) \le r/w_{\min}^{\text{tail}}$:
    \[ W(\text{head}) c_{\max} a_2 \le \left( \frac{c_{\max} a_2}{w_{\min}^{\text{tail}}} \right) r. \]
    Let $k_2 = c_{\max} a_2 (w_{\min}^{\text{tail}})^{-1}$. Thus, $h_\omega(V)$ is contained in a ball of radius $k_2 r$.
\end{enumerate}

The collection $\mathcal{V}$ consists of disjoint open sets. Furthermore, we have established that every set $h_\omega(V) \in \mathcal{V}$ contains a ball of radius $k_1 r$ and is contained in a ball of radius $k_2 r$. Thus, the collection satisfies the hypothesis of Lemma \ref{lem:hutchinson_counting}. Consequently, the number of closures $\overline{h_\omega(V)}$ intersecting the ball $U$ is bounded by a constant $M$.
Since $h_\omega(A) \subset \overline{h_\omega(V)}$, the number of sets $h_\omega(A)$ intersecting $U$ is also bounded by $M$.

The mass of $U$ is bounded by the sum of the masses of these intersecting sets:
\[ \mu(U) \le \sum_{\substack{\omega \in \mathcal{Q} \\ h_\omega(A) \cap U \neq \emptyset}} \mu(h_\omega(A)) = \sum_{\substack{\omega \in \mathcal{Q} \\ h_\omega(A) \cap U \neq \emptyset}} (W(\omega))^s. \]
Since the number of terms in the sum is at most $M$ and $W(\omega) \le r$, we have:
\[ \mu(U) \le M r^s. \]
Since any set $E$ of diameter $d < 1$ is contained in a ball of radius $d$, we have $\mu(E) \le M d^s = M (\text{diam}(E))^s$. Hence by the Mass Distribution Principle (Proposition \ref{prop:mass_distribution}), $\mathcal{H}^s(A) > 0$.
\end{proof}

\begin{example} \label{ex:hybrid_system}
We illustrate Theorem \ref{thm:hybrid_dimension} with a hybrid system on $\mathbb{R}^2$ consisting of two affine maps $\mathcal{F} = \{f_1, f_2\}$ and one similarity map $\mathcal{G} = \{g_1\}$.

Let the affine maps be defined by:
\begin{align*}
    f_1(x, y) &= \left( -\frac{y}{3} - 1, \; x \right), \\
    f_2(x, y) &= \left( -\frac{y}{3} + 1, \; x \right).
\end{align*}
Let the similarity map be defined by:
\begin{align*}
    g_1(x, y) &= \left( \frac{x}{3}, \; \frac{y}{3} \right).
\end{align*}

The attractor generated by this system is visualized in Figure \ref{fig:hybrid_attractor}.

\begin{figure}[htbp]
    \centering
    \includegraphics[width=0.7\textwidth]{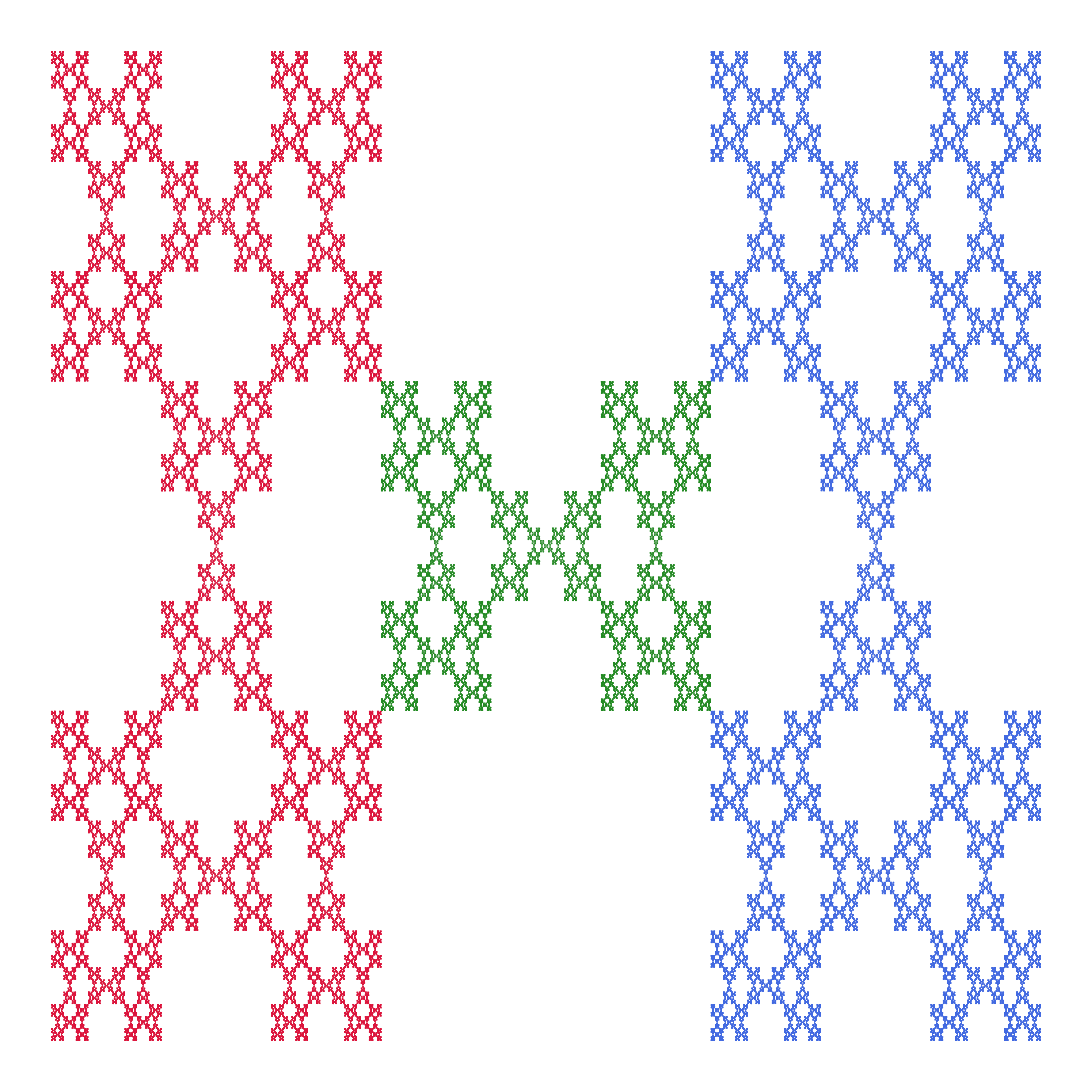}
    \caption{The attractor of the hybrid system defined in Example \ref{ex:hybrid_system}. The red region corresponds to the map $f_1$, the blue region corresponds to the map $f_2$, and the green region corresponds to the map $g_1$.}
    \label{fig:hybrid_attractor}
\end{figure}

Verification of Conditions:
Although $f_1$ and $f_2$ are not similarities, every composition of length $n=2$ (i.e., $f_1^2, f_1 f_2, f_2 f_1, f_2^2$) is a similarity with contraction ratio $c=1/3$. The similarity map $g_1$ is universally aligned as its linear part is a scalar multiple of the identity. The system satisfies the Open Set Condition (OSC).

Dimension Calculation:
We apply the hybrid formula \eqref{eq:hybrid_formula} with $n=2$, affine ratios $c_1=c_2=1/3$, and similarity ratio $r_1=1/3$:
\[
    2\left(\frac{1}{3}\right)^{s/2} + \left(\frac{1}{3}\right)^s = 1.
\]
Substituting $u = (1/3)^{s/2}$ leads to $u^2 + 2u - 1 = 0$. The positive solution is $u = \sqrt{2} - 1$. Solving for $s$:
\[
    3^{s/2} = \sqrt{2} + 1 \implies s = \frac{2 \ln(\sqrt{2} + 1)}{\ln 3} \approx 1.60.
\]
\end{example}

\begin{example} \label{ex:five_map_hybrid}
We consider a hybrid system on $\mathbb{R}^2$ composed of three affine maps $\mathcal{F} = \{f_1, f_2, f_3\}$ and two similarity maps $\mathcal{G} = \{g_1, g_2\}$.

The affine maps are defined by:
\begin{align*}
    f_1(x, y) &= \left( -y, \; \frac{x}{2} \right), \\
    f_2(x, y) &= \left( \frac{y}{2}, \; \frac{x}{4} + 1 \right), \\
    f_3(x, y) &= \left( -\frac{y}{2}, \; -\frac{x}{4} - 1 \right).
\end{align*}
The similarity maps are defined by:
\begin{align*}
    g_1(x, y) &= \left( \frac{x}{4} + 1, \; \frac{y}{4} + 1 \right), \\
    g_2(x, y) &= \left( \frac{x}{4} - 1, \; \frac{y}{4} - 1 \right).
\end{align*}

The attractor generated by this system is visualized in Figure \ref{fig:five_map_attractor}.

\begin{figure}[htbp]
    \centering
    \includegraphics[width=0.7\textwidth]{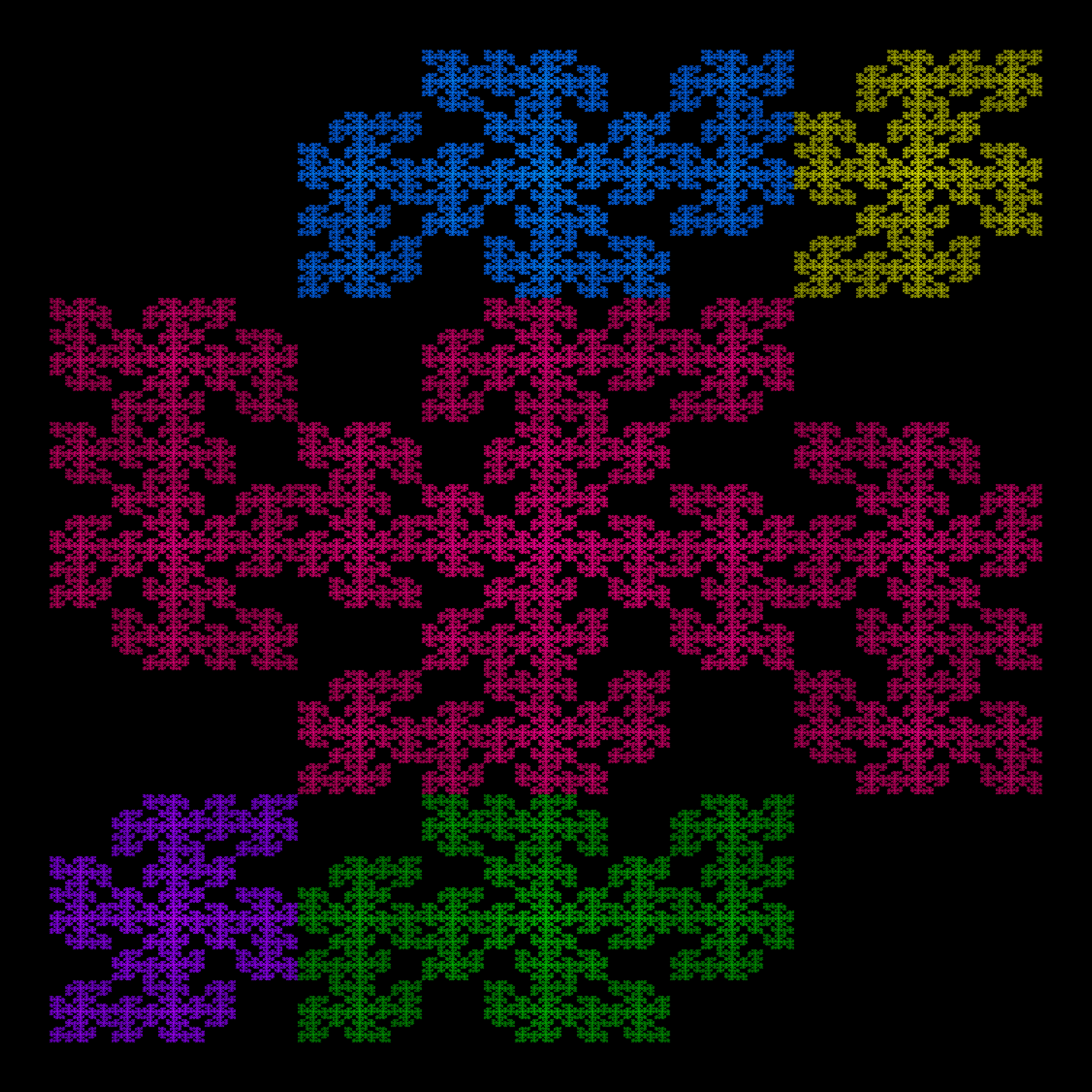}
    \caption{The attractor of the hybrid system defined in Example \ref{ex:five_map_hybrid}. The red structure corresponds to $f_1$, blue to $f_2$, green to $f_3$, while gold and violet correspond to the similarities $g_1$ and $g_2$.}
    \label{fig:five_map_attractor}
\end{figure}

Verification of Conditions:
Although the maps in $\mathcal{F}$ are not similarities individually, every composition of length $n=2$ generated by $\mathcal{F}$ is a similarity contraction. Specifically, the map $f_1^2$ is a similarity with ratio $1/2$. The four mixed compositions involving $f_1$ (i.e., $f_1 \circ f_2$, $f_2 \circ f_1$, $f_1 \circ f_3$, and $f_3 \circ f_1$) are all similarities with ratio $1/4$. The four compositions $f_2^2$, $f_3^2$, $f_2 \circ f_3$, and $f_3 \circ f_2$ are similarities with ratio $1/8$. The maps $g_1$ and $g_2$ are standard similarities with ratio $r_1 = r_2 = 1/4$. The system satisfies the Open Set Condition with the open set $V = (-4/3, 4/3) \times (-4/3, 4/3)$.

Dimension Calculation:
We apply the hybrid formula \eqref{eq:hybrid_formula} with $n=2$. The Hausdorff dimension $s$ is the unique solution to:
\[
    \left(\frac{1}{2}\right)^{s/2} + 2\left(\frac{1}{8}\right)^{s/2} + 2\left(\frac{1}{4}\right)^s = 1.
\]
Substituting $u = (1/2)^{s/2}$, the equation simplifies to the polynomial $2u^4 + 2u^3 + u - 1 = 0$. The relevant positive root is $u \approx 0.5337$, which yields the dimension:
\[
    s = \frac{2 \ln u}{\ln(1/2)} \approx 1.8118.
\]
\end{example}

\section{Topological Properties of Planar $f$-Aligned Systems}

\subsection{Motivating Example: A Class of Planar IFSs}

Before establishing the general structural classification of $f$-aligned systems in the plane, we provide a motivating example that demonstrates the remarkable geometric rigidity of these IFSs. Consider a two-map affine system $\mathcal{F} = \{f, g\}$ on $\mathbb{R}^2$ where the generating functions are given by:
\begin{equation}
    f(x) = \begin{pmatrix} 0 & 0.5 \\ 1 & 0 \end{pmatrix} x + b_f, \quad \quad g(x) = \begin{pmatrix} 0.5 & 0 \\ 0 & 0.5 \end{pmatrix} x + b_g,
\end{equation}
with translation vectors $b_f, b_g \in \mathbb{R}^2$. Let $A$ and $S$ denote the linear parts of $f$ and $g$, respectively. 

We first verify that this system satisfies the algebraic constraints required to apply the exact dimension formula from Theorem \ref{thm:general_dimension}. The matrix $S = 0.5I$ is a uniform scalar multiple of the identity, representing a similarity contraction with ratio $r = 0.5$. Because scalar matrices commute with all square matrices, $S$ trivially commutes with the symmetric matrix $A^\top A$:
\begin{equation}
    S(A^\top A) = (0.5 I) \begin{pmatrix} 1 & 0 \\ 0 & 0.25 \end{pmatrix} = \begin{pmatrix} 0.5 & 0 \\ 0 & 0.125 \end{pmatrix} = (A^\top A)S.
\end{equation}
Thus, $g$ is an $f$-aligned similarity. Next, we examine the iterates of $f$. While $f$ itself is not a similarity, squaring its linear part yields:
\begin{equation}
    A^2 = \begin{pmatrix} 0 & 0.5 \\ 1 & 0 \end{pmatrix} \begin{pmatrix} 0 & 0.5 \\ 1 & 0 \end{pmatrix} = \begin{pmatrix} 0.5 & 0 \\ 0 & 0.5 \end{pmatrix} = 0.5 I.
\end{equation}
Consequently, $f^2$ is a similarity transformation with ratio $c = 0.5$. This classifies $f$ strictly as a $G^2$-similarity contraction ($n=2$). 

Because the system is $f$-aligned and contains a $G^2$-similarity, we can invoke Theorem \ref{thm:general_dimension}. Assuming the open set condition holds, the exact Hausdorff dimension $s$ of the attractor is the unique real solution to the equation $c^{s/2} + r^s = 1$. Substituting $c = 0.5$ and $r = 0.5$ yields:
\begin{equation}
    (0.5)^{s/2} + (0.5)^s = 1.
\end{equation}
Setting $t = (0.5)^{s/2}$ results in the quadratic equation $t^2 + t - 1 = 0$. Taking the positive root gives $t = \frac{\sqrt{5}-1}{2}$. Solving for $s$ provides the exact dimension:
\begin{equation} \label{eq:dim_eval}
    s = \frac{2 \log\left(\frac{\sqrt{5}-1}{2}\right)}{\log(0.5)} \approx 1.3884.
\end{equation}

To observe the physical behavior of this system, we visualize the attractors for six distinct pairs of translation vectors $(b_f, b_g)$.

\begin{figure}[htpb]
    \centering
    \includegraphics[width=0.95\textwidth]{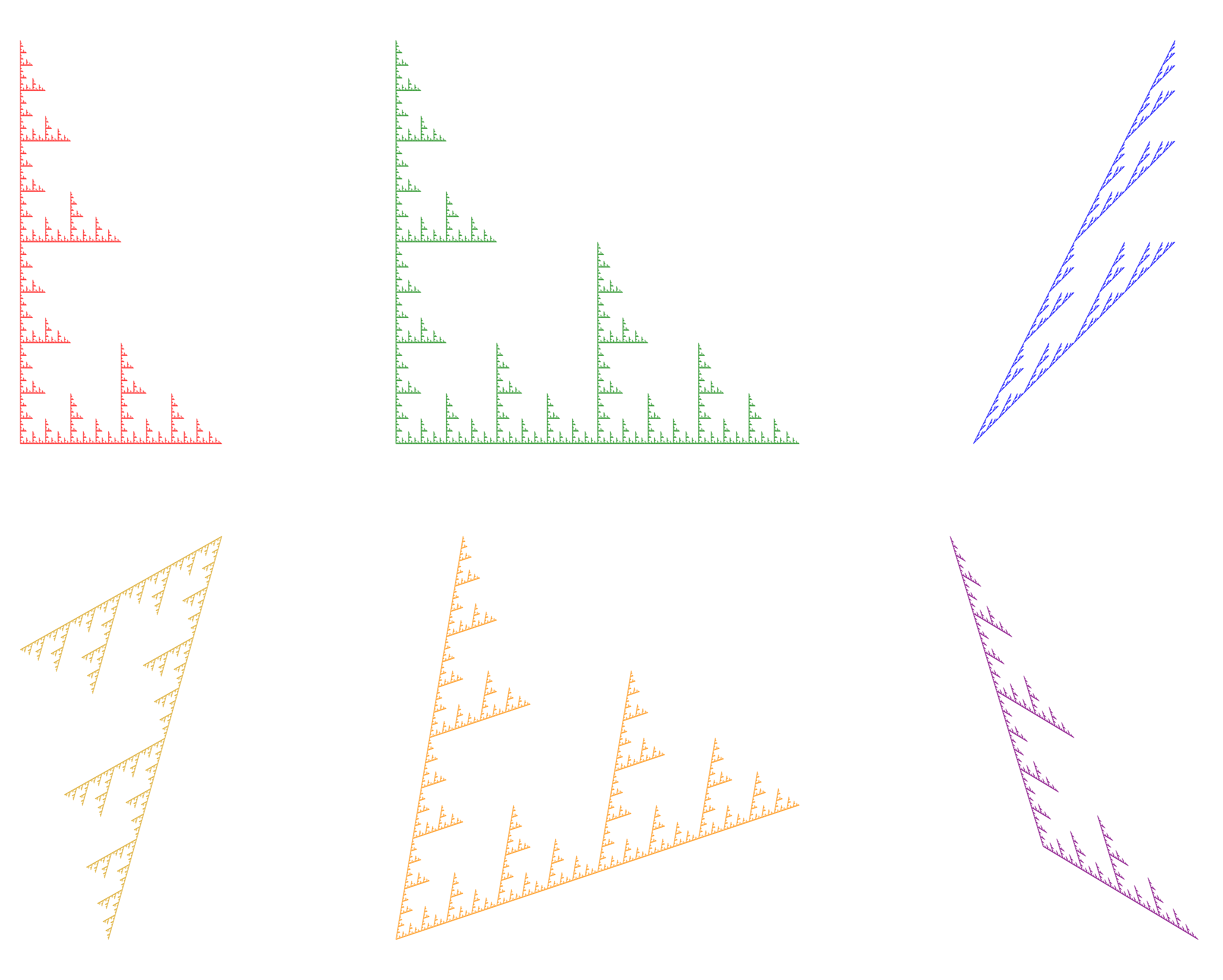}
    \caption{Attractors of the IFS $\{f, g\}$ for various translation vectors: 
(Red) $b_f = (0, 0), b_g = (0, 0.5)$; 
(Green) $b_f = (0, 0.5), b_g = (1, 0.5)$; 
(Blue) $b_f = (-0.5, 0), b_g = (0, 0.5)$; 
(Gold) $b_f = (0.2, -0.2), b_g = (-0.8, -3.2)$; 
(Orange) $b_f = (0, 0), b_g = (3, 1)$; 
(Purple) $b_f = (0, 0.5), b_g = (0.1, 1)$.}
    \label{fig:translation_variations}
\end{figure}

As shown in Figure \ref{fig:translation_variations}, changing the translation vectors moves the fixed points of the maps, drastically altering the overall shape of the attractor. However, the system maintains strict topological invariants. It can be analytically verified that for each pair of translation vectors depicted in Figure \ref{fig:translation_variations}, the IFS satisfies the open set condition. Consequently, by Theorem \ref{thm:general_dimension}, every attractor shown in the figure possesses the exact Hausdorff dimension $s \approx 1.3884$. Furthermore, in every such case, the attractor remains a connected set.

The underlying mechanism enforcing this topological rigidity is the precise algebraic balance between the contraction ratios: $c + r = 0.5 + 0.5 = 1$. It is entirely due to this single parametric constraint that the open set condition and global connectivity are unconditionally guaranteed for any choice of translation vectors $b_f, b_g \in \mathbb{R}^2$ (provided the attractor does not collapse into a line). 

This geometric rigidity extends far beyond the preceding example. In the remainder of this section, we formalize these phenomena. We first provide a complete classification of planar $G^2$-similarity contractions, and then classify the system $\{f, g\}$ to prove exactly when the condition $c + r = 1$ forces separation and connectivity, and when it fails.

\subsection{Classification and Separation Conditions}

\begin{lemma}\label{lem:classification}
Let $f: \mathbb{R}^2 \to \mathbb{R}^2$ be an affine transformation with linear part $A$. If $f$ is a strict $G^2$-similarity contraction, then $A^2 = cI$ or $A^2 = -cI$ for some $c \in (0, 1)$.
\end{lemma}

\begin{proof}
Since $f^2$ is a similarity contraction, its linear part $A^2$ must be a similarity matrix. Thus, there exists a constant $c \in (0, 1)$ and a $2 \times 2$ orthogonal matrix $U$ such that $A^2 = cU$. 

Taking the determinant yields $(\det A)^2 = c^2 \det(U)$. Since $A$ is a real matrix, $(\det A)^2 \ge 0$. Because $c > 0$, this strictly requires $\det(U) \ge 0$. The only orthogonal matrices in $\mathbb{R}^2$ with a positive determinant are rotation matrices. Therefore, $\det(U) = 1$, and $U$ is a rotation matrix $R_\theta$, giving $A^2 = c R_\theta$. This also implies $\det(A) = \pm c$.

By the Cayley-Hamilton theorem, the $2 \times 2$ matrix $A$ satisfies its characteristic equation:
\begin{equation} \label{eq:cayley}
    A^2 - \text{tr}(A)A + \det(A)I = 0.
\end{equation}
Substituting $A^2 = c R_\theta$ into Equation \ref{eq:cayley} yields:
\begin{equation} \label{eq:trace}
    \text{tr}(A)A = c R_\theta + \det(A)I.
\end{equation}

Assume for contradiction that $\text{tr}(A) \neq 0$. Then $A = \frac{1}{\text{tr}(A)} (c R_\theta + \det(A)I)$. In $\mathbb{R}^2$, both the identity matrix $I$ and any rotation matrix $R_\theta$ are conformal matrices of the form $\left(\begin{smallmatrix} x & -y \\ y & x \end{smallmatrix}\right)$. Since the set of conformal matrices is closed under addition and scalar multiplication, $A$ must also be a conformal matrix. A conformal matrix in $\mathbb{R}^2$ corresponds to a similarity transformation, which contradicts the hypothesis that $f$ is a strict $G^2$-similarity contraction (i.e., $f$ is not a similarity). 

Therefore, it must be that $\text{tr}(A) = 0$. Substituting this back into Equation \ref{eq:trace} gives:
\begin{equation*}
    0 = c R_\theta + \det(A)I \implies c R_\theta = -\det(A)I.
\end{equation*}
Since the right side is a strictly diagonal matrix, $R_\theta$ must also be diagonal. The only diagonal rotation matrices in $\mathbb{R}^2$ are $R_0 = I$ and $R_\pi = -I$. 
Consequently, we have exactly two possibilities:
\begin{enumerate}
    \item If $\det(A) = -c$, then $c R_\theta = cI$, which implies $A^2 = cI$.
    \item If $\det(A) = c$, then $c R_\theta = -cI$, which implies $A^2 = -cI$.
\end{enumerate}
Thus, every strict $G^2$-similarity contraction in the plane falls into one of these two algebraic forms.
\end{proof}

\begin{lemma}\label{lem:f_aligned_classification}
Let $f: \mathbb{R}^2 \to \mathbb{R}^2$ be a strict $G^2$-similarity contraction with linear part $A$. If $g: \mathbb{R}^2 \to \mathbb{R}^2$ is an $f$-aligned similarity transformation with ratio $r > 0$ and linear part $S$, then $S$ is either a scalar multiple of the identity ($S = \pm rI$) or a scaled reflection across an eigenvector of $A^\top A$.
\end{lemma}

\begin{proof}
Since $f$ is a strict $G^2$-similarity contraction, it is not a uniform similarity transformation. Consequently, the symmetric matrix $A^\top A$ is not a scalar multiple of the identity, meaning it possesses two distinct positive real eigenvalues, $\lambda_1 \neq \lambda_2$.

By the Spectral Theorem, there exists a $2 \times 2$ orthogonal matrix $Q$ whose columns are the eigenvectors of $A^\top A$, such that $A^\top A = Q \Lambda Q^\top$, where $\Lambda = \text{diag}(\lambda_1, \lambda_2)$.

By Definition \ref{def:f_aligned}, $g$ is $f$-aligned, so $S$ commutes with $A^\top A$:
\begin{equation*}
    S(A^\top A) = (A^\top A)S.
\end{equation*}
Substituting the spectral decomposition yields $S Q \Lambda Q^\top = Q \Lambda Q^\top S$. Multiplying by $Q^\top$ on the left and $Q$ on the right gives:
\begin{equation*}
    (Q^\top S Q) \Lambda = \Lambda (Q^\top S Q).
\end{equation*}
Let $D = Q^\top S Q$. Since $D$ commutes with the diagonal matrix $\Lambda$, which has strictly distinct diagonal entries, $D$ itself must be a diagonal matrix. 

Because $D$ is real and diagonal, it is symmetric ($D^\top = D$). It follows that $S = Q D Q^\top$ is also symmetric ($S^\top = S$). 

Since $g$ is a similarity transformation with ratio $r$, its linear part must satisfy $S^\top S = r^2 I$. Because $S$ is symmetric, this strictly reduces to $S^2 = r^2 I$. Consequently, the eigenvalues of $S$ must be drawn from the set $\{r, -r\}$.

Since $D$ is a diagonal matrix containing the eigenvalues of $S$, there are exactly two possibilities for $D$:
\begin{enumerate}
    \item $D = \pm rI$. In this case, $S = Q(\pm rI)Q^\top = \pm rI$. Geometrically, this corresponds to a uniform homothety ($rI$) or a point reflection ($-rI$).
    \item $D = \left(\begin{smallmatrix} r & 0 \\ 0 & -r \end{smallmatrix}\right)$ or $D = \left(\begin{smallmatrix} -r & 0 \\ 0 & r \end{smallmatrix}\right)$. In this case, $S = Q D Q^\top$ is a symmetric matrix with eigenvalues $r$ and $-r$. Geometrically, this corresponds to a reflection across one of the principal axes defined by the column vectors of $Q$ (the eigenvectors of $A^\top A$), scaled by the ratio $r$.
\end{enumerate}
Thus, the linear part $S$ of any planar $f$-aligned similarity must fall into one of these specific algebraic forms.
\end{proof}

\begin{theorem}\label{thm:neg_cI}
Let $\mathcal{F} = \{f, g\}$ be an iterated function system on $\mathbb{R}^2$ consisting of affine transformations satisfying the following conditions:
\begin{enumerate}
    \item $f(x) = Ax + b_f$, where $b_f \in \mathbb{R}^2$, $\det(A) > 0$, and $A^2 = -cI$ for some $c \in (0, 1)$.
    \item $g(x) = Sx + b_g$, where $b_g \in \mathbb{R}^2$ and $S = \pm rI$ for some $r \in (0, 1)$. 
    \item $c + r \le 1$.
    \item The respective fixed points $z_f$ and $z_g$ of $f$ and $g$ are distinct.
\end{enumerate}
Then $\mathcal{F}$ satisfies the open set condition.
\end{theorem}

\begin{proof}
Let $z_f, z_g \in \mathbb{R}^2$ denote the unique fixed points of $f$ and $g$, respectively. By hypothesis, $v = z_f - z_g \neq 0$. We first show that $v$ and $Av$ are linearly independent. Assume there exists $\gamma \in \mathbb{R}$ such that $Av = \gamma v$. Then $A^2 v = \gamma^2 v$. Since $A^2 = -cI$, this gives $-cv = \gamma^2 v$, so $\gamma^2 = -c$. Since $c \in (0, 1)$, this has no real solution, a contradiction. Thus, $v$ and $Av$ are linearly independent.

We consider two cases based on $S$.

\vspace{0.5em}
\noindent\textbf{Case 1: $S = rI$}
\vspace{0.5em}

Here, $g(x) = rx + (1 - r)z_g$. Choose $n \in \mathbb{R}^2 \setminus \{0\}$ such that $\langle n, A(z_f - z_g) \rangle = 0$. By linear independence, $\langle n, z_f - z_g \rangle \neq 0$. Without loss of generality, assume $\langle n, z_f - z_g \rangle > 0$.

Define $m = A^\top n$. Since $A^\top$ has no real eigenvalues, $m$ is not a scalar multiple of $n$, so $n$ and $m$ are linearly independent. Projecting the fixed points onto $m$ gives $\langle m, z_f - z_g \rangle = \langle n, A(z_f - z_g) \rangle = 0$. Thus, there exists $k \in \mathbb{R}$ such that $\langle m, z_f \rangle = \langle m, z_g \rangle = k$. Additionally, $A^\top m = (A^\top)^2 n = -cn$.

Define $\pi_n(x) = \langle n, x \rangle$ and $\pi_m(x) = \langle m, x \rangle$. Using $b_f = z_f - Az_f$, the projections of $f(x)$ are:
\begin{align*}
    \pi_n(f(x)) &= \langle A^\top n, x \rangle + \pi_n(z_f - Az_f) = \pi_m(x) - k + \pi_n(z_f), \\
    \pi_m(f(x)) &= \langle A^\top m, x \rangle + \pi_m(z_f - Az_f) = -c\pi_n(x) + k + c\pi_n(z_f).
\end{align*}
For $g(x)$, the projections are:
\begin{align*}
    \pi_n(g(x)) &= r \pi_n(x) + (1 - r)\pi_n(z_g), \\
    \pi_m(g(x)) &= r \pi_m(x) + (1 - r)k.
\end{align*}

Define the open parallelogram $O = \{ x \in \mathbb{R}^2 \mid \alpha_1 < \pi_n(x) < \alpha_2 \text{ and } \beta_1 < \pi_m(x) < \beta_2 \}$. Let $L_n = (1+c)(\pi_n(z_f) - \pi_n(z_g)) > 0$. Set the boundaries as $\alpha_1 = \pi_n(z_g)$, $\alpha_2 = \alpha_1 + L_n$, $\beta_1 = k - c(\alpha_2 - \pi_n(z_f))$, and $\beta_2 = k - c(\alpha_1 - \pi_n(z_f))$. 

Note that $\beta_2 - \beta_1 = cL_n > 0$. For $x \in O$, $\pi_n(g(O)) = (\alpha_1, \alpha_1 + rL_n)$. Since $k \in (\beta_1, \beta_2)$ and $r \in (0, 1)$, $\pi_m(g(O)) \subset (\beta_1, \beta_2)$.

For $f$, since $\pi_n(x) \in (\alpha_1, \alpha_2)$ and $-c < 0$, $\pi_m(f(O)) = (\beta_1, \beta_2)$. Since $\pi_m(x) \in (\beta_1, \beta_2)$, the bounds for $\pi_n(f(x))$ are $\inf \pi_n(f(O)) = \beta_1 - k + \pi_n(z_f) = -c\alpha_2 + (1+c)\pi_n(z_f)$. Using $\alpha_2 = \alpha_1 + L_n$ and $(1+c)\pi_n(z_f) = L_n + (1+c)\alpha_1$, this simplifies to $\alpha_1 + (1-c)L_n = \alpha_2 - cL_n$. The supremum is $\sup \pi_n(f(O)) = \beta_2 - k + \pi_n(z_f) = \alpha_2$. Thus, $\pi_n(f(O)) = (\alpha_2 - cL_n, \alpha_2)$.

Because $c + r \le 1$, we have $rL_n \le (1-c)L_n$. This ensures $\sup \pi_n(g(O)) = \alpha_1 + rL_n \le \alpha_2 - cL_n = \inf \pi_n(f(O))$. Therefore, $\pi_n(g(O))$ and $\pi_n(f(O))$ are disjoint sub-intervals of $(\alpha_1, \alpha_2)$, yielding $f(O) \cap g(O) = \emptyset$ and $f(O) \cup g(O) \subset O$.

\vspace{0.5em}
\noindent\textbf{Case 2: $S = -rI$}
\vspace{0.5em}

Here, $g(x) = -r(x - z_g) + z_g$. We construct an oblique projection axis by defining the scalar:
\begin{equation*}
    \lambda = \frac{c(c-r)}{1-cr}.
\end{equation*}
Choose $n \in \mathbb{R}^2 \setminus \{0\}$ such that $\langle n, A(z_f - z_g) \rangle = \lambda \langle n, z_f - z_g \rangle$. Since $v$ and $Av$ are linearly independent, $\langle n, z_f - z_g \rangle \neq 0$. Assume $\langle n, z_f - z_g \rangle > 0$.

Define $m = A^\top n$. Projecting onto $m$ gives $\langle m, z_f - z_g \rangle = \lambda \langle n, z_f - z_g \rangle$. Thus, $\langle m, z_f \rangle - \lambda \langle n, z_f \rangle = \langle m, z_g \rangle - \lambda \langle n, z_g \rangle = k$ for some $k \in \mathbb{R}$. Also, $A^\top m = -cn$.

The projections of $f$ are:
\begin{align*}
    \pi_n(f(x)) &= \pi_m(x) - k + (1-\lambda)\pi_n(z_f), \\
    \pi_m(f(x)) &= -c\pi_n(x) + k + (\lambda+c)\pi_n(z_f).
\end{align*}
For $g(x)$:
\begin{align*}
    \pi_n(g(x)) &= -r\pi_n(x) + (1+r)\pi_n(z_g), \\
    \pi_m(g(x)) &= -r\pi_m(x) + (1+r)\pi_m(z_g).
\end{align*}

Define $O$ by the bounds $\alpha_2 = \frac{1}{1-cr}((1+c)\pi_n(z_f) - c(1+r)\pi_n(z_g))$, $\alpha_1 = (1+r)\pi_n(z_g) - r\alpha_2$, $\beta_1 = -c\alpha_2 + k + (\lambda+c)\pi_n(z_f)$, and $\beta_2 = -c\alpha_1 + k + (\lambda+c)\pi_n(z_f)$.

Note that $\alpha_2 - \alpha_1 = \frac{(1+r)(1+c)}{1-cr}(\pi_n(z_f) - \pi_n(z_g)) > 0$. Let $L_n = \alpha_2 - \alpha_1$. Then $\beta_2 - \beta_1 = cL_n > 0$.

For $f$, since $-c < 0$, $\pi_m(f(O)) = (\beta_1, \beta_2)$. The primary projection yields $\inf \pi_n(f(O)) = \beta_1 - k + (1-\lambda)\pi_n(z_f) = -c\alpha_2 + (1+c)\pi_n(z_f)$. Using $\alpha_1 = (1+r)\pi_n(z_g) - r\alpha_2$, we rewrite $\alpha_2$ to obtain $(1+c)\pi_n(z_f) = \alpha_2 + c\alpha_1$. Substituting this yields $\inf \pi_n(f(O)) = -c\alpha_2 + \alpha_2 + c\alpha_1 = \alpha_2 - cL_n$. The supremum is $\sup \pi_n(f(O)) = \alpha_2$. Thus, $\pi_n(f(O)) = (\alpha_2 - cL_n, \alpha_2)$.

For $g$, $\inf \pi_n(g(O)) = -r\alpha_2 + (1+r)\pi_n(z_g) = \alpha_1$ and $\sup \pi_n(g(O)) = -r\alpha_1 + (1+r)\pi_n(z_g) = -r\alpha_1 + \alpha_1 + r\alpha_2 = \alpha_1 + rL_n$. Thus, $\pi_n(g(O)) = (\alpha_1, \alpha_1 + rL_n)$.

Because $c + r \le 1$, we have $\alpha_1 + rL_n \le \alpha_2 - cL_n$. This guarantees that $\pi_n(g(O))$ and $\pi_n(f(O))$ are disjoint sub-intervals of $(\alpha_1, \alpha_2)$.

For the dual axis, $\inf \pi_m(g(O)) = -r\beta_2 + (1+r)\pi_m(z_g)$. Using $\pi_m(z_g) = k + \lambda \pi_n(z_g)$ and $\lambda = \frac{c(c-r)}{1-cr}$, this algebraically strictly reduces to $\beta_1$. This directly implies $(1+r)\pi_m(z_g) = \beta_1 + r\beta_2$. The upper bound is $\sup \pi_m(g(O)) = -r\beta_1 + (1+r)\pi_m(z_g) = -r\beta_1 + \beta_1 + r\beta_2 = \beta_1 + r(\beta_2 - \beta_1) = \beta_1 + rcL_n < \beta_2$. Thus, $\pi_m(g(O)) \subset (\beta_1, \beta_2)$.

In both cases, there exists a non-degenerate open set $O$ such that $f(O) \cup g(O) \subset O$ and $f(O) \cap g(O) = \emptyset$, satisfying the open set condition.
\end{proof}

We provide two examples satisfying the hypotheses of Theorem \ref{thm:neg_cI}, corresponding to the cases $S = rI$ and $S = -rI$.

\begin{example}[$S = rI$] \label{ex:neg_cI_homothety}
Let $\mathcal{F} = \{f, g\}$ be the IFS on $\mathbb{R}^2$ given by:
\begin{align*}
    f(x) &= \begin{pmatrix} 0.4 & -0.6 \\ 1.1 & -0.4 \end{pmatrix} x, \\
    g(x) &= \begin{pmatrix} 0.5 & 0 \\ 0 & 0.5 \end{pmatrix} x + \begin{pmatrix} 1.0 \\ 0.5 \end{pmatrix}.
\end{align*}
Let $A$ and $S$ be the linear parts of $f$ and $g$. We have:
\begin{equation*}
    A^2 = \begin{pmatrix} -0.5 & 0 \\ 0 & -0.5 \end{pmatrix} = -0.5 I.
\end{equation*}
Thus, $A^2 = -cI$ with $c = 0.5$. Furthermore, $\det(A) = 0.5 > 0$, and $S = rI$ with $r = 0.5$. The parameter constraint $c + r \le 1$ holds perfectly as $0.5 + 0.5 = 1$.

The fixed points are $z_f = (0, 0)^\top$ and $z_g = (2, 1)^\top$. Since $z_f \neq z_g$, the hypotheses of Theorem \ref{thm:neg_cI} are satisfied.

To explicitly construct the bounding open set $O$, we determine the invariant dual axes. The connecting vector is $v = z_f - z_g = (-2, -1)^\top$. We choose a normal vector $n = (-9, 1)^\top$ such that $\langle n, Av \rangle = 0$. The corresponding dual axis is $m = A^\top n = (-2.5, 5)^\top$. 

The projections of the fixed points onto these axes yield $\pi_n(z_f) = 0$, $\pi_n(z_g) = -17$, and $k = \pi_m(z_f) = 0$. Applying the formulas derived in Theorem \ref{thm:neg_cI}, the primary projection length is $L_n = (1+c)(\pi_n(z_f) - \pi_n(z_g)) = 1.5(17) = 25.5$. The exact affine boundaries evaluate to:
\begin{align*}
    \alpha_1 &= \pi_n(z_g) = -17, \\
    \alpha_2 &= \alpha_1 + L_n = 8.5, \\
    \beta_1 &= k - c(\alpha_2 - \pi_n(z_f)) = 0 - 0.5(8.5) = -4.25, \\
    \beta_2 &= k - c(\alpha_1 - \pi_n(z_f)) = 0 - 0.5(-17) = 8.5.
\end{align*}
Thus, the canonical open set is exactly defined in the dual abstract space as:
\begin{equation*}
    O_{\pi} = \{ x \in \mathbb{R}^2 \mid -17 < \pi_n(x) < 8.5 \quad \text{and} \quad -4.25 < \pi_m(x) < 8.5 \}.
\end{equation*}

To map this abstract parallelogram back to standard Euclidean space, we construct the projection matrix $P$ whose rows are $n^\top$ and $m^\top$, and compute its inverse:
\begin{equation*}
    P = \begin{pmatrix} -9 & 1 \\ -2.5 & 5 \end{pmatrix} 
    \implies 
    P^{-1} = \frac{1}{-42.5} \begin{pmatrix} 5 & -1 \\ 2.5 & -9 \end{pmatrix}.
\end{equation*}
Multiplying the four boundary corners of $O_{\pi}$ by $P^{-1}$ yields their exact Cartesian coordinates in $\mathbb{R}^2$:
\begin{align*}
    v_1 &= P^{-1} \begin{pmatrix} -17 \\ -4.25 \end{pmatrix} = \begin{pmatrix} 1.9 \\ 0.1 \end{pmatrix}, & 
    v_2 &= P^{-1} \begin{pmatrix} 8.5 \\ -4.25 \end{pmatrix} = \begin{pmatrix} -1.1 \\ -1.4 \end{pmatrix}, \\
    v_3 &= P^{-1} \begin{pmatrix} 8.5 \\ 8.5 \end{pmatrix} = \begin{pmatrix} -0.8 \\ 1.3 \end{pmatrix}, & 
    v_4 &= P^{-1} \begin{pmatrix} -17 \\ 8.5 \end{pmatrix} = \begin{pmatrix} 2.2 \\ 2.8 \end{pmatrix}.
\end{align*}
The polygon defined by the vertices $\{v_1, v_2, v_3, v_4\}$ constitutes the bounding set $O$. The resulting attractor and this rigorously derived Euclidean bounding box are plotted alongside each other in Figure \ref{fig:example_3_1}, visually confirming that $f(O) \cap g(O) = \emptyset$ and $f(O) \cup g(O) \subset O$.

\begin{figure}[htpb]
    \centering
    \includegraphics[width=\textwidth]{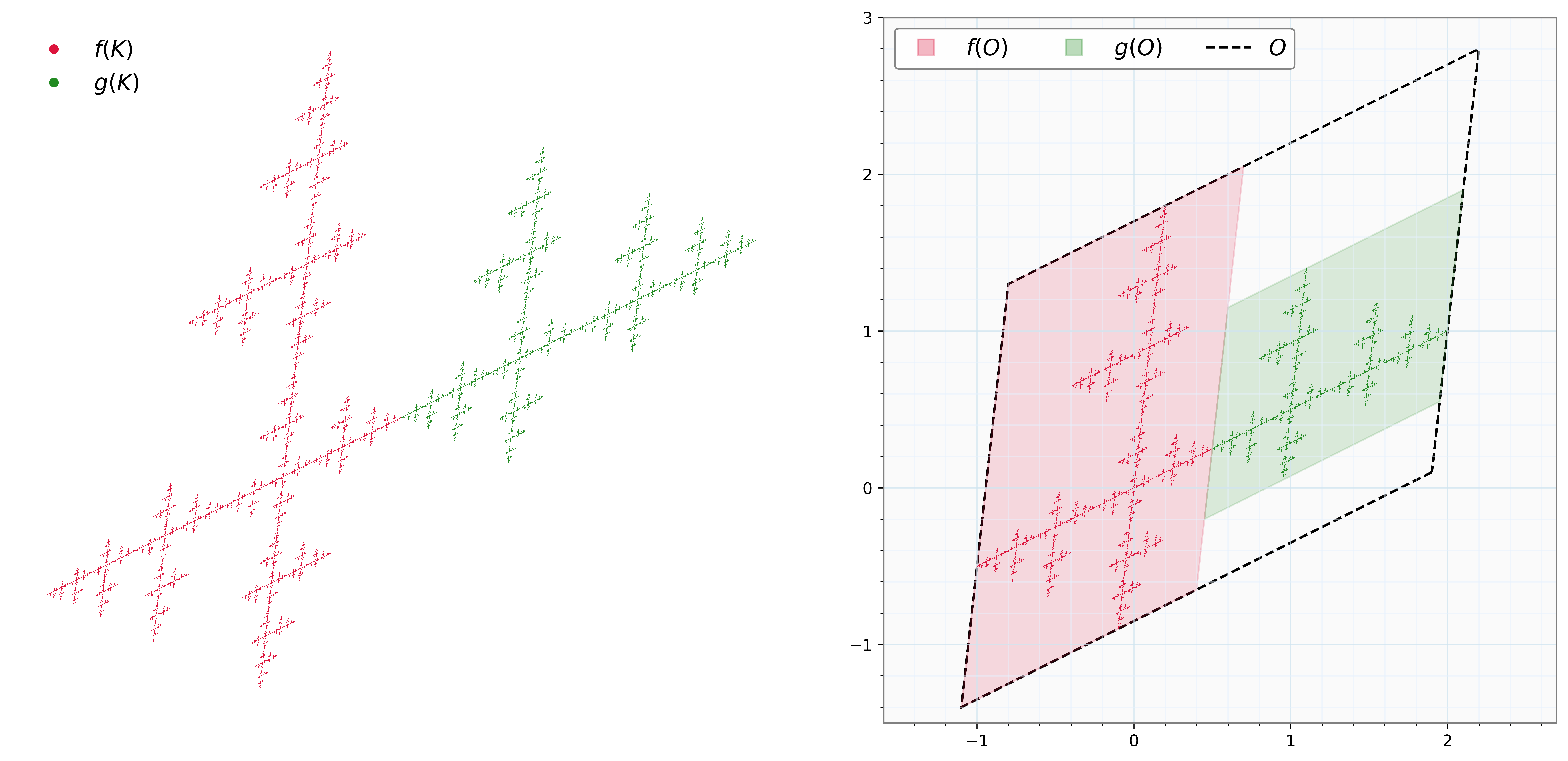} 
    \caption{Geometric validation of the Open Set Condition for Example \ref{ex:neg_cI_homothety} ($A^2 = -cI, S = rI$). The left panel displays the attractor $K = f(K) \cup g(K)$. The right panel illustrates the rigorously constructed Euclidean bounding set $O$, mapped from the theoretical dual axes, alongside its strictly disjoint affine images $f(O)$ and $g(O)$. This provides direct geometric corroboration of the algebraic separation guarantees, confirming $f(O) \cap g(O) = \emptyset$ and $f(O) \cup g(O) \subset O$.}
    \label{fig:example_3_1}
\end{figure}
\end{example}

\begin{example}[$S = -rI$] \label{ex:neg_cI_reflection}
Let $\mathcal{F} = \{f, g\}$ be the IFS given by:
\begin{align*}
    f(x) &= \begin{pmatrix} 0.2 & -0.7 \\ 1.2 & -0.2 \end{pmatrix} x, \\
    g(x) &= \begin{pmatrix} -0.2 & 0 \\ 0 & -0.2 \end{pmatrix} x + \begin{pmatrix} 1.0 \\ 0 \end{pmatrix}.
\end{align*}
Let $A$ and $S$ be the linear parts of $f$ and $g$. Squaring $A$ yields $A^2 = -0.8 I$, so $c = 0.8$. We have $\det(A) = 0.8 > 0$ and $S = -rI$ with $r = 0.2$. The parameter constraint $c + r \le 1$ holds perfectly as $0.8 + 0.2 = 1$.

The fixed points are $z_f = (0, 0)^\top$ and $z_g = (5/6, 0)^\top$. Since $z_f \neq z_g$, the hypotheses of Theorem \ref{thm:neg_cI} are satisfied.

To construct the open set $O$, we calculate the orientation-reversing scalar $\lambda = \frac{c(c-r)}{1-cr} = \frac{4}{7}$. The corresponding integer normal vector is $n = (-42, -13)^\top$, yielding the dual axis $m = A^\top n = (-24, 32)^\top$. 

Following the boundary formulas from Case 2 of Theorem \ref{thm:neg_cI}, the abstract open set $O_{\pi}$ in the dual space evaluates exactly to:
\begin{equation*}
    O_{\pi} = \{ x \in \mathbb{R}^2 \mid -50 < \pi_n(x) < 40 \quad \text{and} \quad -32 < \pi_m(x) < 40 \}.
\end{equation*}

Applying the inverse projection matrix $P^{-1}$, constructed from $n^\top$ and $m^\top$ as detailed in Example \ref{ex:neg_cI_homothety}, we map these abstract boundaries back to $\mathbb{R}^2$. This yields the vertices of the bounding set $O$:
\begin{equation*}
    v_1 = \begin{pmatrix} 28/23 \\ -2/23 \end{pmatrix}, \quad
    v_2 = \begin{pmatrix} -12/23 \\ -32/23 \end{pmatrix}, \quad
    v_3 = \begin{pmatrix} -25/23 \\ 10/23 \end{pmatrix}, \quad
    v_4 = \begin{pmatrix} 15/23 \\ 40/23 \end{pmatrix}.
\end{equation*}
The resulting attractor and this rigorously derived Euclidean bounding box are plotted alongside each other in Figure \ref{fig:example_3_2}, visually confirming the separation condition $f(O) \cap g(O) = \emptyset$.

\begin{figure}[htpb]
    \centering
    \includegraphics[width=\textwidth]{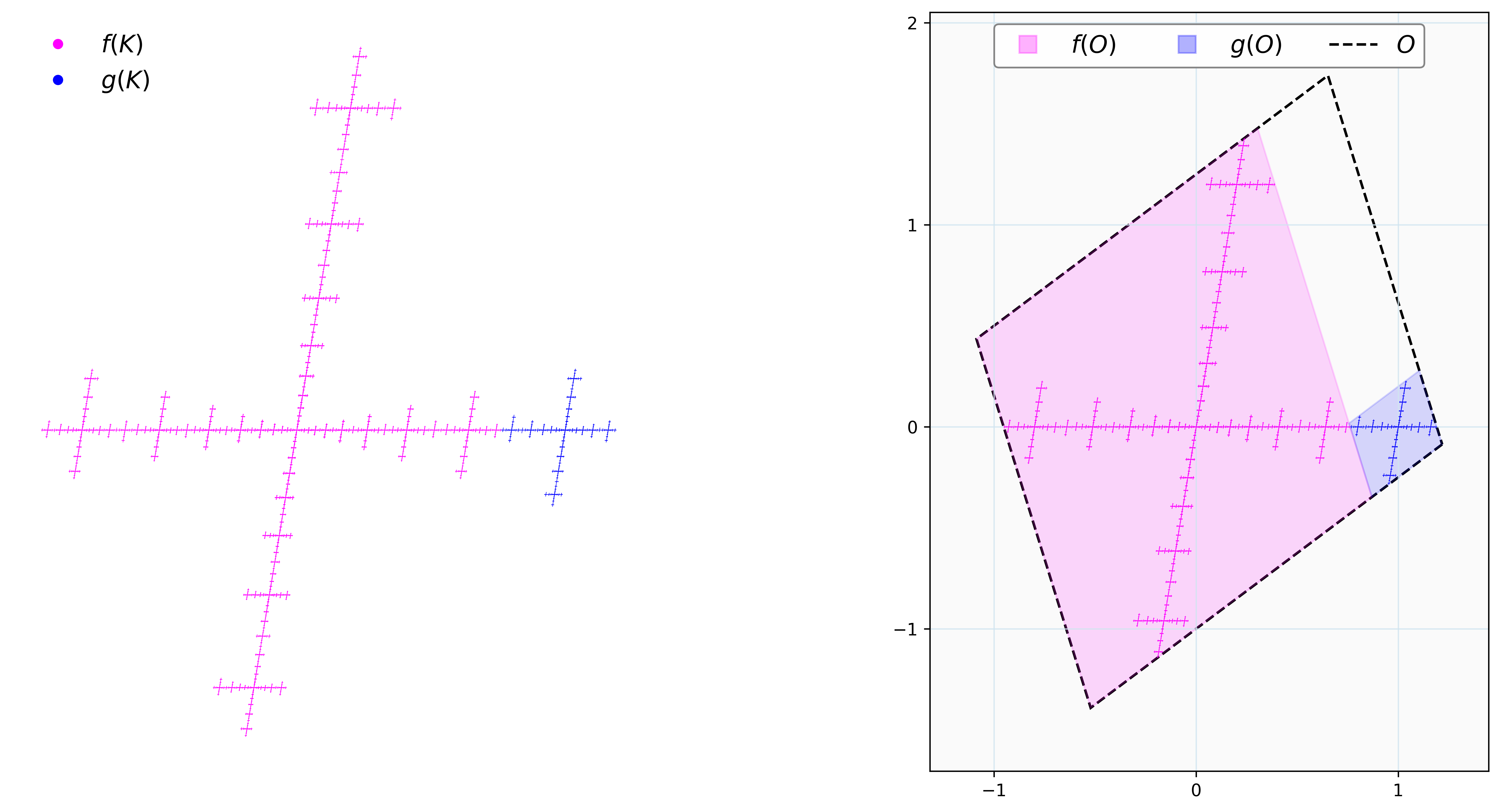} 
   \caption{The attractor (left) and the Euclidean bounding set $O$ (right) for Example \ref{ex:neg_cI_reflection}. The lack of interior overlap between the regions $f(O)$ and $g(O)$ provides visual proof of the open set condition.}
    \label{fig:example_3_2}
\end{figure}
\end{example}

\begin{theorem}\label{thm:pos_cI}
Let $\mathcal{F} = \{f, g\}$ be an iterated function system on $\mathbb{R}^2$ consisting of affine transformations satisfying the following conditions:
\begin{enumerate}
    \item $f(x) = Ax + b_f$, where $b_f \in \mathbb{R}^2$, $\det(A) < 0$, and $A^2 = cI$ for some $c \in (0, 1)$.
    \item $g(x) = Sx + b_g$, where $b_g \in \mathbb{R}^2$ and $S = \pm rI$ for some $r \in (0, 1)$. 
    \item $c + r \le 1$.
    \item The respective fixed points $z_f$ and $z_g$ of $f$ and $g$ are distinct, and the attractor $K$ is not contained in a single 1-dimensional line.
\end{enumerate}
Then $\mathcal{F}$ satisfies the open set condition.
\end{theorem}

\begin{proof}
Let $z_f, z_g \in \mathbb{R}^2$ denote the unique fixed points of $f$ and $g$, respectively. Since they are distinct, $v = z_f - z_g \neq 0$. If $v$ were an eigenvector of $A$, then $Av = \pm \sqrt{c} v$. Because $S = \pm rI$, both $f$ and $g$ would map the line passing through $z_f$ and $z_g$ strictly into itself, collapsing the attractor $K$ to a 1-dimensional line segment. This contradicts the hypothesis. Thus, $v$ and $Av$ are linearly independent over $\mathbb{R}$.

Choose a non-zero normal vector $n \in \mathbb{R}^2$ orthogonal to $v - c^{-1}Av$, such that:
\begin{equation*}
    \langle n, A(z_f - z_g) \rangle = c \langle n, z_f - z_g \rangle.
\end{equation*}
By linear independence, $\langle n, z_f - z_g \rangle \neq 0$. Without loss of generality, assume $L = \langle n, z_f - z_g \rangle > 0$.

Define $m = A^\top n$. We verify that $n$ and $m$ are linearly independent. Assume for contradiction that $m = \lambda n$ for some $\lambda \in \mathbb{R}$. Then $A^\top n = \lambda n$. Evaluating $\langle n, Av \rangle$ yields $\langle A^\top n, v \rangle = \lambda \langle n, v \rangle$. By construction, $\langle n, Av \rangle = c\langle n, v \rangle$. Since $\langle n, v \rangle \neq 0$, this implies $\lambda = c$, yielding $A^\top n = cn$. Applying $A^\top$ a second time yields $(A^\top)^2 n = c^2 n$. However, since $(A^\top)^2 = cI$, it must also hold that $(A^\top)^2 n = cn$. Thus $cn = c^2 n$, which implies $c = c^2$. Since $c \in (0, 1)$, this is a contradiction. Consequently, $n$ and $m$ are linearly independent.

Evaluating the projection of the fixed points onto $m$ yields:
\begin{equation*}
    \langle m, z_f - z_g \rangle = \langle A^\top n, z_f - z_g \rangle = \langle n, A(z_f - z_g) \rangle = cL.
\end{equation*}
Thus, $\langle m, z_f \rangle - c\langle n, z_f \rangle = \langle m, z_g \rangle - c\langle n, z_g \rangle = k$ for some $k \in \mathbb{R}$. Applying $A^\top$ to $m$ gives:
\begin{equation*}
    A^\top m = (A^\top)^2 n = (A^2)^\top n = cn.
\end{equation*}

Define $\pi_n(x) = \langle n, x \rangle$ and $\pi_m(x) = \langle m, x \rangle$. Using $b_f = z_f - Az_f$ and $\pi_m(z_f) = \langle n, Az_f \rangle = c\pi_n(z_f) + k$, the projections of $f$ are:
\begin{align*}
    \pi_n(f(x)) &= \langle A^\top n, x \rangle + \pi_n(z_f - Az_f) = \pi_m(x) + \pi_n(z_f) - \pi_m(z_f), \\
    \pi_m(f(x)) &= \langle A^\top m, x \rangle + \pi_m(z_f - Az_f) \\
    &= c\pi_n(x) + \pi_m(z_f) - c\pi_n(z_f) = c\pi_n(x) + k.
\end{align*}

We divide the construction of the canonical open set $O$ into two cases based on $S$. Let $O = \{ x \in \mathbb{R}^2 \mid \alpha_1 < \pi_n(x) < \alpha_2 \text{ and } \beta_1 < \pi_m(x) < \beta_2 \}$. Due to the linear independence of $n$ and $m$, $O$ is a non-degenerate, 2-dimensional open set.

\vspace{0.5em}
\noindent\textbf{Case 1: $S = rI$}
\vspace{0.5em}

Here, $g(x) = rx + (1 - r)z_g$. The projections are:
\begin{align*}
    \pi_n(g(x)) &= r\pi_n(x) + (1-r)\pi_n(z_g), \\
    \pi_m(g(x)) &= r\pi_m(x) + (1-r)\pi_m(z_g).
\end{align*}
Define the boundaries of $O$ as $\alpha_1 = \pi_n(z_g)$, $\alpha_2 = \pi_n(z_f) = \alpha_1 + L$, $\beta_1 = c\alpha_1 + k$, and $\beta_2 = c\alpha_2 + k$. Let $L_n = \alpha_2 - \alpha_1 = L > 0$. Note that $\beta_2 - \beta_1 = cL_n > 0$. 

For $x \in O$, $\pi_n(g(O)) = (r\alpha_1 + (1-r)\alpha_1, r\alpha_2 + (1-r)\alpha_1) = (\alpha_1, \alpha_1 + rL_n)$. On the dual axis, since $\pi_m(z_g) = c\pi_n(z_g) + k = \beta_1$, $\pi_m(g(O)) = (\beta_1, r\beta_2 + (1-r)\beta_1) = (\beta_1, \beta_1 + r c L_n) \subset (\beta_1, \beta_2)$. Thus, $g(O) \subset O$.

For $f$, since $c > 0$, $\pi_m(f(O)) = (c\alpha_1 + k, c\alpha_2 + k) = (\beta_1, \beta_2)$. Since $\pi_m(x) \in (\beta_1, \beta_2)$, the bounds for $\pi_n(f(x))$ are computed using $\pi_m(z_f) = c\alpha_2 + k$:
\begin{align*}
    \inf \pi_n(f(O)) &= \beta_1 + \alpha_2 - (c\alpha_2 + k) \\
    &= c\alpha_1 + k + \alpha_2 - c\alpha_2 - k \\
    &= \alpha_2 - c(\alpha_2 - \alpha_1) = \alpha_2 - cL_n, \\[1ex]
    \sup \pi_n(f(O)) &= \beta_2 + \alpha_2 - (c\alpha_2 + k) \\
    &= c\alpha_2 + k + \alpha_2 - c\alpha_2 - k = \alpha_2.
\end{align*}
Thus, $\pi_n(f(O)) = (\alpha_2 - cL_n, \alpha_2)$. Because $c + r \le 1$, we have $rL_n \le (1-c)L_n$. This yields $\sup \pi_n(g(O)) = \alpha_1 + rL_n \le \alpha_2 - cL_n = \inf \pi_n(f(O))$. Therefore, $\pi_n(g(O))$ and $\pi_n(f(O))$ are disjoint, satisfying $f(O) \cap g(O) = \emptyset$.

\vspace{0.5em}
\noindent\textbf{Case 2: $S = -rI$}
\vspace{0.5em}

Here, $g(x) = -rx + (1 + r)z_g$. The projections are:
\begin{align*}
    \pi_n(g(x)) &= -r\pi_n(x) + (1+r)\pi_n(z_g), \\
    \pi_m(g(x)) &= -r\pi_m(x) + (1+r)\pi_m(z_g).
\end{align*}
Define the boundaries of $O$ as $\alpha_1 = \pi_n(z_g) - rL$, $\alpha_2 = \pi_n(z_f) = \pi_n(z_g) + L$, $\beta_1 = c\alpha_1 + k$, and $\beta_2 = c\alpha_2 + k$. Let $L_n = \alpha_2 - \alpha_1 = L(1+r) > 0$. Note that $\beta_2 - \beta_1 = cL_n > 0$. 

Isolating $\pi_n(z_g)$ from the boundary definitions yields $(1+r)\pi_n(z_g) = \alpha_1 + r\alpha_2$. The primary projection of $O$ under $g$ gives:
\begin{align*}
    \inf \pi_n(g(O)) &= -r\alpha_2 + (1+r)\pi_n(z_g) \\
    &= -r\alpha_2 + \alpha_1 + r\alpha_2 = \alpha_1, \\[1ex]
    \sup \pi_n(g(O)) &= -r\alpha_1 + (1+r)\pi_n(z_g) \\
    &= -r\alpha_1 + \alpha_1 + r\alpha_2 \\
    &= \alpha_1 + r(\alpha_2 - \alpha_1) = \alpha_1 + rL_n.
\end{align*}
Thus, $\pi_n(g(O)) = (\alpha_1, \alpha_1 + rL_n)$. 

For the dual axis, $\pi_m(z_g) = c\pi_n(z_g) + k$. Substituting the expression for $\pi_n(z_g)$ yields $(1+r)\pi_m(z_g) = c(\alpha_1 + r\alpha_2) + k(1+r) = (c\alpha_1 + k) + r(c\alpha_2 + k) = \beta_1 + r\beta_2$. The bounds evaluate as:
\begin{align*}
    \inf \pi_m(g(O)) &= -r\beta_2 + (1+r)\pi_m(z_g) \\
    &= -r\beta_2 + \beta_1 + r\beta_2 = \beta_1, \\[1ex]
    \sup \pi_m(g(O)) &= -r\beta_1 + (1+r)\pi_m(z_g) \\
    &= -r\beta_1 + \beta_1 + r\beta_2 \\
    &= \beta_1 + r(\beta_2 - \beta_1) < \beta_2.
\end{align*}
Thus, $\pi_m(g(O)) \subset (\beta_1, \beta_2)$.

For $f$, since $\pi_m(f(O)) = (\beta_1, \beta_2)$, the bounds for $\pi_n(f(x))$ follow identically from the logic in Case 1. Specifically:
\begin{align*}
    \inf \pi_n(f(O)) &= \beta_1 + \alpha_2 - \pi_m(z_f) = c\alpha_1 + k + \alpha_2 - (c\alpha_2 + k) = \alpha_2 - cL_n, \\[1ex]
    \sup \pi_n(f(O)) &= \beta_2 + \alpha_2 - \pi_m(z_f) = c\alpha_2 + k + \alpha_2 - (c\alpha_2 + k) = \alpha_2.
\end{align*}
Thus, $\pi_n(f(O)) = (\alpha_2 - cL_n, \alpha_2)$. Because $c + r \le 1$, we have $\alpha_1 + rL_n \le \alpha_2 - cL_n$. This guarantees that $\pi_n(g(O))$ and $\pi_n(f(O))$ are disjoint.

In both cases, there exists a non-degenerate open set $O$ such that $f(O) \cup g(O) \subset O$ and $f(O) \cap g(O) = \emptyset$, satisfying the open set condition.
\end{proof}

The following examples demonstrate Theorem \ref{thm:pos_cI} for the cases $S = rI$ and $S = -rI$.

\begin{example}[$S = rI$] \label{ex:pos_cI_homothety}
Let $\mathcal{F} = \{f, g\}$ be the IFS on $\mathbb{R}^2$ given by:
\begin{align*}
    f(x) &= \begin{pmatrix} 0.3 & 0.4 \\ 0.4 & -0.3 \end{pmatrix} x, \\
    g(x) &= \begin{pmatrix} 0.75 & 0 \\ 0 & 0.75 \end{pmatrix} x + \begin{pmatrix} 1.0 \\ 2.0 \end{pmatrix}.
\end{align*}
Let $A$ and $S$ be the linear parts of $f$ and $g$. We have:
\begin{equation*}
    A^2 = \begin{pmatrix} 0.25 & 0 \\ 0 & 0.25 \end{pmatrix} = 0.25 I.
\end{equation*}
Thus, $A^2 = cI$ with $c = 0.25$. Furthermore, $\det(A) = -0.25 < 0$, strictly satisfying the negative determinant condition. The linear part of $g$ acts as a positive homothety $S = rI$ with $r = 0.75$. The parameter constraint $c + r \le 1$ holds perfectly as $0.25 + 0.75 = 1$.

The fixed points evaluate to $z_f = (0, 0)^\top$ and $z_g = (4, 8)^\top$. Since $z_f \neq z_g$, the hypotheses of Theorem \ref{thm:pos_cI} are satisfied.

To explicitly construct the bounding open set $O$, we determine the invariant dual axes. The connecting vector is $v = z_f - z_g = (-4, -8)^\top$. We choose a normal vector $n = (-14, -17)^\top$ which is strictly orthogonal to $v - c^{-1}Av$, ensuring that $\langle n, v \rangle = 192 > 0$. The corresponding dual axis evaluates exactly to $m = A^\top n = (-11, -0.5)^\top$. 

The projections of the fixed points onto the primary axis yield $\pi_n(z_f) = 0$ and $\pi_n(z_g) = -192$. The offset is $k = \pi_m(z_f) - c\pi_n(z_f) = 0$. By applying the formulas derived in Case 1 of Theorem \ref{thm:pos_cI} (where $S = rI$), the exact affine boundaries evaluate to:
\begin{align*}
    \alpha_1 &= \pi_n(z_g) = -192, & \alpha_2 &= \pi_n(z_f) = 0, \\
    \beta_1 &= c\alpha_1 + k = 0.25(-192) = -48, & \beta_2 &= c\alpha_2 + k = 0.
\end{align*}
Thus, the open set is exactly defined in the abstract dual space as:
\begin{equation*}
    O_{\pi} = \{ x \in \mathbb{R}^2 \mid -192 < \pi_n(x) < 0 \quad \text{and} \quad -48 < \pi_m(x) < 0 \}.
\end{equation*}

To map this abstract parallelogram back to standard Euclidean space, we construct the projection matrix $P$ whose rows are $n^\top$ and $m^\top$, and compute its inverse:
\begin{equation*}
    P = \begin{pmatrix} -14 & -17 \\ -11 & -0.5 \end{pmatrix} 
    \implies 
    P^{-1} = \frac{1}{-180} \begin{pmatrix} -0.5 & 17 \\ 11 & -14 \end{pmatrix}.
\end{equation*}
Multiplying the four boundary corners of $O_{\pi}$ by $P^{-1}$ yields their exact Cartesian coordinates in $\mathbb{R}^2$:
\begin{align*}
    v_1 &= P^{-1} \begin{pmatrix} -192 \\ -48 \end{pmatrix} = \begin{pmatrix} 4 \\ 8 \end{pmatrix}, & 
    v_2 &= P^{-1} \begin{pmatrix} 0 \\ -48 \end{pmatrix} = \begin{pmatrix} 68/15 \\ -56/15 \end{pmatrix}, \\
    v_3 &= P^{-1} \begin{pmatrix} 0 \\ 0 \end{pmatrix} = \begin{pmatrix} 0 \\ 0 \end{pmatrix}, & 
    v_4 &= P^{-1} \begin{pmatrix} -192 \\ 0 \end{pmatrix} = \begin{pmatrix} -8/15 \\ 176/15 \end{pmatrix}.
\end{align*}
Notice that $v_1$ and $v_3$ map directly to the fixed points $z_g$ and $z_f$, a direct geometric consequence of the positive homothety in Case 1. The polygon defined by the vertices $\{v_1, v_2, v_3, v_4\}$ constitutes the bounding set $O$. The resulting attractor and this rigorously derived Euclidean bounding box are plotted alongside each other in Figure \ref{fig:example_3_3}, visually confirming that $f(O) \cap g(O) = \emptyset$ and $f(O) \cup g(O) \subset O$.

\begin{figure}[htpb]
    \centering
    \includegraphics[width=\textwidth]{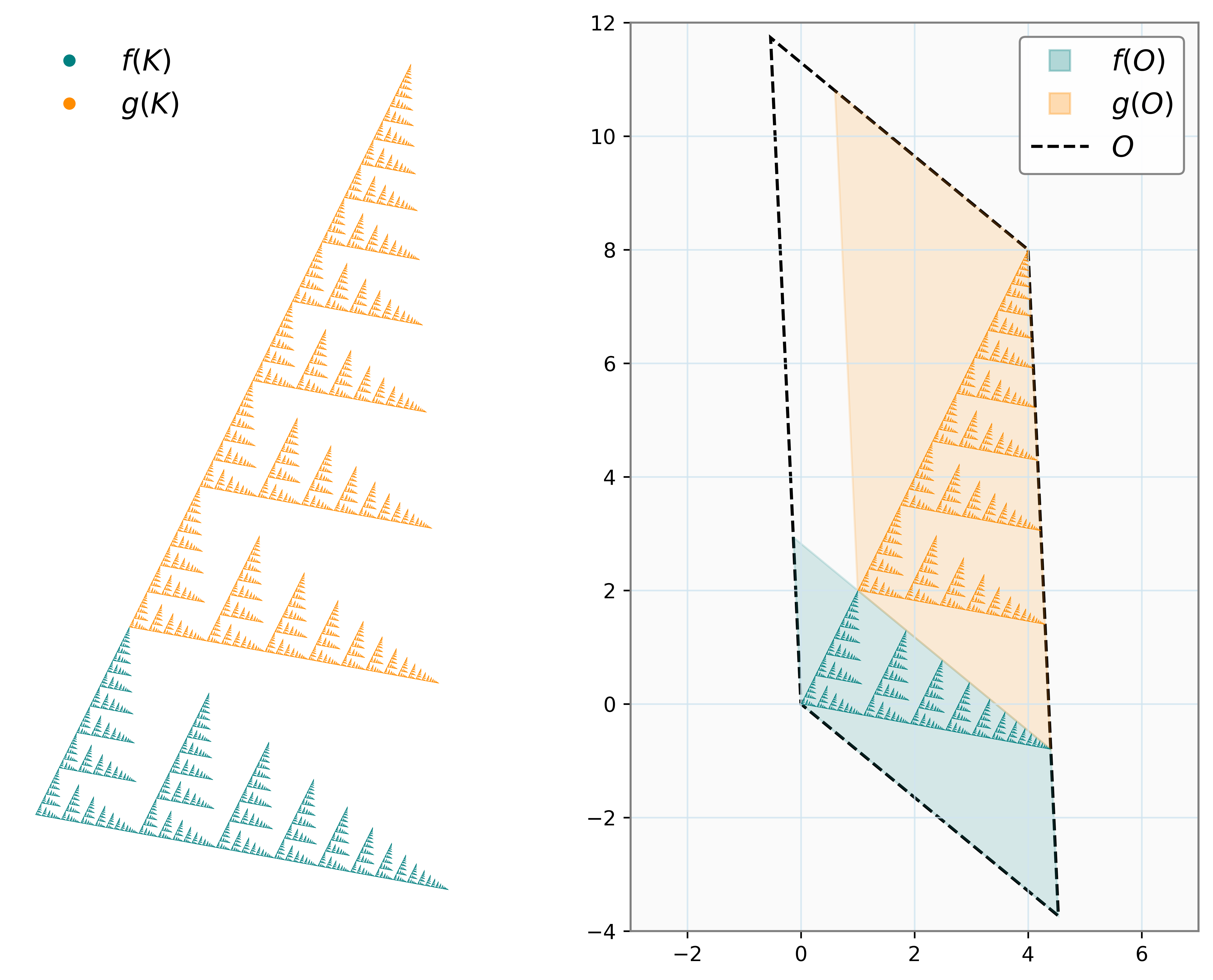} 
   \caption{The attractor (left) and bounding set $O$ (right) for Example \ref{ex:pos_cI_homothety}. The lack of interior overlap between $f(O)$ and $g(O)$ provides direct geometric proof of the open set condition.}
    \label{fig:example_3_3}
\end{figure}
\end{example}

\begin{example}[$S = -rI$] \label{ex:pos_cI_chosen}
Let $\mathcal{F} = \{f, g\}$ be the IFS given by:
\begin{align*}
    f(x) &= \begin{pmatrix} 0 & 1 \\ 0.25 & 0 \end{pmatrix} x, \\
    g(x) &= \begin{pmatrix} -0.75 & 0 \\ 0 & -0.75 \end{pmatrix} x + \begin{pmatrix} 1.75 \\ 1.75 \end{pmatrix}.
\end{align*}
Squaring the linear part of $f$ yields $A^2 = 0.25 I$, so $c = 0.25$. The determinant is $\det(A) = -0.25 < 0$. The linear part of $g$ acts as a scaled point reflection with $r = 0.75$. The parameter constraint $c + r = 1$ is satisfied.

The fixed points evaluate to $z_f = (0, 0)^\top$ and $z_g = (1, 1)^\top$. To construct the canonical open set $O$, Theorem \ref{thm:pos_cI} dictates that the normal vector $n$ must be orthogonal to $v - c^{-1}Av$. Calculating this vector yields $(3, 0)^\top$. Thus, we choose the strictly orthogonal normal vector $n = (0, -1)^\top$. The corresponding dual axis is $m = A^\top n = (-0.25, 0)^\top$.

Following Case 2 of Theorem \ref{thm:pos_cI}, we determine the abstract boundaries in the dual space:
\begin{equation*}
    O_{\pi} = \{ x \in \mathbb{R}^2 \mid -1.75 < \pi_n(x) < 0 \quad \text{and} \quad -0.4375 < \pi_m(x) < 0 \}.
\end{equation*}

Mapping these boundaries back to $\mathbb{R}^2$ via the inverse projection matrix $P^{-1}$ yields the vertices for the bounding set $O$:
\begin{equation*}
    v_1 = \begin{pmatrix} 1.75 \\ 1.75 \end{pmatrix}, \quad
    v_2 = \begin{pmatrix} 1.75 \\ 0 \end{pmatrix}, \quad
    v_3 = \begin{pmatrix} 0 \\ 0 \end{pmatrix}, \quad
    v_4 = \begin{pmatrix} 0 \\ 1.75 \end{pmatrix}.
\end{equation*}

Thus, the bounding set $O$ corresponds to a Euclidean square. The resulting attractor and its bounding box are plotted in Figure \ref{fig:example_3_4}. The images $f(O)$ and $g(O)$ are strictly disjoint within $O$, satisfying the open set condition.

\begin{figure}[htpb]
    \centering
    \includegraphics[width=\textwidth]{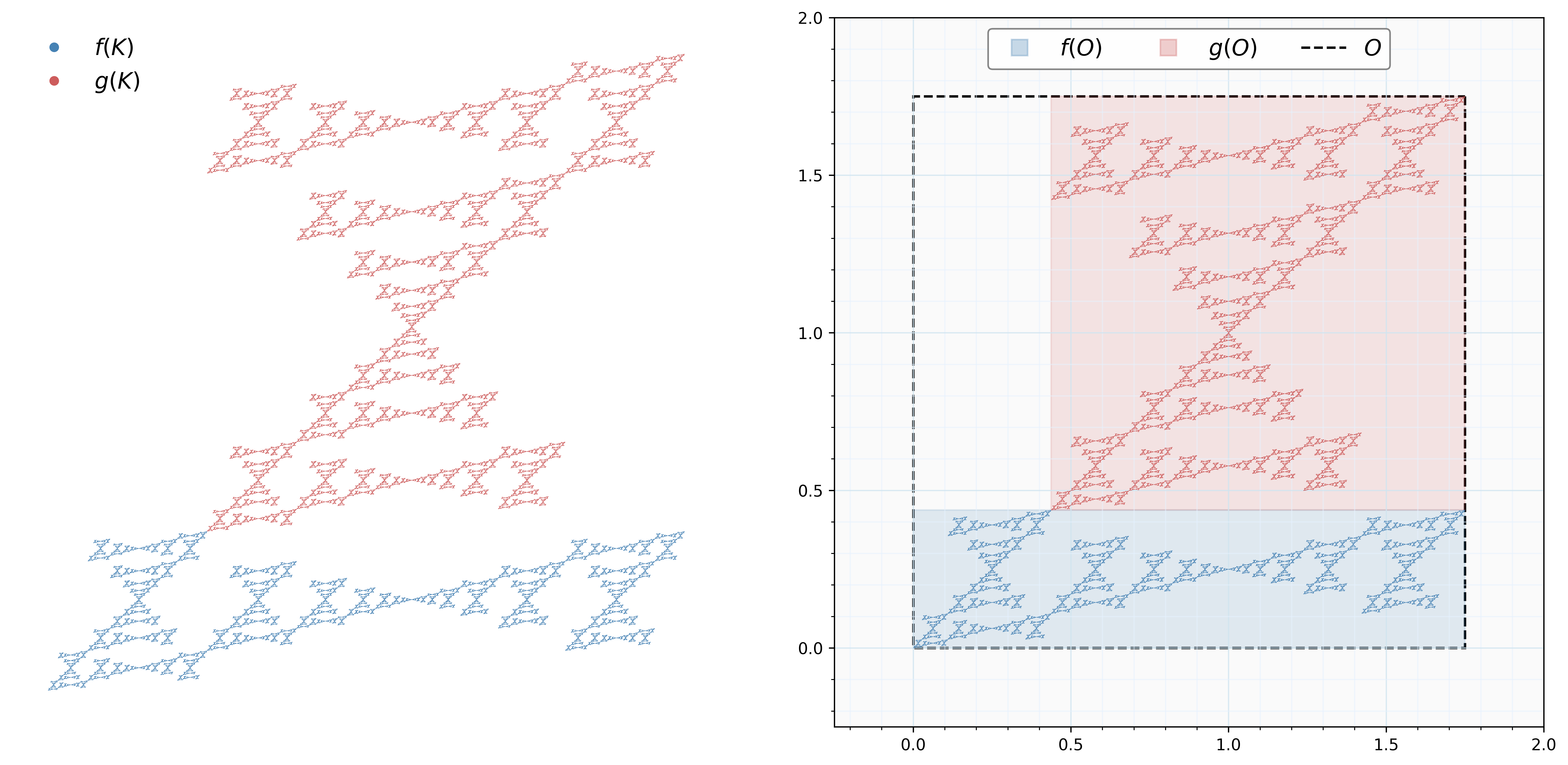} 
    \caption{The attractor (left) and the bounding square $O$ (right) for Example \ref{ex:pos_cI_chosen}. The images $f(O)$ and $g(O)$ demonstrate strict interior separation without overlap, geometrically validating the open set condition.}
    \label{fig:example_3_4}
\end{figure}
\end{example}

\begin{theorem}\label{thm:connectivity}
Let $\mathcal{F} = \{f, g\}$ be an iterated function system satisfying the hypotheses of either Theorem \ref{thm:neg_cI} or Theorem \ref{thm:pos_cI}, with the exception that the constraint $c + r \le 1$ is relaxed to allow any sum $c + r > 0$. Assuming $z_f \neq z_g$, the topological structure of the attractor $K$ is characterized as follows:
\begin{enumerate}
    \item If $c + r \ge 1$, then $K$ is connected.
    \item If $c + r < 1$, then $K$ is totally disconnected.
\end{enumerate}
\end{theorem}

\begin{proof}
Let $z_f, z_g \in \mathbb{R}^2$ denote the distinct fixed points of $f$ and $g$. The infinite line passing through $z_f$ and $z_g$ is invariant under the subsystem $\{f^2, g\}$. We parameterize this line via the bijection $h: \mathbb{R} \to \mathbb{R}^2$ defined by $h(t) = z_f + t(z_g - z_f)$. Under this parametrization, $h(0) = z_f$ and $h(1) = z_g$. 

The subsystem induces 1D affine maps $\tilde{f}, \tilde{g}: \mathbb{R} \to \mathbb{R}$ such that $f^2(h(t)) = h(\tilde{f}(t))$ and $g(h(t)) = h(\tilde{g}(t))$. The respective contraction ratios are $c$ and $r$.

\vspace{0.5em}
\noindent\textbf{Case 1: $c + r \ge 1$}
\vspace{0.5em}

By standard topological properties, $K$ is connected if $f(K) \cap g(K) \neq \emptyset$. It suffices to show the physical segment $L = h([0,1])$ connects $z_f$ to $z_g$ within $K$. We construct a closed interval $H$ containing $[0, 1]$ such that $H \subset \tilde{f}(H) \cup \tilde{g}(H)$, rendering $h(H)$ invariant under the subsystem.

For $A^2 = cI$, $\tilde{f}(t) = ct$.
\begin{itemize}
    \item \textit{Subcase 1a ($S = rI$):} $\tilde{g}(t) = rt + 1 - r$. Let $H = [0, 1]$. Then $\tilde{f}(H) = [0, c]$ and $\tilde{g}(H) = [1-r, 1]$. Since $c + r \ge 1$, we have $c \ge 1-r$, yielding $\tilde{f}(H) \cup \tilde{g}(H) = [0, 1] = H$.
    \item \textit{Subcase 1b ($S = -rI$):} $\tilde{g}(t) = -rt + 1 + r$. Let $H = [0, 1+r]$. Then $\tilde{f}(H) = [0, c(1+r)]$ and $\tilde{g}(H) = [1-r^2, 1+r]$. Since $c + r \ge 1$, $c(1+r) \ge (1-r)(1+r) = 1-r^2$. Thus, the intervals overlap and cover $H$.
\end{itemize}

For $A^2 = -cI$, $\tilde{f}(t) = -ct$.
\begin{itemize}
    \item \textit{Subcase 2a ($S = rI$):} $\tilde{g}(t) = rt + 1 - r$. Let $H = [-c, 1]$. Then $\tilde{f}(H) = [-c, c^2]$ and $\tilde{g}(H) = [1 - r(1+c), 1]$. Since $c + r \ge 1$, we have $r(1+c) \ge (1-c)(1+c) = 1-c^2$, which implies $1 - r(1+c) \le c^2$. Thus, the intervals overlap and cover $H$.
    \item \textit{Subcase 2b ($S = -rI$):} $\tilde{g}(t) = -rt + 1 + r$. Let $H = \left[ \frac{-c(1+r)}{1-rc}, \frac{1+r}{1-rc} \right]$. Applying the mappings yields $\tilde{f}(H) = \left[ \frac{-c(1+r)}{1-rc}, \frac{c^2(1+r)}{1-rc} \right]$ and $\tilde{g}(H) = \left[ \frac{(1+r)(1-r-rc)}{1-rc}, \frac{1+r}{1-rc} \right]$. The union completely covers $H$ if $c^2(1+r) \ge (1+r)(1-r-rc)$. Dividing by $(1+r)$ and rearranging simplifies this strictly to $(c+1)(c+r-1) \ge 0$. Since $c > 0$, this holds exactly when $c+r \ge 1$.
\end{itemize}

In all configurations, $c+r \ge 1$ guarantees $\tilde{f}(H) \cup \tilde{g}(H) = H$. Because $[0, 1] \subset H$, the segment $L \subset K$. This implies $f(K) \cap g(K) \neq \emptyset$, hence $K$ is connected.

\vspace{0.5em}
\noindent\textbf{Case 2: $c + r < 1$}
\vspace{0.5em}

Let $O \subset \mathbb{R}^2$ be the bounded open set from Theorem \ref{thm:neg_cI} or \ref{thm:pos_cI}. Since $f(O) \cup g(O) \subset O$, we have $K \subset \overline{O}$. 

The projection $\pi_n(\overline{O})$ is a closed interval of length $L_n > 0$. The projected images $\pi_n(f(\overline{O}))$ and $\pi_n(g(\overline{O}))$ have lengths exactly $cL_n$ and $rL_n$, respectively. Since $c+r < 1$, the sum of their lengths satisfies:
\begin{equation*}
    cL_n + rL_n < L_n.
\end{equation*}

By the constructions in Theorems \ref{thm:neg_cI} and \ref{thm:pos_cI}, the mappings $f$ and $g$ anchor these projections to opposite extremal endpoints of $\pi_n(\overline{O})$. Because their combined length is strictly less than $L_n$, they are disjoint. This implies $f(\overline{O}) \cap g(\overline{O}) = \emptyset$.

Since $K \subset \overline{O}$, it strictly follows that $f(K) \cap g(K) = \emptyset$. As $K = f(K) \cup g(K)$ is the union of strictly disjoint compact sets, $\mathcal{F}$ satisfies the strong separation condition. Therefore, $K$ is totally disconnected.
\end{proof}

\subsection{Failure of the Open Set Condition and Connectedness for $f$-Aligned Reflections}

Theorems \ref{thm:neg_cI} and \ref{thm:pos_cI} establish that the open set condition (OSC) holds when the linear part of $g$ acts as a uniform scaling or a point reflection ($S = \pm rI$). We now show that if $S$ is an axial reflection, the OSC may fail even when all other hypotheses ($A^2 = \pm cI$, $c + r \le 1$, and $z_f \neq z_g$) are satisfied. We provide two counterexamples corresponding to the structural cases $A^2 = -cI$ and $A^2 = cI$.

\begin{example}[$A^2 = -cI$ and $S$ is a reflection] \label{ex:osc_fail_neg_cI}
Consider the IFS generated by the following affine transformations:
\begin{align*}
    f(x) &= \begin{pmatrix} 0 & -1/5 \\ 1 & 0 \end{pmatrix} x, \\
    g(x) &= \begin{pmatrix} 4/5 & 0 \\ 0 & -4/5 \end{pmatrix} x + \begin{pmatrix} 1.0 \\ 0.87 \end{pmatrix}.
\end{align*}
Let $A$ and $S$ denote the linear parts of $f$ and $g$, respectively. Here, $A^2 = -cI$ with $c = 1/5$. The matrix $S$ is a reflection across the $x$-axis scaled by $r = 4/5$. The contraction ratios perfectly satisfy the parameter constraint $c + r \le 1$ (as $1/5 + 4/5 = 1$). However, as illustrated in Figure \ref{fig:overlap_reflections}(a), the images $f(K)$ and $g(K)$ overlap, violating the OSC.
\end{example}

\begin{example}[$A^2 = cI$ and $S$ is a reflection] \label{ex:osc_fail_pos_cI}
Consider the IFS generated by:
\begin{align*}
    f(x) &= \begin{pmatrix} 0 & 1/5 \\ 1 & 0 \end{pmatrix} x, \\
    g(x) &= \begin{pmatrix} 4/5 & 0 \\ 0 & -4/5 \end{pmatrix} x + \begin{pmatrix} 1.0 \\ 0.92 \end{pmatrix}.
\end{align*}
Here, $c = 1/5$ and $r = 4/5$, satisfying the constraint $c + r \le 1$, and the fixed points are distinct. Although the system meets the algebraic prerequisites of Theorem \ref{thm:pos_cI}, Figure \ref{fig:overlap_reflections}(b) demonstrates a clear overlap between the images $f(K)$ and $g(K)$, violating the OSC.
\end{example}

\begin{figure}[htpb]
    \centering
    \begin{subfigure}[b]{0.48\textwidth}
        \centering
        \includegraphics[width=\textwidth]{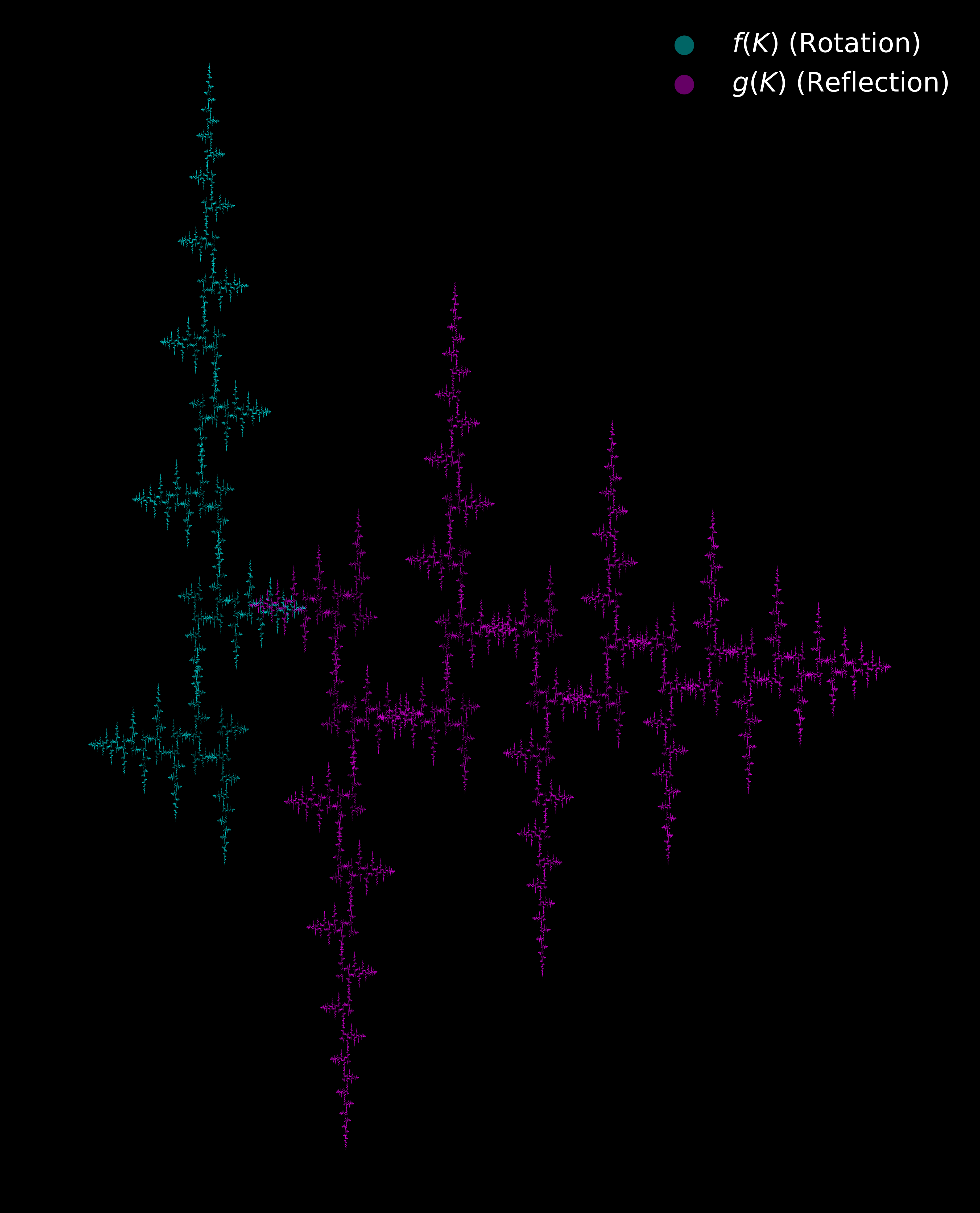}
        \caption{$A^2 = -cI$. Branches $f(K)$ (cyan) and $g(K)$ (magenta) overlap.}
        \label{fig:overlap_rot_ref}
    \end{subfigure}
    \hfill
    \begin{subfigure}[b]{0.48\textwidth}
        \centering
        \includegraphics[width=\textwidth]{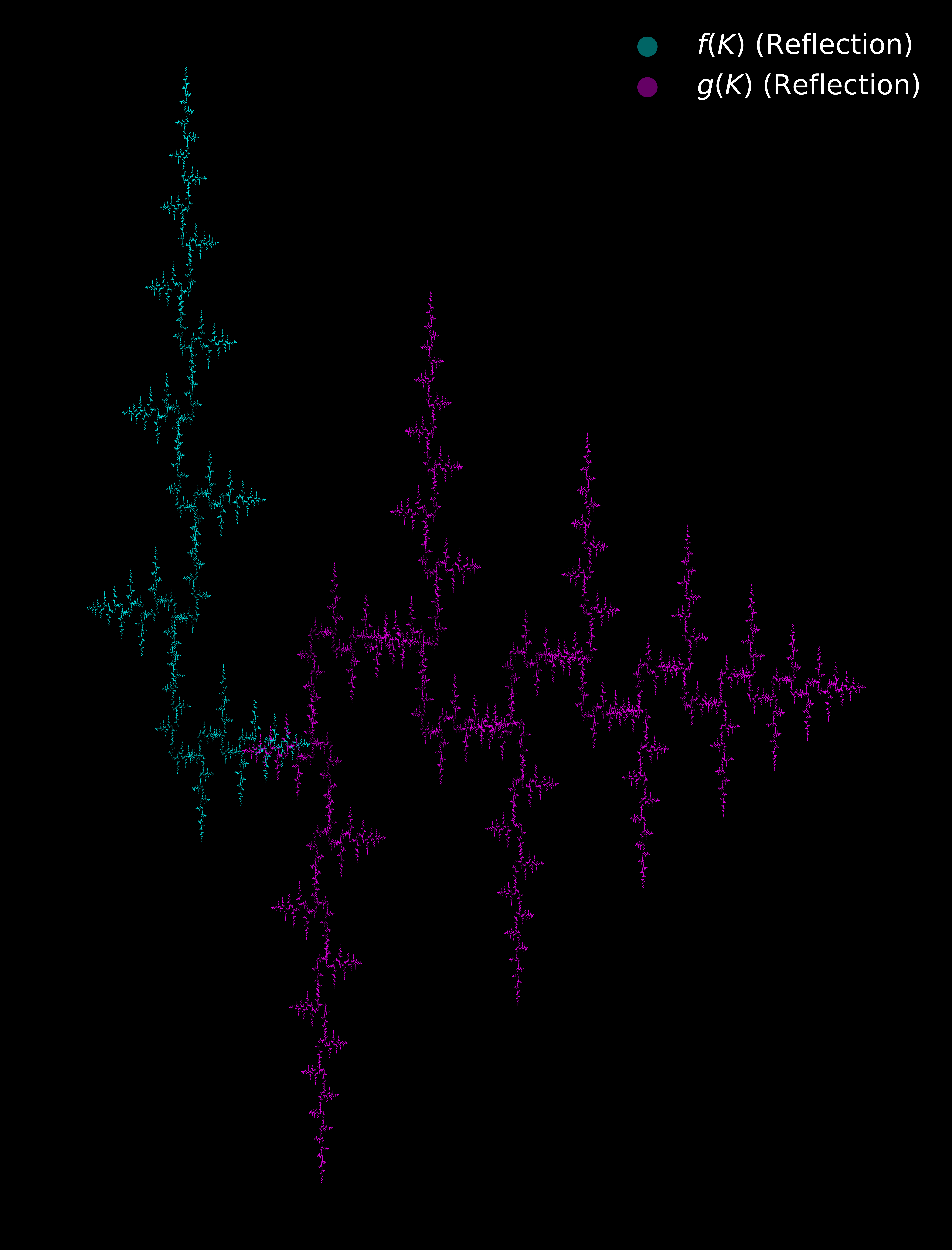}
        \caption{$A^2 = cI$. Orientation-reversing matrices induce an overlap.}
        \label{fig:overlap_ref_ref}
    \end{subfigure}
    \caption{Failure of the Open Set Condition when $S$ is an axial reflection. In both examples, the parameters satisfy $c=1/5$, $r=4/5$, and the constraint $c+r \le 1$, yet the sets $f(K)$ and $g(K)$ overlap.}
    \label{fig:overlap_reflections}
\end{figure}

These examples demonstrate that the open set condition does not generalize to axial reflections, even when all parameter constraints are perfectly met.

\begin{example}[Topological Disconnection for $c + r > 1$ under Reflection] \label{ex:connect_fail}
Consider the IFS consisting of:
\begin{align*}
    f(x) &= \begin{pmatrix} 0 & 2/3 \\ 1 & 0 \end{pmatrix} x, \\
    g(x) &= \begin{pmatrix} 1/2 & 0 \\ 0 & -1/2 \end{pmatrix} x + \begin{pmatrix} 1 \\ 2 \end{pmatrix}.
\end{align*}
The linear part of $f$ is a matrix $A$ satisfying $A^2 = cI$ with $c = 2/3$. The linear part of $g$, denoted $S$, is an axial reflection scaled by $r = 1/2$. Because $A^\top A = \text{diag}(1, 4/9)$ commutes with the diagonal matrix $S$, $g$ is $f$-aligned per Definition \ref{def:f_aligned}.

The sum of the parameters is $c + r = 7/6 > 1$. If $S$ were a uniform scaling or point reflection, Theorem \ref{thm:connectivity} would guarantee a connected attractor. However, the attractor here is totally disconnected (see Figure \ref{fig:connectivity_failure}), confirming that the condition $c + r \ge 1$ fails to guarantee connectedness for axial reflections.

\begin{figure}[htpb]
    \centering
    \includegraphics[width=0.6\textwidth]{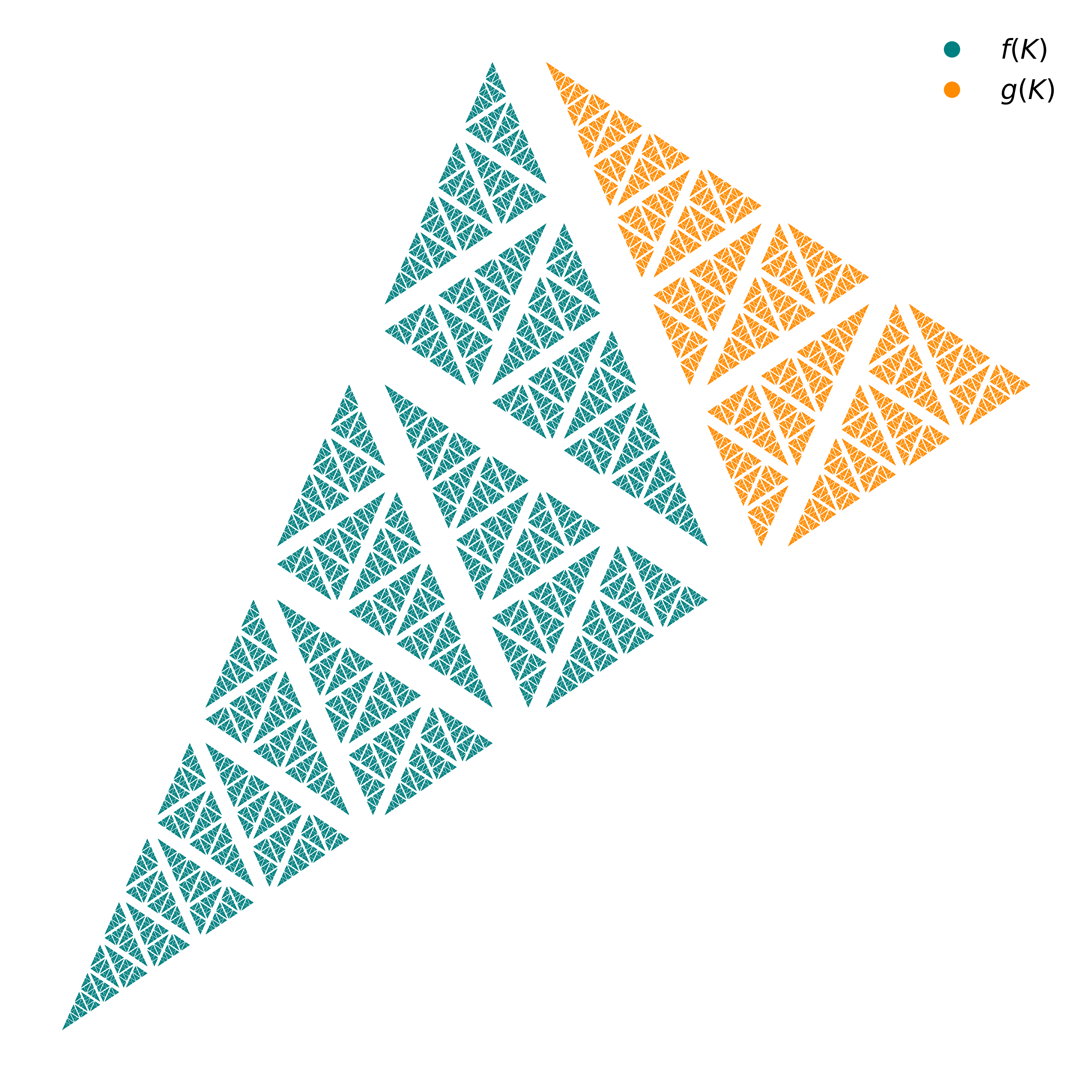} 
    \caption{The attractor of the IFS defined in Example \ref{ex:connect_fail}, illustrating the failure of connectedness under axial reflection. Although the parameters $c = 2/3$ and $r = 1/2$ satisfy $c + r > 1$, the orientation-reversing nature of $S$ results in a totally disconnected set.}
    \label{fig:connectivity_failure}
\end{figure}
\end{example}

\begin{remark}
The preceding theorems reveal a strict topological dichotomy for this family of systems. For an IFS $\mathcal{F} = \{f, g\}$ where $f$ is a strict $G^2$-similarity contraction and $g$ is an $f$-aligned similarity (excluding axial reflections), the open set condition requires $c + r \le 1$, whereas connectedness dictates $c + r \ge 1$. Therefore, the exact parametric equality $c + r = 1$ represents the unique condition under which the system can simultaneously produce a connected attractor and satisfy the open set condition.
\end{remark}

\section*{Acknowledgements}
This research was financially supported by the University Grants Commission (UGC), Government of India, through the award of the Junior Research Fellowship (JRF) (Ref. No. [1090(CSIR-UGC NET JUNE 2019)]) awarded to the first author. The first and third authors acknowledge the institutional support provided by Rajagiri School of Engineering and Technology, Kerala, India and the second author acknowledges the institutional support provided by Muthoot Institute of Technology and Science, Kerala, India.

\bibliographystyle{amsplain}
\bibliography{references}

\end{document}